\documentclass[a4paper,11pt,oneside]{article}
\usepackage[margin=3cm]{geometry}

\usepackage{latexsym,amsmath,amsthm,amssymb,calc}
\usepackage[utf8]{inputenc}
\usepackage{dsfont}
\usepackage{color}
\usepackage{graphicx}
\usepackage{mdwlist}
\usepackage{enumerate}
\usepackage{setspace}
\usepackage{multicol}
\usepackage{multirow}
\usepackage{cite}

\usepackage{mathtools}

\usepackage[section]{algorithm}
\usepackage{algorithmicx}
\usepackage{algpseudocode}
\algrenewcommand\algorithmicrequire{\makebox[32pt][l]{\textrm{input}}}
\algrenewcommand\algorithmicensure{\makebox[32pt][l]{\textrm{output}}}
\algrenewcommand\algorithmicfunction{\textrm{function}}
\algrenewcommand\algorithmicwhile{\textrm{while}}
\algrenewcommand\algorithmicdo{}
\algrenewcommand\algorithmicend{\textrm{end}}
\algrenewcommand\algorithmicforall{\textrm{for all}}
\algrenewcommand\algorithmicfor{\textrm{for}}
\algrenewcommand\algorithmicrepeat{\textrm{repeat}}
\algrenewcommand\algorithmicuntil{\textrm{until}}

\DeclareFontFamily{U}{mathx}{\hyphenchar\font45}
\DeclareFontShape{U}{mathx}{m}{n}{
      <5> <6> <7> <8> <9> <10>
      <10.95> <12> <14.4> <17.28> <20.74> <24.88>
      mathx10
      }{}
\DeclareSymbolFont{mathx}{U}{mathx}{m}{n}
\DeclareMathSymbol{\bigtimes}{1}{mathx}{"91}

\theoremstyle{plain}
\newtheorem{theorem}{Theorem}[section]
\newtheorem{lemma}[theorem]{Lemma}
\newtheorem{proposition}[theorem]{Proposition}

\newtheorem{remark}[theorem]{Remark}
\newtheorem{assumptions}[theorem]{Assumptions}
\theoremstyle{definition}

\newtheorem{example}[theorem]{Example}

\numberwithin{equation}{section}

\newcommand{\N}{\mathds{N}}
\newcommand{\R}{\mathds{R}}

\DeclareMathOperator{\linspan}{span}
\DeclareMathOperator{\divergence}{div}
\DeclareMathOperator{\supp}{supp}

\DeclareMathOperator{\ops}{ops}

\newcommand{\id}{{\rm I}}

\newcommand{\spl}[1]{\ell^{#1}}

\newcommand{\spH}[1]{H^{#1}}

\newcommand{\spL}[1]{L^{#1}}

\newcommand{\be}{\begin{equation}}
\newcommand{\ee}{\end{equation}}

\newcommand{\Ocal}{{\mathcal{O}}}

\newcommand{\Hcal}{{\mathcal{H}}}

\newcommand{\Ical}{{\mathcal{I}}}

\newcommand{\Acal}{{\mathcal{A}}}
\newcommand{\Gcal}{{\mathcal{G}}}

\DeclareMathOperator{\rank}{rank}

\DeclareMathOperator{\apply}{\textsc{apply}}
\DeclareMathOperator{\coarsen}{\textsc{coarsen}}
\DeclareMathOperator{\rhs}{\textsc{rhs}}
\DeclareMathOperator{\recompress}{\textsc{recompress}}
\DeclareMathOperator{\solve}{\textsc{solve}}

\providecommand{\abs}[1]{\lvert#1\rvert}

\providecommand{\norm}[1]{\lVert#1\rVert}
\providecommand{\bignorm}[1]{\bigl\lVert#1\bigr\rVert}
\providecommand{\Bignorm}[1]{\Bigl\lVert#1\Bigr\rVert}
\providecommand{\biggnorm}[1]{\biggl\lVert#1\biggr\rVert}

\providecommand{\ceil}[1]{\lceil#1\rceil}

\def\bu{{\bf u}}
\def\bU{{\bf U}}
\def\bV{{\bf V}}

\newcommand{\bv}{{\bf v}}
\newcommand{\bw}{{\bf w}}
\newcommand{\ba}{{\bf a}}
\def\e2{\spl{2}(\nabla^d)}
\def\cA{{\cal A}}
\def\ga{\gamma}

\newcommand{\As}{{\Acal^s}}

\newcommand{\AH}[1]{{\Acal_\Hcal({#1})}}

\newcommand{\kk}[1]{{\mathsf{#1}}}

\newcommand{\hatPsvd}[1]{\operatorname{\hat P}_{#1}}

\newcommand{\hatCctr}[1]{\operatorname{\hat C}_{#1}}
\newcommand{\Restr}[1]{\operatorname{R}_{#1}}

\newcommand\eref[1]{(\ref{#1})}
\newcommand{\beqn}{\begin{equation}}
\newcommand{\eeqn}{\end{equation}}

\newcommand{\bA}{\mathbf{A}}
\newcommand{\bM}{\mathbf{M}}

\newcommand{\bbf}{\mathbf{f}}

\newcommand{\cD}{{\cal D}}

\newcommand{\bB}{\mathbf{B}}

\usepackage{tikz}
\usetikzlibrary{arrows}
\usepackage{mathdots}
\usepackage{mathtools}

\newcommand{\pdom}{Y}

\newcommand{\sidx}{\mathcal{S}}
\newcommand{\pidx}{\mathcal{F}}

\newcommand{\Sb}{\Sigma^{\bar s}}

\newcommand{\rs}{{\rm x}}
\newcommand{\rp}{{\rm y}}

\newcommand{\pf}{\theta}

\newcommand{\xidx}{\mathcal{G}}

\newcommand{\bb}{\mbox{\boldmath${\beta}$}}
\newcommand{\ab}{\mbox{\boldmath${\alpha}$}}

\def\Chi{\raise .3ex
\hbox{\large $\chi$}}

\title{Parametric PDEs: Sparse or low-rank approximations?\thanks{This work has been supported by ERC AdG 338977 BREAD, DFG SFB-Transregio 40, DFG Research Group 1779, the Excellence Initiative of the German Federal and State Governments
 (RWTH Aachen  Distinguished Professorship), by DARPA-BAA-15-13, and by the Hausdorff Center of Mathematics, University of Bonn.}}

\author{Markus Bachmayr\thanks{Hausdorff Center for Mathematics \& Institute for Numerical Simulation, Wegelerstr.\ 6, 53115 Bonn, Germany (bachmayr@ins.uni-bonn.de)}, Albert Cohen\thanks{Sorbonne Universit\'es, UPMC Univ Paris 06, CNRS, UMR 7598, Laboratoire Jacques-Louis Lions, 4 place Jussieu, 75005 Paris, France (cohen@ljll.math.upmc.fr)} and Wolfgang Dahmen\thanks{Institut f\"ur Geometrie und Praktische Mathematik, RWTH Aachen, Templergraben 55, 52056 Aachen, Germany (dahmen@igpm.rwth-aachen.de)}}

\date{\today}

\begin{document}

\maketitle

\begin{abstract}
We consider adaptive approximations of the parameter-to-solution map for elliptic operator equations depending on a large or infinite number of parameters,
 comparing approximation strategies of different degrees of nonlinearity: sparse polynomial expansions, general low-rank approximations separating spatial and parametric variables, and hierarchical tensor decompositions separating all variables.
We describe corresponding adaptive algorithms based on a common generic template and show their near-optimality with respect to natural approximability assumptions for each type of approximation.
A central ingredient in the resulting bounds for the total computational complexity are new operator compression results for the case of infinitely many parameters.
We conclude with a comparison of the complexity estimates based on the actual approximability properties of classes of parametric model problems, which shows that the computational costs of optimized low-rank expansions can be significantly lower or higher than those of sparse polynomial expansions, depending on the particular type of parametric problem.


\end{abstract}

\section{Introduction}
\newcommand{\cI}{\mathcal{I}}
\newcommand{\cH}{\mathcal{H}}
\newcommand{\cF}{\mathcal{F}}
Complex design, optimization, or uncertainty quantification tasks 
based on parameter dependent families of PDEs arise in virtually all branches of science and engineering.
Typical scenarios are models whose physical properties -- such as diffusivity, transport velocity or domain geometry -- are 
described by a {\em finite} number of real parameter values. In certain instances, one may even encounter
{\em infinitely} many parameters of decreasing influence. This occurs for instance in the case of 
a random stochastic diffusion field represented by an infinite expansion in a given basis. The development
and analysis of numerical strategies for capturing the dependence of the PDE on the parameters 
has been the subject of intensive research efforts in recent years.

\subsection{Problem formulation}

The problems that are addressed in this paper have the following general form.
Let $V$ be a separable Hilbert space. We consider a parametric operator $A(y)\colon V\to V'$ of the form
\be\label{Arep} 
 A(y) := A_0 + \sum_{j\in\cI} y_j A_j\,,\quad y\in\pdom := [-1,1]^\cI \,,  
 \ee
where $\cI=\{1,\dots,d\}$ 
or $\cI= \N$ in the finite or infinite dimensional case, respectively.
In the infinite dimensional case, we require that the above series converges 
in ${\cal L}(V,V')$ for any $y\in \pdom$. We assume uniform boundedness and ellipticity of $A(y)$
over the parameter domain, that is
\be
\label{ellip}
\langle A(y)v,w\rangle \leq R \|v\|_V\|w\|_V \quad {\rm and} \quad
\langle A(y)v,v\rangle \geq r \|v\|_V^2, \quad v,w\in V, \; y\in \pdom,
\ee
for some $0<r\leq R<\infty$, which implies in particular that
$A(y)$ is boundedly 
invertible uniformly in $y\in \pdom$, with
\be
\|A(y)\|_{{\cal L}(V',V)}\leq r^{-1}, \quad y\in \pdom.
\ee
We also consider parametric data $f\colon \pdom \to V'$, and for each $y\in Y$, we define
$u(y)\in V$ the solution to the equation
\begin{equation}\label{paramgeneral}
A(y)\,u(y) = f(y).   
\end{equation}
A guiding example is provided by {\em affinely parametrized} diffusion problems of the form
\begin{equation} \label{paramdiffusion-0}
   A(y) u := -\divergence\bigl( a(y) \nabla u \bigr) = f , \quad  a(y) := \bar a + \sum_{j \in \cI} y_j \pf_j ,
\end{equation}
 with homogeneous Dirichlet boundary conditions, posed in the weak sense on a spatial domain $D \subset \R^m$. 
 In this particular case of frequent interest, the data $f\in V'$ is independent of $y$.
 The validity of \eref{ellip} is then usually ensured by the \emph{uniform ellipticity assumption}
\begin{equation}\label{uea}  
  \sum_{j\geq 1} \abs{\pf_j(x)} \leq  \bar a(x)  - \underline{\alpha}, \quad x\in D,
\end{equation}
for some $\underline{\alpha}>0$. 
We then have $V=\spH{1}_0(D)$, and the corresponding operators $A_j \colon V \to V'$ for $j \in \{0\}\cup \cI$ are defined by
\begin{equation*}
  \langle A_0 u, v \rangle :=  \int_D \bar{a} \nabla u\cdot\nabla v \,dx \,,\qquad   \langle A_j u, v \rangle :=  \int_D \pf_j \nabla u\cdot\nabla v \,dx ,\quad i \in\cI,
\end{equation*}
for $u,v\in V$.

Thus,  $V$ is typically a function space defined over some physical domain $D\subset \R^m$, with $m=1,2,3$, for example the 
Sobolev space $H^1_0(D)$ in the above case of second order elliptic equations with homogeneous Dirichlet
boundary conditions. Therefore the solution may either be viewed
as the Hilbert space valued map 
\begin{equation}
y\mapsto u(y), 
\label{solutionmap}
\end{equation}
which acts from $\pdom$ to $V$ or as the scalar valued map
\be
(x,y)\mapsto u(x,y):=u(y)(x),
\ee
where $x\in D$ and $y\in \pdom$ are referred to as the spatial and parametric variables.

Approximating such solution maps amounts to approximating functions of a large or even infinite number of variables. In applications, one is often interested in specific functionals of the solution. Here we focus on the basic question of approximating the entire solution map in an appropriate norm.

The guiding questions, to be made precise below, are the following:
What are the most suitable approximations to cope with the high dimensionality in problems of the form \eqref{paramgeneral}, and what features of problems \eqref{paramgeneral} favor certain approaches over others?
Moreover, at what numerical cost can one find these approximations, and how do these costs depend on particular features of the given problem?
To address the latter question, for each setting we construct adaptive computational methods that exhibit near-optimal complexity, in a sense to be made precise below.

\subsection{Sparse and low-rank approximability}\label{sec:introapprox}

Before addressing any concrete numerical schemes, we discuss basic concepts of approximations for the solution map $u$ in \eqref{solutionmap}. 
We focus on the mean-square error $\|u-\tilde u\|_{L^2(\pdom,V)}$ for an approximation $\tilde u$, where
\[
   \|v\|_{L^2(\pdom,V)}^2:=\int_\pdom \|v(y)\|_V^2 d\mu(y)
\]
for a given probability measure  $\mu$ over $\pdom$. In what follows, we assume that $\mu$ is the uniform probability measure on $\pdom$. The results carry over, however, to other product measures on $\pdom$.
The following types of approximation make essential use of the {\em tensor product structure} of the Bochner space 
$L^2(\pdom,V)=V\otimes L^2(\pdom)$ where $L^2(\pdom)=L^2(\pdom,\mu)$.

\paragraph{Sparse polynomial expansions.} 
A first approach to approximating $y\mapsto u(y)$ is to employ an {\em a priorily
chosen} basis  
$\{u_1^{\rp}, \dots, u_n^{\rp}\}\subset L^2(Y)$, and compute the $u_i^{\rs}\in V$
as the corresponding coefficients of this approximation.  One prominent example
of this approach are {orthogonal polynomial expansion} methods, see e.g. \cite{GhanemSpanos:90, GhanemSpanos:07,MK:10,Xiu:10}. In this case, the parametric
functions $u_i^{\rp}$ are picked  from the set of {\em tensorized} Legendre
polynomials
\be
L_\nu(y)=\prod_{j\geq 1} L_{\nu_j}(y_j), \quad \nu=(\nu_j)_{j\geq 1},
\label{tensorleg}
\ee
with $(L_k)_{k\geq 1}$ the univariate Legendre polynomial sequence normalized in $L^2([-1,1],\frac {dt} 2)$.
The functions $(L_\nu)_{\nu\in\cF}$ are an orthonormal basis of $L^2(\pdom)$, where $\cF$ is $\N_0^d$ in the case 
$\cI=\{1,\dots,d\}$ or the set of finitely supported sequences of non-negative integers in the case $\cI=\N$, that is
\be
\pidx := \{ \nu\in\N_0^{\N} \colon \#\supp \nu < \infty \}.
\ee
One thus
has
\be
\label{Legendre}
u(y)=\sum_{\nu\in\cF} u_\nu L_\nu(y), \quad u_\nu= \int_\pdom u(y)L_\nu(y) d\mu(y).
\ee
Then, one natural choice is the best $n$-term approximation $u_n$ obtained by restricting the above expansion
to the set $\Lambda^\rp_{n}\subset \cF$ of indices $\nu$ corresponding to the $n$ largest $\|u_\nu\|_V$,
since this set minimizes the error $\|u-u_n\|_{L^2(\pdom,V)}$ among all possible choices
of $n$-term truncations. This strategy for generating sparse polynomial 
approximations in the context of parametric PDEs was first introduced and analyzed in \cite{cds1,cds2}. In practice, 
the set $\Lambda^\rp_n$ is not accessible, but provides a benchmark for the performance of algorithms.

This {\em representational} complexity, however, does not yet determine the resulting {\em computational} complexity, since the coefficients $u_\nu$ in \eqref{Legendre} in turn need to be approximated as well. For instance, one may choose a fixed basis $\{ \psi_\lambda\}_{\lambda\in\sidx}$ of $V$ and expand
\[
     u_\nu = \sum_{\lambda\in \sidx} \bu_{\lambda,\nu} \psi_\lambda ,
\]
in \eqref{Legendre}, where $\bu_{\lambda,\nu}\in\R$. The simplest strategy is to use the same discretization for all $u_\nu$ by selecting a finite $\Lambda^\rs$, which yields the approximation
\[
  u \approx \sum_{\nu \in \Lambda_{n}^\rp}  \Bigl( \sum_{\lambda\in \Lambda^\rs} \bu_{\lambda,\nu} \psi_\lambda \Bigr)\otimes L_\nu.
\]
Using instead an independently adapted spatial discretization for each $u_\nu$ corresponds to \emph{adaptive sparse polynomial} approximations of the form
\begin{equation}\label{fullnterm}\tag{ASP}
     u \approx \sum_{(\lambda, \nu) \in \Lambda} \bu_{\lambda,\nu}\, \psi_\lambda\otimes L_\nu.
\end{equation}
with $\Lambda \subset \sidx \times \pidx$. It is natural to quantify the complexity of such an approximation by the number of activated degrees of freedom $\#\Lambda$.
Here one can again ask for best $N$-term approximations, now with respect to the fixed basis $\{\psi_\lambda\otimes L_\nu\}_{\lambda\in\sidx,\nu\in\pidx}$, obtained by minimizing the error over all $\Lambda$ with $\#\Lambda = N$. This now results in a fully discrete approximation.

\paragraph{Low-rank approximation.}
More generally, one may consider approximations of the form
\begin{equation}\label{basiclrapprox}
 u \approx u_n :=\sum_{k=1}^n u^\rs_k \otimes u^\rp_k,
\end{equation}
where $u^\rs_k$ and $u^\rp_k$ are functions of the spatial and parametric variable, respectively. This contains \eqref{Legendre} as a special case, but we now allow also $u^\rp_k \in L^2(Y)$ to be arbitrary functions that are not given a priori, but adapted to the given problem. 

The shortest expansion of the form \eqref{basiclrapprox} that achieves a prescribed error in  $L^2(\pdom,V)$ is  given by truncation of the Hilbert-Schmidt decomposition of $u$ interpreted as the operator 
\be\label{Tu}
T_u:  v \to \int_Y u(x,y)v(y) d\mu(y),
\ee
acting from $L^2(\pdom)$ to $V$. 
In this context, we define $\rank(u)$ as the rank of the operator $T_u$, so that in particular $u_n$ with a representation by $n$ separable terms as in \eqref{basiclrapprox} has $\rank(u_n)\leq n$.
The functions $u_1^\rs, \dots, u_n^\rs$ and $u_1^\rp, \dots, u_n^\rp$ are given by the left and right singular functions, respectively, which yield the optimal rank-$n$ approximation of $u$ in $L^2(\pdom,V)$.

 This particular system of  basis functions is a natural benchmark as it minimizes the rank $n=n(\varepsilon)$ required to ensure
a mean-square accuracy $\varepsilon$. However, it is  not obvious how to compute sufficiently good approximations of these
basis functions at affordable cost, a point to be taken up again later.

The methods considered in this paper are based on computing approximations of both $u_k^\rs$ and $u_k^\rp$.
A low-rank approximation trying to approximately realizing a truncated Hilbert-Schmidt decomposition would be a first example
for this category   aiming at meeting the above mentioned benchmark. In this case the error caused by truncation
should ideally be balanced against the error in approximating the {\em unknown} basis functions $u^\rs_k, u^\rp_k$.

Note that there exist alternative approaches for deriving computable expansions of the form
\eqref{basiclrapprox}, where only the functions $u_1^{\rs}, \dots, u_n^{\rs}$ and their span $V_n$ are constructed, which we comment on in \S\ref{sec:relationprevious}.

To obtain numerically realizable approximations, we may again use bases of $V$ and $L^2(Y)$ as in \eqref{fullnterm} and consider expansions for $u^\rs_k$ and $u^\rp_k$ to arrive at fully discrete \emph{low-rank} approximations of the form
\begin{equation}\label{fulllr}\tag{LR}
    u \approx \sum_{k=1}^n \Bigl( \sum_{\lambda\in \Lambda^\rs_k} \bu^\rs_{k,\lambda} \psi_\lambda \Bigr)\otimes \Bigl(\sum_{\nu\in \Lambda^\rp_k} \bu^\rp_{k,\nu} L_\nu \Bigr) .
\end{equation}
with $\Lambda^\rs_k \subset \sidx$, $\Lambda^\rp_k\subset \pidx$, $k=1,\ldots,n$.

\paragraph{Hierarchical tensor decompositions.}
One may as well go beyond the Hilbert-Schmidt decomposition \eqref{basiclrapprox}
and consider higher-order low-rank tensor representations that correspond to further decompositions of the factors $u^\rp_k$ in \eref{basiclrapprox}. 
For simplicity, at this point let us consider this in the finite-dimensional case $d<\infty$, possibly after truncating the expansion \eref{Arep} for $A(y)$.
Introducing an additional tensor decomposition of the factors $\bu^\rp_k$, we obtain the general approximations in \emph{subspace-based tensor formats},
\be
\label{general}\tag{STF}
u_n = \sum_{k_\rs=1}^{r_\rs} u_{k_\rs}^\rs \otimes \Big(\sum_{k_1=1}^{r_1}\cdots\sum_{k_d=1}^{r_d} \ba_{k_\rs,k_1,\ldots,k_d}\bigotimes_{j=1}^d u^{\rp,j}_{k_j}
\Big),
\ee
where each $u_{k_j}^{\rp,j}$ is a function of the individual variable $y_j$. The minimal $r_j$ such that $u_n$ can be represented in the form \eqref{general} are called \emph{multilinear ranks} of $u_n$.

We confine our discussion to \emph{hierarchical tensor representations} (with the \emph{tensor train format} as a special case), see e.g.\ \cite{Grasedyck:10,Hackbusch:12,Oseledets:09}, where the high-order core tensor ${\ba}=(\ba_{k_\rs,k_1,\ldots,k_d})_{k_\rs,k_1,\ldots,k_d}$ is further decomposed in terms of lower-order tensors, based on matricizations of $\ba$.
For instance, if
\be
\label{ranks}
\tilde r_i= \rank\big(\ba_{(k_0,\ldots, k_i),(k_{i+1},\ldots k_d)}\big),
\ee
one has a factorized representation of the form
\begin{equation}\label{htcore}
	  {\ba}_{k_\rs,k_1,\ldots,k_d} = \sum_{\ell_1=1}^{\tilde r_1}\bM^{(1)}_{k_\rs,\ell_1} \sum_{\ell_2=1}^{\tilde r_2} \bM^{(2)}_{\ell_1,k_1,\ell_2}  \cdots \sum_{\ell_{d-1}=1}^{\tilde r_{d-1}} \bM^{(d-1)}_{\ell_{d-2},k_{d-2},\ell_{d-1}} \bM^{(d)}_{\ell_{d-1},k_{d-1},k_d} 
\end{equation}
in terms of the tensors $\bM^{(i)}$, $i=1,\ldots,d$, of order at most three, and only these low-order tensors need to be stored and manipulated.

The representation \eqref{general} contains \eqref{fullnterm} and \eqref{fulllr} as special cases. For instance, to recover a sparse polynomial expansion \eqref{fullnterm}, let $\nu(k_\rs)$, $k_\rs=1,\ldots,r_\rs$, be an enumeration of elements of $\N^d_0$, and 
choose $\ba_{k_\rs,k_1,\ldots,k_d} = \delta_{\nu(k_\rs),(k_1,\ldots,k_d)}$, $u^{\rp,j}_{k_j}= L_{k_j}$. 
With $r_\rs = r_1=\ldots=r_d$ and diagonal core tensor $\ba$ having nonzero entries $\ba_{k,k,\ldots,k}=1$ for $k = 1,\ldots, r_\rs$, one obtains representations \eqref{fulllr}.

\subsection{Guiding questions}

In the different types of approximation outlined above, the degrees of freedom enter in varying degrees of nonlinearity. More strongly nonlinear approximations \eqref{general} with hierarchical decomposition \eqref{htcore} can potentially yield more strongly compressed representations, in the sense that the number of degrees of freedom $n_{\text{dof}}(\varepsilon)$ required for a target accuracy $\varepsilon$ in $L^2(Y,V)$ scales more favorably. Handling this stronger compression in the computation of such representations, however, leads to additional difficulties, and the number of required operations $n_{\rm op}(\varepsilon)$ may in fact scale less favorably than $n_{\text{dof}}(\varepsilon)$.

Here we aim for algorithms which, for each of the above types of approximation, are guaranteed to achieve any prescribed accuracy $\varepsilon$, and which are universal. This means that they do not require a priori knowledge on the approximability of the solution (e.g., on the decay of coefficients), but adjust to such approximability automatically. This goes hand in hand with a mechanism for obtaining a posteriori error bounds, making use only of the given data.
This leads us to our first guiding question:\medskip

	\noindent{\rm(I)} {\it For a given parametric problem and approximation format \eref{fullnterm}, \eref{fulllr} or \eref{general}, can one contrive a universal numerical scheme that can achieve any given target accuracy $\varepsilon$, with approximate solutions close to the minimum required representation complexity $n_{\rm dof} (\varepsilon)$, and can $n_{\rm op}$ be related to $\varepsilon$ and hence to $n_{\rm dof}(\varepsilon)$?}\medskip

The minimum required representation complexity can be expressed in terms of the intrinsic approximability properties of the parametrized solutions $u$ in each of the formats. The corresponding required number of operations also depends on the problem data that are used in the solution process.
We construct algorithms, based on a common generic strategy for \eref{fullnterm}, \eref{fulllr}, and \eref{general}, which are near-optimal in this regard. With such algorithms at hand, a natural further question is the following. \medskip

\noindent{\rm (II) \it Which of the approximation types \eref{fullnterm}, \eref{fulllr}, or \eref{general} is best suited for a given parametric problem, in the sense of leading to the smallest growth of $n_{\rm op}(\varepsilon)$ as $\varepsilon\to 0$?}\medskip
	
This amounts to asking for which parametric problems the investment into approximations of higher structural nonlinearity pays off, or conversely, for which problems possible gains in approximation efficiency are offset by more demanding computations.
We address this point by analyses of the approximability of model problems, complemented by numerical experiments, with conclusions depending on the particular problem type. 

For problems with finitely many parameters that are each of comparable influence, hierarchical tensor representations of the form \eqref{general} with \eqref{htcore} turn out to be clearly advantageous. 
	In the case of an anisotropic dependence on infinitely many parameters, for representative model problems we demonstrate that \eqref{fullnterm} can in general yield faster convergence than \eqref{fulllr} or \eqref{general}.
The particular structure of such infinite parameter expansions also turns out to have a major influence on the efficiency of the adaptive schemes.

\subsection{Relation to previous work}
\label{sec:relationprevious}

There is a variety of results on the convergence of sparse polynomial expansions \eqref{Legendre}, see, e.g., \cite{cds1,cds2,BCM:2015}. Furthermore, some estimates are available that include multilevel spatial discretizations and hence provide upper bounds for the error of best $n$-term approximation \eqref{fullnterm}, see, e.g., \cite{cds2,Dung:15}.
Concerning our question (II), there are only few specialized results comparing the different approximation
formats. In the case of general bivariate functions, a systematic comparison between sparse grids and 
low rank approximation is discussed in \cite{GH:14}, showing in particular that for 
Sobolev classes the latter does not bring any improvement. In the case of high-dimensional functions
associated to parametric PDEs, possible gains by low-rank approximations
have been identified in \cite{LMQR:13,BC:15} by exploiting the particular structure of the problem,  
all concerning the case of finitely many parameters. 

There are various approaches for generating sparse polynomial expansions, for instance based on collocation \cite{BNT:10,BTNT:12} or adaptive Taylor expansion \cite{CCS:14}. Note that these strategies do not currently yield {\it a posteriori} error bounds for the computed solutions, and their performance is thus described by {\it a priori} estimates which may not be sharp.

The adaptive methods proposed in \cite{EGSZ:14,EGSZ:15}, based on finite element discretization for the spatial variable, yields a posteriori error bounds for the full approximations. However, the complexity bounds proven in \cite{EGSZ:15} 
 are given only in terms of the resulting finite element meshes.

Adaptive schemes using wavelet-based spatial discretizations, which yield approximations of the form \eqref{fullnterm}, have been studied by Gittelson \cite{Gittelson:13, Gittelson:14}. In this case, bounds for the complete computational complexity are proven which, however, do not fully comply with
the approximability properties of the solution.

Reduced basis and POD methods \cite{RHP:08,KV:07,LMQR:13,R:14} correspond to expansions of the form \eqref{basiclrapprox}, where only the spatial basis elements $u^\rs_k$ spanning $V_n$ are explicitly computed in an offline stage. Then, in an online stage, for any given $y\in \pdom$, the approximate solution $u_r(y)$ is defined as the Galerkin projection of $u(y)$
on the space $V_n$. For known variants of these methods, accuracy guarantees in the respective norms (where reduced basis methods usually aim at the error in $L^\infty$-norm $\|v\|_{L^\infty(\pdom,V)} :=\sup_{y\in \pdom} \|v(y)\|_V$) require a sufficiently dense sampling of the parameter domain.
This  becomes prohibitive for large $d$, and one only obtains a posteriori bounds for the resulting $V$-error in each given $y\in\pdom$.

In methods based on higher-order tensor representations, instead of sampling in the parameter domain, one also approximates $u^\rp_k$ as in \eqref{general}, at the price of additional approximability requirements as in \eqref{htcore}. 
A variety of schemes have been proposed that operate on fixed discretizations \cite{Khoromskij:11-1,KressnerTobler:11,KO:10,Matthies:12}, which do not yield information on the discretization error.
Based on \cite{EGSZ:14}, an adaptive scheme for hierarchical tensor approximation is proposed in \cite{EPS:15}. It provides rigorous a posteriori bounds for the approximation error, but is not proven to converge.

\subsection{Novelty of the paper and outline}

Question (I) is addressed in sections \S \ref{sec:generic} to \S \ref{sec:lowrank}. A generic
algorithm is described in \S \ref{sec:generic} based on the work in \cite{BD}, which is {\em guaranteed} to converge without any a priori assumptions 
on the solution. Furthermore, it yields rigorous a posteriori error bounds, using only information on the problem data. Suitable specifications cover all above mentioned types of approximations \eqref{fullnterm}, \eqref{fulllr}, and \eqref{general}. The scheme is formulated in 
a general sequence space framework, using a 
discretization of the space $L^2(\pdom,V)$ through a basis with elements
of the form $\psi_\lambda\otimes L_\nu$. Here, $\{\psi_\lambda\}_{\mu\in \sidx}$
is a given Riesz basis of $V$ (for example, a wavelet basis in the case where $V$ is a
Sobolev space) and $\{L_\nu\}_{\nu \in\pidx}$ is the previously
described multivariate Legendre basis. The algorithm performs an iteration
in the sequence space $\ell^2(\sidx\times \pidx)$. It involves at each
step specific routines $\recompress$ and $\coarsen$ aiming at  controlling
the rank of the current approximation as well as the number of degrees of
freedom in each of its factors, respectively.

We then describe realizations
of this generic algorithm corresponding to 
two distinct settings. In \S \ref{sec:iso} we apply the algorithm for the generation
of approximations \eref{general} in the setting of finitely many parametric variables. In this case the 
$\recompress$ routine is based on a truncation of a hierarchical singular value decomposition of the coefficient tensor.
We analyze the performance of the algorithms for classes described
by the decay of the corresponding singular values and joint sparsity of the
corresponding singular vectors. 

\S \ref{sec:sparselegendre} and \S\ref{sec:lowrank} are devoted to the case of {\em anisotropic} dependence on infinitely many parameters
in the diffusion problem \eqref{paramdiffusion-0}.
In \S \ref{sec:sparselegendre} we analyze a specialized version of Algorithm \ref{alg:tensor_opeq_solve} producing 
$n$-term sparse Legendre expansions, see \eqref{fullnterm}. In this version the routine $\recompress$ is simply the identity, and hence
Algorithm \ref{alg:tensor_opeq_solve} agrees with the adaptive solver developed and analyzed in \cite{Cohen:02}.
In \S \ref{sec:lowrank} we consider, in the same setting as in \S \ref{sec:sparselegendre}, a solver for approximations \eqref{fulllr}.
In this case the $\recompress$ routine is based on standard SVD truncation.
The corresponding notions of approximability are analogous to those arising in \S\ref{sec:iso}.

A key ingredient in \S \ref{sec:sparselegendre} and \S\ref{sec:lowrank} is the adaptive approximation of the operator based on matrix compression results in Appendix \ref{sec:opapprox}.
Here we obtain new estimates for wavelet-type multilevel expansions of the parametrized coefficients that are more favorable than what is known for Karhunen-Lo\`eve-type expansions.
Our further algorithmic developments also require substantially weaker assumptions on the $A_j$ in \eqref{Arep} than the methods in \cite{EGSZ:14,EPS:15}, which require summability of $(\norm{A_j})_{j\geq 1}$.
By the new operator compression results, we establish, in particular, computational complexity estimates for \eqref{fullnterm} which significantly improve on those of similar schemes in \cite{Gittelson:13,Gittelson:14}.

Based on these complexity estimates, question (II) is then addressed in \S \ref{sec:approximability}.
While the presented algorithms are guaranteed to converge, the corresponding computational cost can only be quantified in terms of approximability properties of solutions $u$. In \S \ref{sec:approximability}, we study the corresponding properties, which are different for each realization of the scheme, in representative examples of parametric problems of the form \eqref{paramdiffusion-0}. 
In particular, for a certain class of such problems,
we prove that the best $n$-term Legendre approximation is already asymptotically near-optimal among all rank-$n$ approximations. 
For other examples, we prove that optimized low-rank approximations can achieve significantly better complexity than best $n$-term Legendre approximations.
This is illustrated further by numerical tests, demonstrating that these observations also hold for more involved model problems.

\section{A generic algorithm}\label{sec:generic}

In this section, we follow the approach developed in \cite{BD},
by first reformulating the general equation \eqref{paramgeneral} in a sequence space,
and then introducing a generic resolution algorithm based on this 
equivalent formulation.

We first notice that \eqref{paramgeneral} may also be written as
\be
Au=f,
\label{paramgeneral1}
\ee
where $A$ is elliptic and boundedly invertible from $L^2(\pdom,V)$ to $L^2(\pdom,V')$ and can be defined in a weak sense by
\be
\langle Au,v\rangle :=\int_{\pdom}\langle A(y)u(y),v(y)\rangle d\mu(y), \quad u,v\in L^2(\pdom,V).
\ee
We assume that $f\in L^2(\pdom,V')$, so that there exists a unique solution $u\in L^2(\pdom,V)$.

Given a Riesz basis $\{ \psi_\lambda \}_{\lambda \in \sidx}$ of $V$, we tensorize it
with the orthonormal basis $\{L_\nu\}_{\nu\in \pidx}$ of $L^2(\pdom)$. The resulting
system $\{\psi_\lambda \otimes L_\nu\}_{(\lambda,\nu)\in \sidx \times \pidx}$ is a Riesz basis of $L_2(\pdom,V)$,
which we now use to discretize \eref{paramgeneral1}. For this purpose, we define the matrices
\be
\bA_j := \bigl( \langle A_j \psi_{\lambda'}, \psi_{\lambda}\rangle \bigr)_{\lambda,\lambda' \in \sidx}
\quad{\rm and} \quad  \bM_j = \left( \int_{\pdom} y_j  L_\nu(y) L_{\nu'}(y) \,d\mu(y)  \right)_{\nu,\nu'\in\pidx},
\ee
where $\bM_0$ is set to be the identity on $\ell^2(\cF)$, and the right hand side column vector
\be
\bbf := \big(\langle f, \psi_\lambda\otimes L_\nu\rangle\big)_{(\lambda,\nu)\in \sidx\times\pidx}.
\ee
We thus obtain an equivalent problem
\beqn
\label{equiv}
\bA \bu = \bbf
\eeqn
on $\spl{2}(\sidx \times \pidx)$, where
\beqn
\label{Arep2}
 \bA := \sum_{j\geq 0} \bA_j \otimes \bM_j  
 \eeqn
and $\bu=\big(u_{\lambda,\nu}\big)_{(\mu,\nu)\in \sidx\times\pidx}$ is the coordinate vector 
of $u$ in the basis  $\{\psi_\mu \otimes L_\nu\}_{(\mu,\nu)\in \sidx\times\pidx}$. 

Regarding $\nu\in\pidx$ as the column index of the infinite matrix $\bu=(\bu_{\mu,\nu})_{\mu \in \sidx,\nu\in\pidx}$, we denote by $\bu_\nu$ the columns of $\bu$, which are precisely the basis representations of the Legendre coefficients $u_\nu\in V$.

 In what follows we always denote by $\norm{\cdot}$ the $\spl{2}$-norm on the respective index set which could be $\sidx$, $\pidx$ or $\sidx\times\pidx$, or the corresponding operator norm when this is clear from the context.
Since $\{\psi_\mu\}_{\mu\in \sidx}$ is a Riesz basis for $V$ we have $\|u_\nu\|_{V}\sim \|\bu_\nu\|$ uniformly in $\nu\in\pidx$,
which together with boundedness and ellipticity of $A$ implies that $\bA$ is bounded and elliptic on $\spl{2}(\sidx\times\pidx)$ and that we have
\beqn
\label{normequiv2}
\|\bu\|\sim \|\bA\bu\|\sim \|Au\|_{L^2(\pdom,V')}\sim \|u\|_{L^2(\pdom,V)}
\eeqn
with uniform constants. 
On account of \eqref{normequiv2}, solving \eqref{equiv} approximately up to some target accuracy is equivalent to solving \eqref{equiv} in $\spl{2}$
to essentially the same accuracy. 

As a further consequence, one can find a fixed positive $\omega$ such that $\norm{\mathbf{I} - \omega \bA}\le \rho <1$, ensuring that a simple Richardson iteration converges with a fixed error reduction rate per step.
This serves as the conceptual starting point for the adaptive low-rank approximation scheme introduced in \cite{BD} as given in Algorithm \ref{alg:tensor_opeq_solve}.

\begin{algorithm}[!ht]
\caption{$\quad \mathbf{u}_\varepsilon = \solve(\mathbf{A},
\mathbf{f}; \varepsilon)$} \begin{algorithmic}[1]
\Require \begin{minipage}{12cm}$\omega >0$ and $\rho\in(0,1)$ such that
$\norm{\id - \omega\mathbf{A}} \leq \rho$, $\lambda_{\bA} \leq \norm{\bA^{-1}}^{-1}$, \\
$\kappa_1, \kappa_2, \kappa_3 \in (0,1)$ with $\kappa_1 +
\kappa_2 + \kappa_3 \leq 1$, and $\beta \geq 0$.\end{minipage}
\Ensure $\mathbf{u}_\varepsilon$ satisfying $\norm{\mathbf{u}_\varepsilon -
\mathbf{u}}\leq \varepsilon$.
\State $\mathbf{u}_0 := 0$, $\delta := \lambda_\mathbf{A}^{-1}
\norm{\mathbf{f}}$ 
\State $k:= 0$, $J := \min\{ j \colon \rho^j (1 + (\omega + \beta) j) \leq
\frac12\kappa_1\}$\label{alg:jchoice}
\While{$\frac1{2^k} \delta > \varepsilon$}
\State $\mathbf{w}_0:=\mathbf{u}_k$, $j \gets 0$
\Repeat
\State $\eta_j := \rho^{j+1} \frac1{2^k}\delta$
\State $\mathbf{r}_j := \apply( \mathbf{w}_j ; \frac{1}{2}\eta_j)
- \rhs(\frac{1}{2}\eta_j)$ 
\State $\mathbf{w}_{j+1} := \recompress(\mathbf{w}_j - \omega \mathbf{r}_j ;
\beta \eta_j)$ \label{alg:tensor_solve_innerrecomp}
\State $j\gets j+1$.
\Until{($j \geq J \quad \vee \quad \lambda_{\mathbf{A}}^{-1} \rho
\norm{\mathbf{r}_{j-1}} + (\lambda_{\mathbf{A}}^{-1} \rho  + \omega + \beta)
\eta_{j-1} \leq \frac1{2^{k+1}} \kappa_1 \delta$)} \label{alg:cddtwo_looptermination_line}
\State $\mathbf{u}_{k+1} := \coarsen\bigl(\recompress(\mathbf{w}_j;
\frac1{2^{k+1}} \kappa_2  \delta) ; \frac1{2^{k+1}}\kappa_3
\delta\bigr)$\label{alg:cddtwo_coarsen_line} 
\State $k \gets k+1$
\EndWhile
\State $\mathbf{u}_\varepsilon := \mathbf{u}_k$ 
\end{algorithmic}
\label{alg:tensor_opeq_solve}
\end{algorithm}

This basic algorithmic template can be used to produce various types of sparse and low-rank approximations, with appropriate choices of the subroutines $\apply$, $\rhs$, $\coarsen$, and $\recompress$.

The procedures $\coarsen$ and $\recompress$ are independent of the considered $\bA$ and $\bbf$, and satisfy
\begin{equation}\label{craccuracy}
  \norm{\coarsen(\bv ; \eta ) -  \bv } \leq \eta,\quad  \norm{\recompress(\bv;\eta) - \bv} \leq \eta,
\end{equation}
for any $\eta \geq 0$ and any compactly supported $\bv \in \spl{2}(\sidx\times\pidx)$. Here $\coarsen$ is intended to reduce the support of the sequence $\bv$, whereas $\recompress$ reduces the rank of $\bv$ in a low-rank tensor representation. The particular realizations of these routines depend on the dimensionality of the problem and on the type of approximation. We shall use the constructions given in \cite{BD}. In the case of the sparse approximations considered in \S\ref{sec:sparselegendre}, $\recompress$ is chosen as the identity, and Algorithm \ref{alg:tensor_opeq_solve} essentially reduces to the method analyzed in \cite{Cohen:02}.

The routines $\apply$ and $\rhs$ are assumed to satisfy, for compactly supported $\bv$ and any $\eta>0$, the requirements
\begin{equation}\label{applyaccuracy}
  \norm{\apply(\bv ; \eta ) - \bA \bv } \leq \eta,\quad  \norm{\rhs(\eta) - \bbf} \leq \eta.
\end{equation}
Their construction not only depends on the type of approximation, but also on the specific problem under consideration.
These two routines are indeed the main driver of adaptivity in Algorithm \ref{alg:tensor_opeq_solve}, and a major part of what follows concerns the construction of $\apply$ in different scenarios. 

It hinges on the compression of matrices by exploiting their near-sparsity in certain basis representations.
We use the following notion introduced in \cite{Cohen:01}: A bi-infinite matrix $\bB$ is called \emph{$s^*$-compressible} if there exist matrices $\bB_{n}$ with $\alpha_{n} 2^n$ entries per row and column and such that 
\beqn
\label{compress}
\norm{\bB - \bB_{n}} \leq \beta_{n} 2^{-s n}, \quad \mbox{ for  }\,\, 0 < s < s^*,
\eeqn 
and where the sequences $\ab = (\alpha_n)_{n\in\N_0}$ and $\bb = (\beta_n)_{n\in\N_0}$ are summable. Here we always assume $\bB_{0} = 0$.

\begin{remark}
\label{rem:termination}
As shown in \cite{BD}, regardless of the specifications of the routines  $\apply, \rhs$, $\coarsen, \recompress$,
Algorithm \ref{alg:tensor_opeq_solve}
terminates after finitely many steps and its output $\bu_\varepsilon$ satisfies $\|\bu - \bu_\varepsilon\|\le \varepsilon$.
\end{remark}

At this point, we record for later usage a particular feature of $\bA$ that arises as a consequence of our choice of tensor product orthogonal polynomials for the parameter-dependence:
The approximate application of $\bA$ is facilitated by the fact that
the matrices $\bM_j$ are {\em bidiagonal}. That is, in view of the three-term recurrence relation
\beqn
\label{threeterm}
tL_n(t)= p_{n+1} L_{n+1}(t) + p_{n} L_{n-1}(t), \quad L_{-1}\equiv 0,
\eeqn
where
\beqn
\label{recurr_coeff}
p_0 = 0, \qquad p_n= \frac{1}{\sqrt{4 - n^{-2}}} , \quad n>0,
\eeqn
one has $\int_U y_j\,L_\nu(y)L_\mu(y)\,d\mu(y) =0$ whenever $j\notin \supp\,\nu \cup \supp\,\mu$, providing
\beqn
\label{threeterm_matrix}
(\bM_j)_{\nu,\nu'}= p_{\nu_j}\delta_{\nu +e^j,\nu'}+ p_{\nu_j-1}\delta_{\nu-e^j,\nu'}
\eeqn
with the Kronecker sequence $(e^j_i)_{i\in \cI}:=(\delta_{i,j} )_{i\in\cI} \in \pidx$.

\section{Hierarchical tensor representations in the case of finitely many parameters}
\label{sec:iso}

We begin by considering the setting
\begin{equation}
\label{cId}
 \cI=\{1,\dots,d\},
\end{equation}
where $\pidx = \N_0^d$ and $\bu \in \spl{2}(\sidx \times \N_0\times\cdots\times\N_0)$. 

Here we are interested in the case that all coordinates in $\cI$ have comparable influence. As illustrated in \S\ref{sec:approximability}, a direct sparse Legendre expansion of $\bu$ over $\sidx\times\pidx$ will then in general be infeasible already for moderately large $d$. However, one may as well exploit Cartesian product structure in $\pidx$, regarding $\bu$ as a higher-order tensor, and use corresponding hierarchical low-rank representations. As we shall detail in what follows, the results of \cite{BD} can be adapted to this problem in a rather straightforward manner. 

It will be convenient to introduce a numbering of tensor modes as follows: $\xidx_\rs := \sidx$, $\xidx_1 := \N_0$, \ldots, $\xidx_d := \N_0$. We additionally introduce the notation 
\[
  \hat\cI := \{\rs\}\cup\cI .
\]
The representations of higher-order tensors which we consider are built on the Hilbert-Schmidt case via \emph{matricizations}: for each nonempty $M \subset \hat\cI$, any $\bv \in \spl{2}(\sidx\times\pidx)$ induces a compact operator $T_\bv^{(M)} \colon \spl{2}(\bigtimes_{i\in \hat\cI\setminus M} \xidx_i) \to \spl{2}(\bigtimes_{i\in M} \xidx_i)$.

In terms of the left singular vectors $\{ \bU^{(i)}_k \}_{k \in \N}$ of $T_\bv^{(\{i\})}$, $i\in \hat\cI$, we obtain the \emph{HOSVD representation} \cite{Lathauwer:00} in the \emph{Tucker format} \cite{Tucker:64,Tucker:66},
\begin{equation}\label{tuckerformat}
  \bv = \sum_{k_\rs=1}^{r_\rs} \sum_{k_1 = 1}^{r_1} \cdots \sum_{k_d=1}^{r_d} \ba_{k_\rs,k_1,\ldots,k_d} \bigotimes_{i\in\hat\cI} \bU^{(i)}_{k_i}.
\end{equation}
Here the tensor $\ba = (\ba_{\kk{k}})_{\kk{k}\in\N^{d+1}}$ of order $d+1$ is referred to as \emph{core tensor}, and $(r_\rs,r_1,\ldots,r_d)$ as the \emph{multilinear ranks} of $\bv$.

The \emph{hierarchical tensor format} \cite{Hackbusch:09-1}, on which the variant of our scheme described in this section is based, can be interpreted as a further decomposition of $\ba$ into tensors of order at most three.
This decomposition is obtained using further matricizations of the tensor according to a recursive decomposition of the set of modes $\hat\cI$ into a binary tree, which we denote by $\mathcal{D}$. {For simplicity, we focus in our exposition on linear trees corresponding to factorizations \eqref{htcore}, where 
\begin{equation}\label{lindimtree}
  \mathcal{D} = \bigl\{ \hat\cI, \{\rs\}, \{1,\ldots,d\}, \{1\}, \{2,\ldots, d\}, \ldots, \{d-1\}, \{d\} \bigr\}.
\end{equation}}
For each $\alpha \in \mathcal{D}$, the rank of the corresponding matricization $T^{(\alpha)}_\bv$ is denoted by $\rank_\alpha(\bv)$, its singular values by $(\sigma^{(\alpha)}_k)_{k\in\N}$. Here $\rank_{\hat\cI}(\bv) = 1$ for all $\bv \neq 0$, and we set
\begin{equation}
	\rank(\bv) := \bigl(\rank_\alpha(\bv)\bigr)_{\alpha\in\mathcal{D}\setminus\hat\cI}.
	\label{multirank}
\end{equation}

{In this section we specialize the generic template Algorithm \ref{alg:tensor_opeq_solve} to produce approximate solutions to \eqref{paramgeneral1}
of the form \eqref{general} with core tensor in hierarchical form as in \eqref{htcore}. More precisely, the output is given in the form of a tensor $\bu_\varepsilon$ of order $d+1$, supported on $\Lambda^\varepsilon_\rs\times \Lambda^\varepsilon_1 \times\cdots\times\Lambda^\varepsilon_d$ with finite $\Lambda^\varepsilon_i \subset \xidx_i$. This tensor is given (assuming $\mathcal{D}$ as in \eqref{lindimtree}) in the following form: let $r^\varepsilon_i := \rank_{\{i\}}(\bu_\varepsilon)$ for $i \in \hat \cI$ and $\tilde r^\varepsilon_i := \rank_{\{i+1,\ldots, d\}}(\bu_\varepsilon)$ for $i = 1,\ldots,d-1$, then
\begin{multline*}
\bu_{\varepsilon, \lambda,\nu_1,\ldots,\nu_d} = \sum_{k_\rs=1}^{r^\varepsilon_\rs} \cdots \sum_{k_d=1}^{r^\varepsilon_d} \sum_{\ell_1=1}^{\tilde r^\varepsilon_1}\cdots\sum_{\ell_{d-1}=1}^{\tilde r^\varepsilon_{d-1}} \bM^{(1),\varepsilon}_{k_\rs,\ell_1} \,\bM^{(2),\varepsilon}_{\ell_1,k_1,\ell_2}  \cdots   \bM^{(d-1),\varepsilon}_{\ell_{d-2},k_{d-2},\ell_{d-1}} \\ \times \bM^{(d),\varepsilon}_{\ell_{d-1},k_{d-1},k_d} \, \bU^{(\rs),\varepsilon}_{k_\rs, \lambda} \, \bU^{(1),\varepsilon}_{k_1, \nu_1}\cdots \bU^{(d),\varepsilon}_{k_d, \nu_d}, \quad \lambda\in \Lambda^\varepsilon_\rs,\; \nu_i \in \Lambda^\varepsilon_i.
\end{multline*}
The adaptive scheme identifies, in an intertwined fashion, the ranks $r^\varepsilon_i$, $\tilde r^\varepsilon_i$, the sets $\Lambda^\varepsilon_i$, as well as the coefficient tensors $\bM^{(i),\varepsilon}$ and $\bU^{(i),\varepsilon}$ of $\bu_\varepsilon$. The function represented by $\bu_\varepsilon$ has precisely the form \eqref{general}, where 
\[
  u^{\rs}_{k_\rs} = \sum_{\lambda \in \sidx} \bU^{(\rs),\varepsilon}_{k_\rs, \lambda} \psi_\lambda, 
   \quad u^{\rp,i}_{k_i} = \sum_{n \in \xidx_i} \bU^{(i),\varepsilon}_{k_i, n} L_n .
\]
}

The hierarchical format can offer substantially more favorable complexity characteristics for large $d$ than \eqref{tuckerformat}. 
The left singular vectors of the involved matricizations yield a \emph{hierarchical singular value decomposition} \cite{Grasedyck:10}.
We refer also to \cite{Falco:10,Kolda:09,Hackbusch:09-1,Grasedyck:13,Hackbusch:12} for detailed expositions 
regarding the finitely supported case (see also \cite{Oseledets:09,Oseledets:11} for the related \emph{tensor train} representation), and to \cite{BD} for  analogous results for tensors in sequence spaces, with notation analogous to the present paper. 

{For quantifying the approximability of tensors on $\bigtimes_{i \in \hat\cI}\xidx_i$ in terms of the best selection of finite $\Lambda^\varepsilon_i \subset \xidx_i$ as above,} a pivotal role is played by the quantitites
\beqn
\label{contr}
\pi^{(i)}(\bv) \coloneqq \bigl(\pi^{(i)}_{\nu_i}(\bv)\bigr)_{\nu_i\in \xidx_i},\quad \pi^{(i)}_{\mu}(\bv) \coloneqq \Bigl(\sum_{\substack{\nu\in\cF\\ \nu_i=\mu}}
|\bv_\nu|^2\Bigr)^{1/2} , \quad i \in \hat\cI = \{\rs, 1, \ldots, d\}.
\eeqn
They are introduced in \cite{BD} and called \emph{contractions} in analogy to the terminology in tensor analysis.
An efficient evaluation (without any $d$-dimensional summations) is possible due to the relation
\beqn
\label{sumN}
\pi^{(i)}_{\mu}(\bv)= \Bigl(\sum_{k}|\bU^{(i)}_{k,\mu}|^2|\sigma^{(i)}_k|^2\Bigr)^{1/2},
\eeqn
where $\sigma^{(i)}_k$ are the mode-$i$ singular values of $\bv$.
As in our previous notation, we abbreviate $\supp_i \bv:= \supp\bigl(\pi^{(i)}(\bv)\bigr)$, $i \in \hat\cI$.

\subsection{Adaptive scheme}

In the present case, we consider Algorithm \ref{alg:tensor_opeq_solve} with the routines $\recompress$ and $\coarsen$ for the hierarchical format as given in \cite[Rem.\ 15]{BD}. 

$\recompress$ is based on a truncation of a hierarchical singular value decomposition up to a prescribed accuracy $\eta>0$, which can be ensured based on the $\ell^2$-norm of omitted singular values of matricizations. We denote this operation by $\hatPsvd{\eta}$. As shown in \cite{Grasedyck:10}, it satisfies the quasi-optimality property 
\begin{equation}\label{recompress quasiopt}
	 \norm{ \bv - \hatPsvd{\eta}(\bv)} \leq \sqrt{2d-3} \inf \bigl\{ \norm{\bv - \bw} \colon {\rank(\bw)\leq \rank(\hatPsvd{\eta}(\bv))} \bigr\},
\end{equation}
with the inequality between ranks as defined in \eref{multirank} to be understood componentwise.

$\coarsen$ retains the degrees of freedom for each mode that correspond to the largest contractions \eqref{contr}.
Let $(\mu^*_{i,k})_{k \in \N}$ be such that $(\pi^{(i)}_{\mu^*_{i,k}}(\bv))_{k\in\N}$ is nonincreasing.
Denote for $\Lambda \subset \sidx\times\pidx$ by $\Restr{\Lambda}\bv$ the array obtained
by retaining all entries of $\bv$ corresponding to indices in $\Lambda$, while replacing all others by zero. Given $\eta >0$, we define the product set 
\[
\Lambda(\eta)= \bigtimes_{i \in \hat\cI}\{ \mu^*_{i,k} \colon k \leq N_i\},
\]
where $N_i$, $i\in \hat\cI$, are chosen to such that $\sum_{i \in \hat\cI} N_i$ is minimal subject to the condition
\beqn
\label{hatC}
\Big( \sum_{i \in \hat\cI} \sum_{k > N_i\vphantom{\hat\cI}}|\pi^{(i)}_{\mu^*_{i,k}}(\bv)|^2 \Big)^{1/2}
 \leq \eta.
\eeqn
Noting that the left side in \eqref{hatC} is an upper bound for $\norm{\bv - \Restr{\Lambda(\eta)}\bv}$, we define $\coarsen$ as a numerical realization of $ \hatCctr{\eta}\bv := \Restr{\Lambda(\eta)}\bv$, for which one has an analogous quasi-optimality property as in \eqref{recompress quasiopt} with constant $\sqrt{d}$.

Furthermore, $\bA$ as defined in \eqref{Arep2} is in the present case of finitely many parameters a \emph{finite} sum of Kronecker product operators, 
\[
  \bA = \sum_{j=0}^d \bA_j \otimes \bM_j,
\]
which considerably simplifies the construction of the corresponding routine $\apply$. The action of $\bA$ can thus increase each hierarchical rank of its argument at most by a factor of $d+1$.

\begin{remark}
In contrast to the case considered in \cite{BD2}, here the Hilbert space $\cH =V \otimes \spL{2}(\pdom)$ on which the problem is posed is endowed with a cross norm. As a consequence, the isomorphism that takes $v\in \cH$ to its coefficients $\bv\in \spl{2}(\sidx\times\pidx)$ with respect to the tensor product basis is of Kronecker rank one. The original low-rank structure \eqref{Arep} of $A(y)$ 
is therefore preserved in the $\spl{2}$-representation \eqref{Arep2} of the problem.
\end{remark}

Consequently, the routine $\apply$ that adaptively approximates the action of $\bA$ can be obtained following the generic construction given in \cite{BD}, provided that the operators $\bA_j$ and $\bM_j$ acting on each mode have the required compressibility properties.
Recall that by \eqref{threeterm_matrix}, the infinite matrices $\bM_j$ are bidiagonal, and hence do not require any further approximation. 
To use the construction of \cite{BD}, we thus only need that the operators $\bA_0, \ldots, \bA_d$ acting on the spatial variables are $s^*$-compressible.

\subsection{Convergence analysis}

Our complexity results aim at the following type of statements: given a certain approximability of the solution, the algorithm recovers the corresponding convergence rates without their explicit knowledge.

To describe these approximability properties, we now recall the definition of approximation classes to quantify the convergence of hierarchical low-rank approximations from \cite{BD}, in terms of the hierarchical rank defined by \eref{multirank}.
Let $\gamma = \bigl(\gamma(n)\bigr)_{n\in \N_0}$ be positive and strictly
increasing with $\gamma(0)=1$
and $\gamma(n)\to\infty$ as $n\to\infty$, for $\bv \in \spl{2}(\sidx\times\pidx)$ let
\[
  \abs{\bv}_{\AH{\gamma}} \coloneqq \sup_{r\in\N_0} \gamma({r})\,\inf_{\abs{\rank(\bw)}_\infty\leq r} \norm{\bv - \bw} ,
\]
where $\inf \{ \norm{\bv - \bw}\colon \abs{\rank(\bw)}_\infty\leq r\}$ -- that is, the errors of approximation by hierarchical ranks at most $r$ -- can be bounded in terms of the hierarchical singular values $(\sigma^{(\alpha)}_k)_{k\in\N}$ for $\alpha\in\cD$.
We introduce the approximation classes
\[
\AH{\gamma} \coloneqq  \bigl\{\bv\in \spl{2}(\sidx\times\pidx) : 
 \abs{\bv}_{\AH{\gamma}} <\infty \bigr\} , \quad 
 \norm{\bv}_{\AH{\gamma}} \coloneqq \norm{\bv} + \abs{\bv}_{\AH{\gamma}}. 
\]
We restrict our considerations to $\gamma$ that satisfy
\begin{equation*}
\rho_\gamma \coloneqq \sup_{n\in\N} \ga(n)/\ga(n-1)<\infty \,,
\end{equation*}
which corresponds to a restriction to at most exponential growth; our considerations would still apply for faster growth, but lead to less sharp results.

For an approximation $\bv$ of bounded support to $\bu$, the number of nonzero coefficients $\#\supp_i \bv$ required in each tensor mode to achieve a certain accuracy depends on the best $n$-term approximability of the sequences $\pi^{(i)}(\bu)$.

 This approximability by sparse sequences is quantified by the classical {\em approximation classes} $\As = \As(\mathcal{J})$, where $s>0$ and $\mathcal{J}$ is a countable index set, comprised of all $\bw\in \ell_2(\mathcal{J})$
for which the quasi-norm
\begin{equation}
\label{Asdef}
  \norm{\mathbf{\bw}}_{\As(\mathcal{J})} := \sup_{N\in\N_0} (N+1)^s
  \inf_{\substack{\Lambda\subset \mathcal{J}\\
  \#\Lambda\leq N}} \norm{\bw - \Restr{\Lambda} \bw}
\end{equation}
is finite. 

 In particular, if $\pi^{(i)}(\bu) \in \As(\mathcal{G}_i)$,
then these sequences can be approximated within accuracy $\eta$ by finitely supported sequences with ${\cal O}(\eta^{-1/s})$ nonzero entries. In what follows, we do not explicitly specify the index set in the spaces $\As$ when this is clear from the context.

We analyze the complexity of the algorithm under the following benchmark assumptions, see the discussion in \S\ref{sec:approx iso}.

\begin{assumptions}
\label{ass:fullbenchmark}
For the hierarchical tensor approximation in the case \eqref{cId} of $d$ parametric variables, we assume the following:
\begin{enumerate}[{\rm(i)}]
\item $\pi^{(i)}(\bu), \pi^{(i)}(\bbf) \in \Acal^{s}(\Gcal_i)$, $i\in\hat\cI$, for an $s>0$.
\item $\bu, \bbf \in \Acal_\Hcal(\gamma)$, where $\gamma(n) \coloneqq e^{\bar c n^{1/\bar b}}$ with $\bar b\geq 1$, $\bar c>0$. 
\item The $\bA_j$, $j \in \hat\cI$, are $s^*$-compressible for an $s^*>s$, and hence there exist matrices $\bA_{j,n}$ with $\alpha_{j,n} 2^n$ entries per row and column and such that 
$\norm{\bA_j - \bA_{j,n}} \leq \beta_{j,n} 2^{-s n}$,
and where the sequences $\ab_j = (\alpha_{j,n})_{n\in\N_0}$ and $\bb_j = (\beta_{j,n})_{n\in\N_0}$ are summable.
\end{enumerate}	
\end{assumptions}

We will use the above assumptions as a reference point for the scaling with respect to $\varepsilon$ of the computational complexity.
Note that Assumptions \ref{ass:fullbenchmark}(i),(ii) mean in particular that best $N$-term approximations of $\pi^{(i)}(\bu)$ converge at least as $\Ocal (N^{-s})$ for $i\in\hat\cI$, and low-rank approximations $\bu_n$ of $\bu$ with $\abs{\rank(\bu_n)}_\infty \leq n$ converge as $\Ocal(e^{-\bar c n^{1/\bar b}})$. Assumption \ref{ass:fullbenchmark}(iii) needs to be established for each given problem.

Under these conditions, it is possible to realize a routine $\rhs$ that satisfies the following conditions: for sufficiently small $\eta>0$ and $\bbf_\eta:=\rhs(\eta)$,
\begin{equation}\label{rhsass1}
\begin{gathered}
	\sum_{i\in \hat\cI} \#\supp_i (\bbf_\eta) \lesssim d \eta^{-\frac1s} \Bigl( \sum_{i\in \hat\cI} \norm{\pi^{(i)}(\bbf)}_{\cA^s} \Bigr)^{\frac1s},
	 \quad \sum_{i\in \hat\cI} \norm{\pi^{(i)}(\bbf_\eta)}_{\cA^s} \lesssim d^{1+\max\{1,s\}} \sum_{i\in \hat\cI} \norm{\pi^{(i)}(\bbf)}_{\cA^s},
	 \\ \abs{\rank(\bbf_\eta)}_\infty \lesssim  \bigl(d^{-1} \ln (\norm{\bbf}_{\Acal_\Hcal(\gamma)}/\eta)  \bigr)^b,
\end{gathered}
\end{equation}
with hidden constants that do not depend on $d$. Such approximations can be obtained, for instance, by combining $\coarsen$ and $\recompress$ with parameters adjusted as in \cite[Thm.\ 7]{BD}. 
 Assuming full knowledge of $\bbf$, $\rhs$ can be realized such that the required number of operations is bounded,  with $C>0$ independent of $d$, by
\begin{equation}\label{rhsass2}
  	C\Bigl( d \abs{\rank(\bbf_\eta)}_\infty^3 + \abs{\rank(\bbf_\eta)}_\infty \sum_{i \in \hat\cI} \#\supp_i(\bbf_\eta)  \Bigr),
\end{equation}
and we shall make this idealized assumption in the subsequent analysis. Note that these requirements greatly simplify when the corresponding right hand side $f$ is independent of the parametric variable.

In order to also compare different parametric dimensionalities $d$ in the complexity bounds, we additionally need a specific reference family of $d$-dependent problems. We introduce the following model assumptions, which we shall also consider in more detail for a concrete class of problems in \S\ref{sec:approximability}.

\begin{assumptions}
\label{ass:fullbenchmark dim}
For the quantities in Assumptions \ref{ass:fullbenchmark}, in addition let the following hold:
\begin{enumerate}[{\rm(i)}]
	\item There exist $C_1, c_1>0$ independent of $d$ such that
	\begin{multline*}
	  \max\bigl\{\norm{\pi^{(\rs)}(\bu)}_{\Acal^s}, \norm{\pi^{(1)}(\bu)}_{\Acal^s}, \ldots,\norm{\pi^{(x)}(\bbf)}_{\Acal^s}, \norm{\pi^{(1)}(\bbf)}_{\Acal^s},\ldots,\\ \bar c^{-1} ,\, \norm{\bu}_{\AH{\gamma}}, \norm{\bbf}_{\AH{\gamma}}\bigr\}\leq C_1 d^{c_1} .
	  \end{multline*}
	\item There exists $C_2>0$ independent of $d$ such that
    \[ \max\bigl\{ \bar b,\, \norm{\bA_\rs}, \norm{\bA_1}, \ldots,\norm{\boldsymbol{\alpha}_\rs}_{\ell^1},  \norm{\boldsymbol{\alpha}_1}_{\ell^1}, \ldots, \norm{\boldsymbol{\beta}_\rs}_{\ell^1}, \norm{\boldsymbol{\beta}_1}_{\ell^1}, \ldots\} \leq C_2.
    \]
    \end{enumerate}
\end{assumptions}

It needs to be emphasized that Algorithm \ref{alg:tensor_opeq_solve} does not require any knowledge on the approximability of $\bu$ stated in Assumptions \ref{ass:fullbenchmark} and \ref{ass:fullbenchmark dim}; these merely describe a model case for complexity bounds and their dependence on $d$.
Recall from Remark \ref{rem:termination} that Algorithm \ref{alg:tensor_opeq_solve} always produces $\bu_\varepsilon$ satisfying $\norm{\bu-\bu_\varepsilon}\leq \varepsilon$ in finitely many steps.

\begin{theorem}\label{fullsepresult}
Let Assumptions \ref{ass:fullbenchmark} hold, let $\alpha>0$ and let $\kappa_1,\kappa_2,\kappa_3$ in Algorithm \ref{alg:tensor_opeq_solve} be chosen as
\begin{gather*}
	\kappa_1 = \bigl( 1 + (1 +\alpha)(\sqrt{d} + \sqrt{2d-3} + \sqrt{d(2d-3)})  \bigr)^{-1}, \\
	 \kappa_2 = \sqrt{2d-3}(1+\alpha) \kappa_1,\quad \kappa_3 = \sqrt{d}(1+\sqrt{2d-3})(1+\alpha) \kappa_1.
\end{gather*}	
Then for each $\varepsilon>0$ with $\varepsilon < \varepsilon_0$, the approximation $\bu_\varepsilon$ produced by Algorithm \ref{alg:tensor_opeq_solve} satisfies
\begin{equation}\label{htrankbound}
	 \abs{\rank(\bu_\varepsilon)}_\infty 
   \leq \, \bigl( \bar c^{-1}
   \ln\bigl[ 2(\alpha \kappa_1)^{-1} \rho_{\gamma}
   \,\norm{\bu}_{\AH{\gamma}}\,\varepsilon^{-1}\bigr] \bigr)^{ \bar b}
    \lesssim  \biggl(\frac{\abs{\ln \varepsilon} + \ln d}{\bar c}\biggr)^{\bar b},
\end{equation}
as well as
\begin{equation}\label{htsuppbound}
	\quad \sum_{i\in\hat\cI} \#\supp_i(\bu_\varepsilon) \lesssim d^{1+\frac1s}\Bigl( \sum_{i\in\hat\cI} \norm{\pi^{(i)}(\bu)}_{\cA^s} \Bigr)^{\frac1s} \varepsilon^{-\frac1s}.
\end{equation}
Let in addition Assumptions \ref{ass:fullbenchmark dim} hold and let $\rhs$ satisfy \eqref{rhsass1}, \eqref{rhsass2}, then there exist $c,C>0$ such that the number of required operations is bounded by
\begin{equation}\label{htcomplexity}
	C d^{c\ln d} \abs{\ln \varepsilon}^{2\bar b} \varepsilon^{-\frac1s},
\end{equation}
where $c$ and $C$ depend on $\alpha, \rho,\omega, s$, and on the constants $C_1,c_1,C_2$ in Assumptions \ref{ass:fullbenchmark dim}.
\end{theorem}

\begin{proof}

The validity of \eqref{htrankbound} and \eqref{htsuppbound} follows by \cite[Thm.\ 7]{BD}, which can be immediately applied to the result of line \ref{alg:cddtwo_coarsen_line} in Algorithm \ref{alg:tensor_opeq_solve}. Concerning \eqref{htcomplexity}, we can apply \cite[Thm.\ 8]{BD} (with $R_i = d$ and uniform constants $\hat C^{(i)}_{\hat \alpha}$, $\hat C^{(i)}_{\hat \beta}$ and $\hat C^{(i)}_{\tilde\bA}$ in the notation used there) to obtain, for $\bw_\eta\coloneqq \apply(\bv;\eta)$,
\[
   \#\supp_i (\bw_\eta) \lesssim d^{1+s^{-1}} \eta^{-\frac1s} \Bigl(\sum_{j \in \Ical_\rs} \norm{\pi^{(i)}(\bv)}_{\Acal^s} \Bigr)^{\frac1s} , \quad
    \norm{\pi^{(i)}(\bw_\eta)}_{\Acal^s} \lesssim d^{1+s} \norm{\pi^{(i)}(\bv)}_{\Acal^s},
\]
as well as $\rank(\bw_\eta)\leq (d+1) \rank(\bv)$.
With these estimates, \eqref{htcomplexity} follows exactly as in \cite[Thm.\ 9]{BD}.
\end{proof}

For each fixed $d$, the produced solution $\bu_\varepsilon$ thus requires a number of parameters that is proportional to the optimal one for this type of approximation. Taking into account that in the operation count, the best approximation ranks of order $\Ocal(\abs{\ln \varepsilon}^{\bar b})$ enter at least quadratically due to orthogonalizations required in the algorithm, the number of operations in \eqref{htcomplexity} also scales optimally with respect to $\varepsilon$. The dependence of the constant on the parametric dimension is subexponential, since $d^{c\ln d} = e^{c (\ln d)^2}$.

\section{Spatial-parametric sparse approximation}\label{sec:sparselegendre}

We now turn to the case $\cI=\N$, that is, problems involving countably many parameters $(y_j)_{j\geq 1}$ that have decreasing influence as $j$ increases.
Here we consider problems of the type \eqref{paramdiffusion-0}, 
\begin{equation}
\label{thetaj}
	 a(y) = \bar a + \sum_{j = 1}^\infty y_j \pf_j ,
\end{equation}
under the uniform ellipticity assumption \eqref{uea} on $a$.
This variant of Algorithm \ref{alg:tensor_opeq_solve} is  similar to the scheme proposed in  \cite{Gittelson:14}, following the approach of \cite{Cohen:01,Cohen:02}.

In this section we consider a version of Algorithm \ref{alg:tensor_opeq_solve} that produces $n$-term approximations to $u\in L^2(\pdom,V)$ in terms of the wavelet-Legendre tensor product basis $\{\psi_\lambda\otimes L_\nu\}_{\lambda\in\sidx,\nu\in\pidx}$. That is, the approximation that we seek in this case is of the form \eqref{fullnterm}, that is,
\begin{equation}\label{fullsparse}
	 u \approx u_n = \sum_{(\lambda,\nu)\in \Lambda_n} \bu_{\lambda \nu} \psi_\lambda\otimes L_\nu,
\end{equation}
where we aim to identify $\Lambda_n$ which yields an error close to that of the best $n$-term approximation in this basis.

Here, $\coarsen$ performs a standard coarsening operation on a sequence, and we set $\recompress(\bv;\eta) \coloneqq \bv$ for any $\eta$.
The scheme thus reduces to the adaptive method of \cite{Cohen:02}, which has been considered for this particular type of approximation of parametric PDEs also in \cite{Gittelson:14}.
The key ingredient that remains to be described is the adaptive application of $\bA$ to representations of the form \eqref{fullsparse} based on its $s^*$-compressibility.

Let $\bu \in \Acal^s(\sidx\times\pidx)$ and let $\bA$ be $s^*$-compressible with $s^*>s$ according to \eqref{compress}. Then it follows by \cite[Prop.\ 3.8]{Cohen:01} that $\bbf \in \Acal^s$, hence we can construct $\rhs$ satisfying
\[
  \#\supp (\rhs(\eta)) \lesssim \eta^{-\frac1s} \norm{\bbf}_{\Acal^s}^{\frac1s},\quad \norm{\rhs(\eta)}_{\Acal^s} \lesssim \norm{\bbf}_{\Acal^s}.
\]
Moreover, by the standard construction of $\apply$ in \cite[Cor.\ 3.10]{Cohen:01} based on the $s^*$-compressibility of $\bA$, the results in \cite{Cohen:02} yield the following complexity bound for the present realization of Algorithm \ref{alg:tensor_opeq_solve}.

\begin{theorem}\label{cddresult}
Let $\bu\in\Acal^s$ and let $\bA$ be $s^*$-compressible with $0<s<s^*$.
	Then for any given $\varepsilon>0$, the approximation $\bu_\varepsilon$ produced by the above variant of Algorithm \ref{alg:tensor_opeq_solve} operating on approximations of the form \eqref{fullsparse} satisfies 
		\[
	   \# \supp (\bu_\varepsilon ) \lesssim \varepsilon^{-\frac1s} \norm{\bu}_{\Acal^s}^{\frac 1 s}, \quad 
	      \norm{\bu_\varepsilon}_{\Acal^s} \lesssim \norm{\bu}_{\Acal^s},
	\] 
	and the number of operations is bounded up to a multiplicative constant by 
	$	   1 + \varepsilon^{-\frac1s} \norm{\bu}_{\Acal^s}^{\frac 1 s}	   $.
\end{theorem}

We next consider the compressibility of $\bA$, which determines the range of $s$ for which Theorem \ref{cddresult} yields optimality, and a corresponding procedure $\apply$. In \S\ref{sec:approx aniso}, we consider in further detail for which values of $s$ one can indeed expect $\bu \in \Acal^s$.

\subsection{Adaptive operator application}\label{ssec:adop}

Any numerical scheme $\apply$ necessarily involves a truncation of the series \eqref{thetaj}. 
Defining for each nonnegative integer $M$  the corresponding truncation error
\beqn
\label{eM}
    e_{M} := \Bignorm{ \sum_{j > M } \bA_j \otimes \bM_j }
\eeqn 
of replacing $\bA$ by $\sum_{j=1}^M \bA_j \otimes \bM_j$, where $e_0 = \norm{\bA}$, the decay of $e_M$ describes the approximability
of $\bA$. We will be concerned with algebraic rates
\be
\label{S}
e_M\le C M^{-S}, \quad M\in \N,
\ee
where $C, S>0$ are fixed constants. Note that in particular, our further developments do not require summability of $(\norm{\pf_j}_{L^\infty})_{j\geq 1}$ as assumed, e.g., in \cite{EGSZ:14,EPS:15}.

{A first limitation to the $s^*$-compressibility of $\bA$ lies in the decay of the truncation errors \eqref{S}, which arise from replacing $\bA$ by a finite sum.
This amounts to approximating all but finitely many $\bA_j$ by zero.
A second limitation is the compressibility of
the remaining $\bA_j$, which depends on the particular expansion system $(\pf_j)_{j\in\N}$. 
As we show next, a favorable $s^*$-compressibility result for $\bA$ (almost matching the truncation error decay \eqref{S}) can be obtained when $(\pf_j)_{j\in\N}$ have multiscale structure, as summarized in the following set of assumptions.
In \S\ref{sec:aniso}, our findings are compared to previous results that hold under more generic assumptions.
}

\begin{assumptions}\label{ass:multiscale}
Let $\{\xi_\mu\}_{\mu \in \Lambda}$ be a system of compactly supported multilevel basis functions with $\operatorname{diam}({\rm supp}(\xi_\mu))\sim 2^{-|\mu|}$ and $\norm{\xi_\mu}_{L^\infty(D)} = 1$. With $(\mu_j)_{j\geq 1}$ an enumeration of $\Lambda$ by increasing level and some fixed $\alpha>0$, let
\beqn
\label{multiscaleexpansion}
\pf_j = c_{\mu_j} \xi_{\mu_j}, \quad \text{ where $c_{\mu_j} = 2^{-\alpha \abs{\mu_j}}$.}
\eeqn
To simplify notation, let $c_{\mu_0} \coloneqq 1$, $\xi_{\mu_0} \coloneqq \bar a$, and $\abs{\mu_0} \coloneqq  0$.
\end{assumptions}

Note that for what follows, it would in fact suffice to assume $c_{\mu} \sim 2^{-\alpha \abs{\mu}}$, with a constant that is uniform over $\Lambda$, but we assume equality to simplify the exposition.
Under Assumptions \ref{ass:multiscale}, 
\begin{equation}
\label{multiscale eMbound}
  e_M \leq \sup_{y\in Y} \Bignorm{\sum_{j>M} y_j \theta_j }_{L^\infty(D)} 
   \leq \sum_{\ell \geq \max\{ \abs{\mu_j} \colon j \leq M\}} 2^{-\alpha\ell} 
    \lesssim 2^{-\alpha \max\{ \abs{\mu_j} \colon j \leq M\}} \lesssim M^{-\frac{\alpha}{m}},
\end{equation}
and we thus obtain \eqref{S} with $S = \alpha/m$.

We now give a new result for the compressibility of $\bA$ arising from a wavelet-type parametrization as in \eref{multiscaleexpansion}. 
As we shall see, making use of a multilevel structure in the parametrization, one can obtain substantially better compressibility of $\bA$ than under the more generic assumptions used in \cite{Gittelson:14}.

The result is based on compressibility properties of the corresponding matrices $\bA_j$. These are analyzed in Appendix \ref{sec:opapprox}, where the assumptions of the following Proposition are established under conditions on the regularity of $\xi_\mu$ and of the spatial wavelets.

\begin{proposition}\label{prop:sparsecompr}
Let $\{ \pf_j \}_{j\geq 1}$ and $c_{\mu_j}$ be as in Assumptions \ref{ass:multiscale}, and assume that 
there exist a $\tau > \frac{\alpha}m$ and matrices  
$\bA_{j,n}$, $n\in \N_0$, {where $\bA_{j,0} = 0$,} with the following properties: 
\begin{enumerate}[{\rm(i)}]
\item
One has
\beqn
\label{compress2}
\|\bA_j - \bA_{j,n}\|\lesssim c_{\mu_j}  2^{-\tau n},\quad n\in\N_0,
\eeqn
where the hidden constant is independent of $j,n$.
\item  The number of nonvanishing entries in each column of $\bA_{n,j}$ does not exceed
a uniform constant multiple of $\bigl(1+|\mu_j|^q\bigr)2^n$, for some $q\geq 1$.
\end{enumerate}
Then $\bA$ is $s^*$-compressible with
\begin{equation}\label{sparsecomprlimit}
s^* = \frac{\alpha}{m} \frac{2\tau}{1+2\tau}.
\end{equation}
\end{proposition}

Specifically, it is shown in Appendix \ref{sec:opapprox} that the above assumptions can be realized for arbitrarily large $\tau$
by choosing the functions $\xi_\mu$ and the spatial wavelets sufficiently smooth, the latter having sufficiently many vanishing moments.
By the above result, we thus obtain that $\bA$ is then $s^*$-compressible where $s^*<\alpha/m$ comes as close to $\alpha/m$ as one wishes
when $\tau$ is suitably large. As discussed in further detail in \S\ref{sec:approximability}, this means that the $n$-term approximability of $\bu$ can be essentially fully exploited by the adaptive scheme.

\begin{proof}
We construct approximations $\bA_{\mathbf{n}}$ of $\bA$ by choosing sequences $\mathbf{n} = (n_{j})_{j\geq 0}$ of bounded support and defining $\bA_{\mathbf{n}}\colon \spl{2}(\sidx\times\pidx)\to \spl{2}(\sidx\times\pidx)$ by
\begin{equation}\label{defAn}
	\bA_\mathbf{n} \coloneqq \sum_{j\geq 0} \bA_{j, n_{j}} \otimes \bM_j.
\end{equation}
Our aim is to find such $\mathbf{n}^J$ such that the corresponding $\bA^J := \bA_{\mathbf{n}^J}$ satisfy
\begin{equation}
	\norm{\bA - \bA^J} \lesssim J^{-2}2^{-s J},\quad J\in\N,
\end{equation}
with $s<s^*$ and $s^*$ as in the assertion, and such that the number of nonzero entries in the each row and column of $\bA^{J}$
is bounded by a fixed constant multiple of $J^{-2}2^J$.

We take $L\in\N$ arbitrary but fixed. Recall that we assume $\mu_j$ to be ordered by increasing level, that is, $\abs{\mu_{j+1}} \geq \abs{\mu_j}$. We now consider $(n_j)_{j\geq 0}$ satisfying $n_j = 0$ for $j >  M_L \coloneqq \ceil{L^{2m/\alpha}2^{mL}}$, {so that $\bA_{j,n_j} = 0$ for $j > M_L$.} Since  $e_{M_L} \lesssim L^{-2} 2^{-\alpha L}$ by \eqref{S} and \eref{multiscale eMbound}, we obtain
\begin{equation}\label{operr_triineq}
	\norm{\bA - \bA_\mathbf{n}} \lesssim  \biggnorm{ \sum_{j=0}^{M_L} {\bigl( \bA_j - \bA_{j,n_j} \bigr)} \otimes \bM_j }  + L^{-2}  2^{-\alpha L}.
\end{equation}
Within each level $\ell \geq 0$, that is, for each $\mu$ with $\abs{\mu}=\ell$, there are only finitely many $\mu'$ with $\abs{\mu'}=\ell$ such that $\supp \xi_\mu \cap \supp\xi_{\mu'} \neq \emptyset$. Since the images of $\bA_j$ corresponding to $\xi_{\mu_j}$ with disjoint support are orthogonal, we obtain
\begin{equation}
	\biggnorm{ \sum_{j=0}^{M_L} {\bigl( \bA_j - \bA_{j,n_j} \bigr) } \otimes \bM_j  }   	
	\lesssim  \sum_{\ell=0}^{\ceil{L+ \frac{2}{\alpha}\log_2 L}} \Bigl( \sum_{j\colon\abs{\mu_j} = \ell}  \norm{\bA_j - \bA_{j,n_j}}^2  \Bigr)^{\frac12} ,
\end{equation}
where the constant depends on the maximum number of $\xi_\mu$ of overlapping support on each level.
Taking 
\[
n_j = n_\ell = \Bigl\lceil{ {\frac{m \ell}{2\tau} + \frac{\alpha}\tau \Bigl(L+ \frac{2}{\alpha}\log_2 L - \ell\Bigr)} + \frac1\tau \log_2(1+\ell)^2 } \Bigr\rceil
\]
for $\mu_j$ of level $\ell$ and recalling that for such $j$ we have $|c_{\mu_j}|=2^{-\alpha\ell}$  gives
\begin{equation}
\label{errL}
	 \sum_{\ell=0}^{\ceil{L+ \frac{2}{\alpha}\log_2 L}} \Bigl( \sum_{j\colon\abs{\mu_j} = \ell}  \norm{\bA_j - \bA_{j,n_j}}^2  \Bigr)^{\frac12}
	  \lesssim
	  \sum_{\ell = 0}^{\ceil{L+ \frac{2}{\alpha}\log_2 L}} 2^{\frac{m}2 \ell} 2^{-\alpha \ell} 2^{-\tau n_\ell} \lesssim L^{-2} 2^{-\alpha L}.
\end{equation} 
Let $N_L$ be the resulting maximum number of entries per row and column in $\bA_{\mathbf{n}}$, then 
\begin{align}
	N_L &\lesssim \sum_{\ell = 0}^{\ceil{L+ \frac{2}{\alpha}\log_2 L}} (1+\ell^q)2^{m\ell} 2^{n_\ell} 
	\lesssim
	 	L^{\frac{2}{\alpha}}2^{\frac\alpha{\tau} L }  \sum_{\ell = 0}^{\ceil{L+ \frac{2}{\alpha}\log_2 L}} (1+\ell)^{q+\frac2\tau} 2^{(1 + \frac1{2\tau}) m \ell - \frac\alpha{\tau} \ell} \nonumber\\
	& \lesssim L^{q + \frac{2(1+m)}\alpha} 2^{\frac{1+2\tau}{2\tau} m L},
\end{align}
where we have used $\tau > \alpha/m$.

We now fix $s>0$ with $s < t \coloneqq \frac{\alpha}{m} \frac{2\tau}{1+2\tau}$ and
  take $J \coloneqq \ceil{\frac{t}{s} \frac{1+2\tau}{2\tau} m L} = \ceil{\frac\alpha{s} L}$ and $\mathbf{n}^J \coloneqq \mathbf{n}$.
Since then $N_L \lesssim J^{q+ \frac{2(1+m)}\alpha}2^{\frac{s}{t} J}$ we  see that  $N_L\lesssim J^{-2}2^J$
with a constant that depends on $\alpha, m$ and increases when $s$ approaches $t$.
It immediately follows from \eref{errL} that 
\begin{equation}
	 \norm{\bA - \bA^J} \lesssim J^{-2}2^{- s J} 
\end{equation}
with a constant depending on $ m$.
Thus $\bA$ is $s^*$-compressible with $s^* = t$.
\end{proof}

\subsection{Coefficient expansions}\label{sec:aniso}

In our compressibility result Proposition \ref{prop:sparsecompr} for $\bA$, we have made use of the multiscale structure of the expansion functions $\theta_j$. Let us now briefly compare this to previous results for globally supported $\theta_j$ as they arise in Karhunen-Lo\`eve expansions.
In fact, for certain problems one has equivalent expansions in either globally supported or wavelet-type $\theta_j$. This is demonstrated, for instance, in \cite{BCM:2016} for lognormal diffusion coefficients with Gaussian random fields of Mat\'ern covariance.

In order to illustrate the basic issues in approximation $\bA$ in the case of typical globally supported $\theta_j$, we consider the following spatially one-dimensional setting with $D=]0,1[$ as in \cite{Gittelson:14}: for a monotonically decreasing positive sequence $(c_j)_{j\in\N}$ with $\sum_{j\geq 1} c_j \leq \frac12$, take $\pf_j = c_j \sin(j\pi \cdot)$, so that
\begin{equation}
\label{cosexp}
   a(y) = 1 + \sum_{ j \geq 1 } y_j c_j \sin(j\pi \cdot).
\end{equation}
This model is representative in that such increasingly oscillatory $\pf_j$ as $j\to\infty$ also arise in more general Karhunen-Lo\`eve expansions.

As a concrete example, with $\beta> 1$, let $c_j := \delta j^{-\beta}$ with $\delta>0$ sufficiently small. Then $\norm{\bA_j} \sim c_j$, from which we only obtain
\[
  e_M \lesssim M^{-\beta + 1},
\]
and therefore $S = \beta - 1$. 
As shown in \cite{Gittelson:14}, taking the compression of the individual $\bA_j$ into account one obtains $s^*$-compressibility of $\bA$ with $s^* = \frac{1}{2}(\beta - 1) = \frac12 S$.
In this case, instead of conditions (i) and (ii) in Proposition \ref{prop:sparsecompr} one has $ \norm{\bA_j -\bA_{j,n}} \lesssim j^{-(\alpha+\frac12)} 2^{-\gamma n}$ with $\Ocal( j ( 1 + \log_2 j) 2^n)$ entries per row and column. We comment further in Remark \ref{rem:badcompr} on how this leads to the limitation to $s^* = \frac{1}{2}(\beta - 1)$.

\section{Low-rank approximation}\label{sec:lowrank}

We now turn to an adaptive method for finding low-rank approximations of the form \eqref{fulllr}, based on the Hilbert-Schmidt decomposition of $\bu$.
These approximations are of the form
\be\label{fulllrcoeff}
 \bu \approx \sum_{k=1}^n \bu^\rs_{k,\lambda} \otimes  \bu^\rp_{k,\nu}.
\ee
with finitely supported vectors $\bu^\rs_{k}$, $\bu^\rp_{k}$, $k=1,\ldots,n$. 
 As in the scheme considered in \S\ref{sec:iso}, adaptivity in rank and in the basis expansions is intertwined by iteratively improving low-rank expansions of varying ranks, while at the same time identifying finitely supported approximations in $\spl{2}(\sidx)$ and $\spl{2}(\pidx)$, both based on approximate residual evaluations.

In principle, the results of \S \ref{sec:iso} concerning a full separation of variables based on hierarchical tensor formats could be applied with any finite truncation dimension $d$. However, assuming \eqref{S}, a total error of order $\varepsilon$ requires $d(\varepsilon) \sim \varepsilon^{-1/S}$. As a consequence, due to the $d$-dependent quasi-optimality \eqref{recompress quasiopt} of the hierarchical SVD truncation, we can only obtain a highly suboptimal complexity bound in \eqref{htcomplexity} for the hierarchical format.

Concerning low-rank decompositions, we therefore concentrate here on a more basic case, namely a separation of spatial and parametric variables 
as in \eqref{fulllrcoeff}. Since this separation also occurs in any hierarchical representation, the resulting Hilbert-Schmidt rank provides a lower bound for the  hierarchical ranks that are required in a hierarchical format involving further matricizations.

The efficiency of the obtained low-rank approximations is measured against the singular value decomposition of the Hilbert-Schmidt operator $\spl{2}(\pidx)\to \spl{2}(\sidx)$ induced by $\bu$,
\begin{equation}\label{hsapprox}
  \bu = \sum_{k=1}^\infty \sigma_k \bU^{(\rs)}_k \otimes \bU^{(\rp)}_k \,,
\end{equation}
where $\sigma_k \geq 0$, $\{ \bU^{(\rs)}_k \}$, $\{ \bU^{(\rp)}_k \}$ are orthonormal in $\spl{2}(\sidx)$ and $\spl{2}(\pidx)$, respectively, and
\begin{equation}\label{hsoptcoeff}
  \Bignorm{ \bu - \sum_{k=1}^r \sigma_k \bU^{(\rs)}_k \otimes \bU^{(\rp)}_k}^2 =  \sum_{k>r} \sigma_k^2 = \min_{\rank(\bw)\leq r} \norm{\bu-\bw}^2.
\end{equation}
Ideally, the ranks of computed approximations should be comparable to the minimum $r$ for achieving the same error in \eqref{hsoptcoeff}.

Moreover, we quantify in terms of $\#\left( \bigcup_{k=1}^n \supp \bu^i_{k} \right)$, $i=\rs,\rp$ the number of nonzero coefficients in \eqref{fulllrcoeff}. The reasons for not considering each individual $\# \supp  \bu^i_{k}$ separately are mainly algorithmic: since the numerical methods require orthogonalizations of the sets $( \bu^i_{k} )_{k=1,\ldots,n}$, their complexity is determined by the unions of the respective supports.
To understand the joint approximability of the infinite vectors $\bU^{(i)}_k$, $i=\rs,\rp$, in \eqref{hsapprox} serving as our reference point, we consider the particular {contractions} defined, for   $\bv \in \spl{2}(\sidx \times \pidx)$, by
$\pi^{(\rs)}(\bv) = (\pi^{(\rs)}_\lambda(\bv))_{\lambda\in\sidx}$, $\pi^{(\rp)}(\bv) = (\pi^{(\rp)}_\nu(\bv))_{\nu\in\pidx}$
with
\begin{equation}
	\label{contrxy}
\pi^{(\rs)}_\lambda(\bv) := \Big(\sum_{\nu\in \pidx}|\bv_{\lambda,\nu}|^2\Big)^{1/2},\quad \pi^{(\rp)}_\nu(\bv) := \Big(\sum_{\lambda\in \sidx}|\bv_{\lambda,\nu}|^2\Big)^{1/2}.
\end{equation} 
Note that $\pi^{(\rp)}_\nu(\bu)$ is uniformly proportional to the norm of the corresponding Legendre coefficient of $u$, that is, $\pi^{(\rp)}_\nu(\bu) \sim \norm{u_\nu}_V$.

Let $(\lambda^*_{k})_{k \in \N}$ and $(\nu^*_{k})_{k \in \N}$ be such that $(\pi^{(\rs)}_{\lambda^*_{k}}(\bv))_{k\in\N}$ and $(\pi^{(\rp)}_{\nu^*_{k}}(\bv))_{k\in\N}$ are nonincreasing, respectively.
Then the singular values $\sigma_k(\bv)$ of $\bv$ satisfy
\beqn
\label{svpibound}
    \sigma_k(\bv) \, \leq \, \pi^{(\rs)}_{\lambda^*_k}(\bv),\,\pi^{(\rp)}_{\nu^*_k}(\bv), \quad k \in \N.
\eeqn
 
In view of our results for Example \ref{ex2} (and the further numerical experiments of Example \ref{ex:cfdecay}), we cannot generally expect faster than algebraic decay of singular values, which we quantify in terms of classes $\AH{\gamma}$ specialized to tensors of order two and to the specific sequence $\gamma(k):= (1+k)^{\bar s}$. This yields the approximation classes
\[
 \Sb :=\Bigl\{\bv\in \spl{2}(\sidx\times\pidx): \sup_{k\in\N}(1+k)^{\bar s}\Big(\sum_{j>k}\sigma_k(\bv)^2\Big)^{1/2} =: \|\bv\|_{\Sb}<\infty \Bigr\}.
\]
The approximate sparsity of the sequences $\pi^{(\rs)}(\bv)$, $\pi^{(\rp)}(\bv)$ is measured in terms of the largest $s_\rs, s_\rp>0$ such that $\pi^{(\rs)}(\bv) \in \Acal^{s_\rs}(\sidx)$, $\pi^{(\rp)}(\bv) \in \Acal^{s_\rp}(\pidx)$ according to \eqref{Asdef}.

For the low-rank approximation, the routines $\recompress$ and $\coarsen$ used in Algorithm \ref{alg:tensor_opeq_solve}
are based on the specialization to tensors of order two of the routines described in the previous section.
$\recompress(\bv;\eta)$ is a numerical realization of $\hatPsvd{\eta}(\bv)$, which we define as the operator producing the best low-rank approximation of $\bv$ with error at most $\eta$ with respect to $\norm{\cdot}$, obtained by truncating the singular value decomposition of its argument. 
 
The routine $\coarsen(\bv;\eta)$ is constructed as in \S\ref{sec:iso} based on the contractions $\pi^{(\rs)}(\bv)$, $\pi^{(\rp)}(\bv)$ defined as in \eqref{contrxy}.
The following result differs from \cite[Theorem 7]{BD}, which is formulated for general hierarchical tensors, in that we now consider differing sparsity classes for the contractions
$\pi^{(i)}$, $i=\rs,\rp$. In view of the preceding discussion, it is reasonable to assume possibly different but algebraic decay for both contractions.

\begin{theorem}
\label{lmm:combined_coarsening}
Let $\mathbf{u}, \mathbf{v} \in \spl{2}(\sidx\times\pidx)$ with
$\norm{\mathbf{u}-\mathbf{v}} \leq \eta$. 
Then for
\begin{equation}
\label{weta}
\mathbf{w}_{\eta} := \hatCctr{2^{3/2}(1+\alpha)\eta} \bigl(\hatPsvd{
(1+\alpha)\eta} (\mathbf{v}) \bigr) \,,
\end{equation}
we have
\begin{equation}
\label{eq:combinedcoarsen_errest}
  \norm{\bu - \bw_\eta} \leq \bigl( 2+\alpha +  2^{3/2} (1+\alpha)\bigr) \eta .
\end{equation}
Moreover, when $\mathbf{u}\in\Sb$, $\pi^{(i)}(\mathbf{u})  \in \cA^{s_i}$,  $i=\rs,\rp$, we have
\begin{equation}\label{eq:combinedcoarsen_rankest}
 \abs{\rank(\bw_\eta)}_\infty \leq 2
 \bigl(\alpha^{-1} \norm{\bu}_{\Sb} \bigr)^{\frac1{\bar s}} \eta^{-\frac1{\bar s}}\,,\qquad \norm{\bw_\eta}_{\Sb}
 \leq 2 (1 + \alpha^{-1}) \norm{\bu}_{\Sb} ,
\end{equation}
and
\begin{equation}\label{eq:combinedcoarsen_suppest}
\begin{aligned}
 \#\supp_\rs (\mathbf{w}_\eta)  +  \#\supp_\rp (\mathbf{w}_\eta)  &\leq 
 2+  \biggl( \frac{2\|\pi^{(\rs)}(\bu)\|_{\Acal^{s_\rs}}}{\alpha\eta}\biggr)^{\frac1{s_\rs}} +
 \biggl( \frac{2\|\pi^{(\rp)}(\bu)\|_{\Acal^{s_\rp}}}{\alpha\eta}\biggr)^{\frac1{s_\rp}}
   \\
  \norm{\pi^{(i)}(\bw_\eta)}_{\Acal^{s_i}} 
  &\leq C 
     \norm{ \pi^{(i)}(\bu)}_{\Acal^{s_i}} , \quad i=\rs,\rp, 
 \end{aligned}
\end{equation}
where $C$ depends on $\alpha$ and $s_i$, $i=\rs,\rp$.
\end{theorem}

The estimates \eqref{weta}, \eqref{eq:combinedcoarsen_errest} have been already shown in \cite{BD}.
The only deviation concerns the stability estimate \eqref{eq:combinedcoarsen_suppest}, which we prove in Appendix \ref{app:thm3}.

To apply Algorithm \ref{alg:tensor_opeq_solve} it remains to specify the approximate application of $\bA$ by the procedure $\apply$ to representations of the form \eqref{hsapprox}. As part of this procedure, we shall also use a modified routine $\coarsen_\rp$ which operates only on the second tensor mode and leaves $\supp_\rs$ unchanged. For this routine, we shall only use the simpler statement that for any $\bv\in\spl{2}(\sidx\times\pidx)$ with $\pi^{(\rp)}(\bv) \in \Acal^{s_\rp}(\pidx)$, $\bv_\rp \coloneqq \coarsen_\rp(\bv;\eta)$ satisfies
\[  \#\supp_\rp (\bv_\rp) \lesssim \eta^{-\frac1{s_\rp}} \norm{\pi^{(\rp)}(\bv)}_{\Acal^{s_\rp}}^{\frac1{s_\rp}}, \quad  \norm{\pi^{(\rp)}(\bv_\rp)}_{\Acal^{s_\rp}} \lesssim \norm{\pi^{(\rp)}(\bv)}_{\Acal^{s_\rp}}. \]

\subsection{Adaptive operator application}\label{ssec:adaptop}

We now describe a specification of the more generic routine $\apply$ used in \cite{BD} that is tailored  to exploit anisotropy in the parametrizations of parametric operators.
For  any given $\eta>0$ and finitely supported $\bv$ we aim to construct $\bw_\eta$ such that $\norm{\bA \bv - \bw_\eta} \leq \eta$.
We follow here the general strategy of combining a priori knowledge on $\bA$ with a posteriori information on $\bv$, which is given in terms of a
suitable decomposition of $\bv$. The routine $\apply$ is structured as follows:\medskip

\noindent
(S1)  {\em Preprocessing and decomposing the input:}
We first apply a preprocessing step to the finitely supported  input $\bv$ 
that consists of applications of $\recompress$ and $\coarsen_\rp$.
We choose for a given $\eta >0$ the tolerances 
of order $\eta$ in such a way that the resulting $\bv_\eta$ satisfies  
\be
\label{etahalf}
\|\bv - \bv_\eta\|\le \frac\eta{2\|\bA\|}.
\ee
As a consequence, for any positive $s_\rp$, $\bar s$ we have
\be
\label{vcoarsened}
\rank(\bv_\eta)\lesssim \eta^{-\frac{1}{\bar s}}\|\bv\|_{\Sb}^{\frac{1}{\bar s}},\quad  \|\bv_\eta\|_{\Sb}\lesssim \|\bv\|_{\Sb},
\ee
and
\be
\label{supps}
   \#\supp_\rp (\mathbf{v}_\eta )  \lesssim 
  {\eta}^{-\frac1{s_\rp}} \norm{\pi^{(\rp)}(\bv)}^{\frac1{s_\rp}}_{\Acal^{s_\rp}}, \quad
  \norm{\pi^{(\rp)}(\bv_\eta)}_{\Acal^{s_\rp}} 
  \lesssim 
     \norm{ \pi^{(\rp)}( \bv)}_{\Acal^{s_\rp}} .
\end{equation}
We then have the SVD of $\bv_\eta$ at hand,
\beqn
\label{SVD}
 \bv_\eta  = \sum_{k =  1}^K \sigma_k \bU^{(\rs)}_k \otimes \bU^{(\rp)}_k\,,
\eeqn
and set $K_p = \{ 2^p ,\ldots, \min\{K,2^{p+1} - 1\}\}$, for $p = 0,1,\ldots$, $p\le \log_2 K$.
Furthermore,  for $q=0,1,\ldots$, let $\hat\Lambda^{(\rp)}_{q}$ be the support of the best $2^q$-term approximation of $\pi^{(\rp)}(\bv_\eta)$. We set $\Lambda^{(\rp)}_{0} := \hat\Lambda^{(\rp)}_{0}$ and $\Lambda^{(\rp)}_{q} :=\hat \Lambda^{(\rp)}_{q} \setminus \hat\Lambda^{(\rp)}_{q-1}$ for $q>0$.
With this subdivision of  $\supp_\rp(\bv_\eta)$, we now define
\beqn
\label{vpq}
   \bv_{[p,q]} :=  \Restr{ \sidx \times \Lambda^{(\rp)}_{q}} \sum_{k\in K_p}  \sigma_k \bU^{(\rs)}_k \otimes \bU^{(\rp)}_k  = \sum_{k\in K_p} \sigma_k \bU^{(\rs)}_k\otimes \Restr{\Lambda^{(\rp)}_{q}}
   \bU^{(\rp)}_k  \,,
\eeqn
and obtain
\begin{equation}
\label{decoAv}
    \bA \bv_\eta  = \sum_{p, q \geq 0}   \sum_{j=0}^\infty (\bA_j \otimes \bM_j) \bv_{[p,q]} 
      =  \sum_{p, q \geq 0} \sum_{j=0}^\infty  \sum_{k\in K_p} \sigma_k \,\bigl( \bA_j\bU^{(\rs)}_k\bigr) \otimes  \bigl(\bM_j \Restr{\Lambda^{(\rp)}_{q}}  \bU^{(\rp)}_k \bigr) \,.
\end{equation}

\noindent
(S2) {\em Adaptive operator truncation:}
To construct an approximation $\bw_\eta$ of $\bA \bv_\eta$ based on this decomposition, 
we truncate the summations over $j$ for each $p,q$ at some index $M_{p,q}\in \N$,   to be determined later,
and then replace the remaining terms $\bA_j$ by compressed versions, again depending on the respective $p,q$. 
With $e_M$ defined for nonnegative integer $M$ as in \eqref{eM}, for any given choice of $M_{p,q}$ we have 
\beqn
\label{est1}
   \Bignorm{\bA \bv   - \sum_{p, q \geq 0}  \sum_{j=0}^{M_{p,q}} (\bA_{j} \otimes \bM_j) \bv_{[p,q]} } \leq 
      \sum_{p, q \geq 0 }   e_{M_{p,q}} \norm{\bv_{[p,q]}}. 
\eeqn
We now choose the $M_{p,q}= M_{p,q}(\eta)$ such that
\beqn
\label{est2}
     \sum_{p, q \geq 0 }   e_{M_{p,q}} \norm{\bv_{[p,q]}} \leq \frac{\eta}{4}.
\eeqn
This can be done by choosing positive weights $\alpha_{p,q}$ such that $\sum_{p,q}\alpha_{p,q}=1$,
computing $\norm{\bv_{[p,q]}}$, 
 and adjusting the $M_{p,q}$ so as to guarantee that 
\beqn
\label{etapq0}
e_{M_{p,q}}\|\bv_{[p,q]}\|\le \eta_{p,q} := \alpha_{p,q}\eta/4.
\eeqn
We will give an a priori choice for $M_{p,q}$ in \eqref{Mpqest} below, but one may as well use, e.g., the Greedy scheme proposed in \cite{Gittelson:13} for selecting these values.\medskip

\noindent
(S3) {\em Adaptive application of the spatial components $\bA_j$:}
Next, in order to realize an approximate application of the (generally) infinite matrices $\bA_j$ to
$\bU^{(\rs)}_k$ in \eqref{decoAv}  we replace $\bA_j \bv_{[p,q]}$ by  by an approximation $\tilde\bA_{j,p,q} \bv_{[p,q]} $ using   \eqref{compress}  so as to satisfy
\be
\label{etapq}
   \biggnorm{ \sum_{j=0}^{M_{p,q}} (\bA_j -  \tilde\bA_{j,p,q})\otimes \bM_j \, \bv_{[p,q]} } \leq \eta_{p,q}.
\ee
The approximate operators $\tilde\bA_{j,p,q}$ will be specified later.
The sought approximation of $\bA\bv$ can now be obtained as
\beqn
\label{weta assembly} 
    \bw_\eta  := \sum_{p,q \geq 0}   \sum_{j=0}^{M_{p,q}} (\tilde\bA_{j,p,q} \otimes \bM_j) \bv_{[p,q]} ,
\eeqn
which by the above construction satisfies the computable error bound
\begin{equation}
\label{errorbound}
  \norm{\bA \bv_\eta - \bw_\eta} \leq \sum_{p,q \geq 0 } \bigl( e_{M_{p,q}} \norm{\bv_{[p,q]}} +  \eta_{p,q}  \bigr) \leq \eta/2  \,,
\end{equation}
so that in summary $\|\bA \bv - \bw_\eta\|\le \eta$.

In summary, the above adaptive approximation of $\bA$ to a given finitely supported $\bv$ involves the following steps:\\

\begin{samepage}
\noindent
\hrule\vspace{4pt}
$\apply: \bv \to \bw_\eta$, with $\bv$ given by its SVD\vspace{2pt}
\hrule\vspace{-6pt}
\begin{itemize}
\item[(S1):] compute $\bv_\eta \coloneqq \coarsen_\rp(\recompress(\bv; \eta/4\norm{\bA}); \eta/4\norm{\bA})$ and (quasi-)sort\footnote{As usual, to warrant a linear scaling, instead of exact ordering it suffices to perform approximate sorting into buckets according to some fixed exponential decay.} the entries of $\pi^{(\rp)}(\bv_\eta)$ to obtain  
  the sets $\Lambda^{(\rp)}_{q}$;

\item[(S2):] compute the quantities $\norm{\bv_{p,q}}$ and determine the truncation values $M_{p,q}= M_{p,q}(\eta)$;

\item[(S3):] compute the quantities $\bigl(\pi^{(\rs)}_\nu(\bv_{[p,q]})\bigr)_{\nu\in\sidx}$ and use these to obtain the compressed matrices $\tilde\bA_{j,p,q}$, using \eqref{vpq} in the assembly step \eqref{weta assembly}. 
\end{itemize}
\hrule
\end{samepage}

\subsection{Complexity analysis}

To quantify the complexity of computing $\bw_\eta$ in \eqref{weta assembly} we need to specify the properties of the operator $A(y)$ as well as the
sparsity properties of the input.
In view of  our preceding discussion, in the scenario of primary interest, the singular values of the solution $\bu$ as well as
the best $n$-term approximations of the contractions $\pi^{(i)}(\bu)$, $i\in \{\rs, \rp\}$, exhibit  algebraic decay rates.
As before, these rates are denoted by $\bar s$ and $s_\rs$, $s_\rp$, respectively. 

As indicated earlier, the complexity 
  of the above scheme depends, in particular, on the operator approximability by truncation. We
  adhere to the natural assumption  that $e_M \le CM^{-S}$ for some positive $S$, see \eref{S}.
In the subsequent discussion, we assume $S> s_i$, $i\in \{\rs,\rp\}$.
As discussed in detail in \S\ref{sec:approx aniso}, this holds true for the expansion model of Assumptions \ref{ass:multiscale} with $S = \frac{\alpha}{m}$.
 We  gather next the properties upon which the complexity analysis 
will be based.

\begin{assumptions}
\label{ass:benchmark}
The solution $\bu$   to \eqref{equiv} and the matrix $\bA$ have the following properties:
\begin{enumerate}[{\rm(i)}]
\item
\label{piu}
One has $\pi^{(i)}(\bu), \pi^{(i)}(\bbf) \in \Acal^{s_i}$, $i=\rs,\rp$, with $s_\rs, s_\rp >0$.
\item
\label{sk}
$\bu, \bbf \in \Sb$ for some $\bar s  \ge s_\rs, s_\rp$.
\item
\label{MS}
There exists a constant $C$ such that
$e_M \le CM^{-S}$, $M\in \N$, where $e_M$ is defined by \eqref{eM} and
\be
\label{Ssi}
S\ge \bar s , s_\rp.
\ee
\item
\label{Am}
The representations $\bA_j$, $j\in \N$, satisfy the assumptions of Proposition \ref{prop:sparsecompr}, where $\tau$ satisfies
\be
\label{tau}
\frac{2\tau}{1+2\tau} \frac\alpha{m} = 
\frac{2\tau}{1+2\tau} S > s_\rs.
\ee
\end{enumerate}
\end{assumptions}

Under these condition, one can construct a routine $\rhs$ that satisfies, for sufficiently small $\eta>0$ and $\bbf_\eta \coloneqq \rhs(\eta)$,
\begin{gather*}	
\#\supp_i (\bbf_\eta) \lesssim \eta^{-\frac1{s_i}} \norm{\pi^{(i)}(\bbf)}_{\cA^{s_i}},
	 \quad \norm{\pi^{(i)}(\bbf_\eta)}_{\cA^{s_i}} \lesssim \norm{\pi^{(i)}(\bbf)}_{\cA^{s_i}}, \quad i \in \{\rs,\rp \},
	 \\ \rank(\bbf_\eta) \lesssim  \eta^{-\frac1{\bar s}} \norm{\bbf}_{\Sb}^{\frac1{\bar s}} , \quad \norm{\bbf_\eta}_{\Sb} \lesssim \norm{\bbf}_{\Sb}.
\end{gather*}
Assuming full knowledge of $\bbf$, one can also realize $\rhs$ using $\Ocal\bigl(\eta^{-\frac1{\bar s} - \frac{1}{\min\{ s_\rs, s_\rp \}}}\bigr)$ operations, and we shall make this idealized assumption in what follows. As in \S\ref{sec:iso}, the requirements on $\rhs$ simplify substantially when the corresponding right hand side $f$ is independent of the parametric variable.

The main result of this section states that up to a logarithmic factor the sparsity properties of the input are
preserved by the output of $\apply$.

\begin{theorem}
\label{thmaincompl}
Suppose that the properties listed under Assumptions \ref{ass:benchmark} hold. Then, given any finitely supported
input $\bv\in \spl{2}({\cal S}\times {\cal F})$, the output $\bw_\eta$ produced
by the procedure $\apply$, based on the steps {\rm (S1)--(S3)}, satisfies
\be
\label{accuracy}
\|\bA \bv - \bw_\eta\|\le \eta.
\ee
Moreover, with some $b\le 2+ \frac{4}{s_\rs}$ one has
\be
\label{rankbars}
\rank(\bw_\eta)\lesssim \eta^{-\frac{1}{\bar s}}\| \bv\|_{\Sb}^{\frac{1}{\bar s}} (1+\abs{\log\eta})^b 
, \quad 
\|\bw_\eta\|_{\Sb}\lesssim  \|\bv\|_{\Sb} (1+\abs{\log\eta})^{\bar s b },
\ee
and
\begin{equation}
\label{rp}
\begin{gathered}
\#\supp_\rp(\bw_\eta)\lesssim \eta^{-\frac{1}{s_\rp}} \|\pi^{(\rp)}(\bv)\|_{\Acal^{s_\rp}}^{\frac{1}{s_\rp}}(1+\abs{\log \eta})^{b}, \\
\|\pi^{(\rp)}(\bw_\eta)\|_{\Acal^{s_\rp}} \lesssim \|\bv\|_{\Acal^{s_\rp}} (1+\abs{\log \eta})^{s_\rp b},
\end{gathered}
\end{equation}
as well as  
\be
\label{sstab}
\begin{gathered}
\#\supp_\rs(\bw_\eta)\lesssim \eta^{-\frac 1{s_\rs}}\|\pi^{(\rs)}(\bv)\|_{\Acal^{s_\rs}}^{\frac 1{s_\rs}} (1+\abs{\log \eta})^{b}, \\
 \|\pi^{(\rs)}(\bw_\eta)\|_{\Acal^{s_\rs}} \lesssim \|\pi^{(\rs)}(\bv)\|_{\Acal^{s_\rs}} (1+\abs{\log \eta})^{s_\rs b},
\end{gathered}
\ee
where the constants depend also on $s_i$, on $\abs{\log \|\pi^{(i)}(\bv)\|_{\Acal^{s_i}}}$, $i\in \{\rs, \rp\}$, and on $\tau$ in Assumption \ref{ass:benchmark}.
\end{theorem}

\begin{proof}
The error bound \eref{accuracy} is implied by the construction. 
As for the remaining claims,
to assess the complexity of computing $\bw_\eta$, given by \eqref{weta assembly}, we estimate first $M_{p,q}=M_{p,q}(\eta)$ in terms of $\eta$.
To obtain a priori bounds for the $M_{p,q}$, we use Assumptions \ref{ass:benchmark}\eqref{piu} and \eqref{sk}
to conclude that
\beqn
\label{vpqest}
     \norm{\bv_{[p,q]}} \leq 2^{-s_\rp q}\norm{\pi^{(\rp)}(\bv)} _{\mathcal{A}^{s_\rp}} \,, \qquad   \norm{\bv_{[p,q]}} \leq 2^{-\bar s  p}   \norm{ \bv }_{\Sb}  \,.
\eeqn
Then Assumption \ref{ass:benchmark}\eqref{MS} and \eqref{vpqest} yield the sufficient conditions
\beqn
\label{Mpqest}
M_{p,q}= M_{p,q}(\eta) \ge  \Biggl( \frac{ 4 C \min \{ 2^{-\bar s p}  \norm{ \bv  }_{\Sb}  , 2^{-s_\rp q}  \norm{ \pi^{(\rp)}(\bv) }_{\mathcal{A}^{s_\rp}} \}}{ \alpha_{p,q} \, \eta } \Biggr)^{\frac{1}{S}} .
\eeqn
From \eqref{weta assembly} and the decomposition \eqref{decoAv}
we  see that
\beqn
\label{rankw}
  \rank(\bw_\eta) \leq    \sum_{p,q\geq 0}   M_{p,q} 2^p  \,,
  \quad
\#  \supp_\rp(\bw_\eta)  \leq  \sum_{p,q\geq 0 } 3 M_{p,q} 2^q\,.
\eeqn
Note that the factor of $3$ in the bound for $\#  \supp_\rp(\bw_\eta)$ results from the bidiagonal form of the matrices $\bM_j$; that is, the action of each of these matrices can add at most twice the number of nonzero entries in the preimage sequence, in addition to the existing ones. 

The following lemma provides bounds for the right hand sides in \eref{rankw}. 

\begin{lemma}
\label{lemMpq}
For any fixed constant $a>1$ choose
\be
\label{alphapq}
\alpha_{p,q}=c\,\big((1+ p)(1+q)\big)^{-a}, \quad c: = \Big(\sum_{p,q\geq 0} \big((1+ p)(1+q)\big)^{-a}\Big)^{-1},
\ee
as weights in \eref{Mpqest}.
Then for $S\geq  \bar s$ one has
\begin{multline}
	\label{simplep}
\sum_{p,q}2^p M_{p,q}\lesssim \eta^{-\frac 1S}\|\bv\|_{\Sb}^{\frac 1S}\bigl(1+\log_2\#\supp_\rp(\bv)\bigr)^{1+\frac aS}  \\
\times \bigl(1+\log_2\rank(\bv_\eta)\bigr)^{1+\frac aS}\bigl(\rank(\bv_\eta)\bigr)^{1- \frac{\bar s}S},
\end{multline}
where the constant depends on $a,S, \bar s$, on $c$ in \eref{alphapq}, and on $C$ in Assumptions \ref{ass:benchmark}\eref{MS}.

Similarly, for $S \geq s_\rp$ one has 
 \begin{multline} 	
\label{simpleq}
\sum_{p,q}2^q M_{p,q}
\lesssim \eta^{-\frac 1S} \|\pi^{(\rp)}(\bv)\|_{\Acal^{s_\rp}}^{\frac 1S}\bigl(1+ \log_2 \rank(\bv_\eta)\bigr)^{1+\frac aS}  \\
\times \bigl(1+ \log_2\#\supp_\rp(\bv_\eta)\bigr)^{1+\frac aS} \bigl(\#\supp_\rp(\bv_\eta)\bigr)^{1- \frac{s_\rp}S}\end{multline}
with similar dependencies of the constants as before, but with $\bar s$ replaced by $s_\rp$.
\end{lemma}

\begin{proof}
Bounding $M_{p,q}\lesssim \eta^{-\frac 1S} \|\bv\|_{\Sb}^{\frac 1S}(1+q)^{\frac aS}(1+p)^{\frac aS} 2^{-\frac{\bar s p}{S}}$,
we derive
\begin{equation}
\label{pestimate}
\sum_{p,q}2^p M_{p,q}\lesssim \eta^{-\frac 1S}\|\bv\|_{\Sb}^{\frac 1S}\bigl(1+\log_2 \#\supp_\rp(\bv_\eta)\bigr)^{1+\frac aS}\sum_p (1+p)^{\frac aS}
2^{p \left(1-\frac{\bar s}S \right)},
\end{equation}
which gives \eref{simplep},
where the constant depends on $a,S, \bar s$ and $c,C$ from \eref{Mpqest}.

To bound $\sum_{p,q}2^qM_{p,q}$ we use   $M_{p,q}\lesssim \eta^{-\frac 1S}\|\pi^{(\rp)}(\bv)\|_{\Acal^{s_\rp}}^{\frac 1S}
(1+p)^{\frac aS} (1+q)^{\frac aS}2^{-\frac{s_\rp q}S}$ and obtain
 \begin{equation*}
\sum_{p,q}2^q M_{p,q} \lesssim \eta^{-\frac 1S}\|\pi^{(\rp)}(\bv)\|_{\Acal^{s_\rp}}^{\frac 1S}\bigl(1+ \log_2\rank(\bv_\eta)\bigr)^{1+\frac aS}\sum_q (1+q)^{\frac aS}2^{q \left(1- \frac{s_\rp}S \right)}
\end{equation*}
which yields \eref{simpleq}.
 \end{proof}

\newcommand{\rates}{{s_\rs}}

We proceed estimating the various sparsity norms of $\bw_\eta$. We first address rank growth and parametric sparsity, which are independent
of the specific choice of $\tilde\bA_{j,p,q}$. 
Using \eref{rankw} and \eref{simplep} in Lemma \ref{lemMpq} together with \eref{vcoarsened} and \eref{supps}, for $S \geq \bar s$ we obtain 
\begin{align}
\label{rankfirst}
\rank(\bw_\eta) &\lesssim \eta^{-\frac 1S}(1+\abs{\log\eta})^{2\left(1+\frac aS\right)}\|\bv\|^{\frac 1S}_{\Sb}
 \,\eta^{-\frac{1}{\bar s}\left(1-\frac{\bar s}S\right)}\| \bv\|_{\Sb}^{\frac{1}{\bar s}\left(1-\frac{\bar s}S\right)}, \nonumber\\
& = \eta^{-\frac{1}{\bar s}}\| \bv\|_{\Sb}^{\frac{1}{\bar s}} (1+\abs{\log\eta})^{2\left(1+\frac aS\right)}, 
\end{align}
where the constant depends also on $\abs{\log \|\pi^{(i)}(\bv)\|_{\Acal^{s_i}}}$, $i\in \{\rs, \rp\}$.
Now suppose that $N_\eta$ is an upper bound for  $\rank(\bw_\eta)$. 
To simplify the exposition, let us assume without loss of generality that $\eta \in(0,1)$.
Then, by definition,
one has 
\begin{align*}
\|\bw_\eta\|_{\Sb}& = \sup_{N\le N_\eta}N^{\bar s}\inf_{\rank(\bw)\le N}\|\bw_\eta - \bw\|
\le \sup_{B\in [1,\eta^{-1}]}N_{B\eta}^{\bar s}\|\bw_\eta-\bw_{B\eta}\|\\
&\le \sup_{B\in [1,\eta^{-1}]} N_{B\eta}^{\bar s}\big(\|\bw_\eta -\bA \bv_\eta\|+ \|\bA \bv_\eta-\bw_{B\eta}\|\big)\le 
\sup_{B\in [1,\eta^{-1}]} 2B\eta N_{B\eta}^{\bar s}.
\end{align*}
Now we can invoke for each $B\in [1,\eta^{-1}]$ the upper bound for $\rank(\bv_\eta)$ given by \eref{rankfirst},
and observe that the resulting bound is maximized for $B=\eta^{-1}$ when $S\ge \bar s$. This gives
\begin{equation}
\label{stabbars}
\|\bw_\eta\|_{\Sb}\lesssim  \|\bv\|_{\Sb} (1+\abs{\log\eta})^{2\bar s\left(1+\frac aS\right)},
\ee
which  confirms \eref{rankbars}.

 Similarly, using the second estimate in \eref{rankw} and \eref{simpleq} in Lemma \ref{lemMpq} and invoking \eref{supps}
yields, for $S \geq s_\rp$,
\begin{equation}
\label{suppp}
\#\supp_\rp(\bw_\eta)
 \lesssim \eta^{-\frac 1S} \|\pi^{(\rp)}(\bv)\|_{\Acal^{s_\rp}}^{\frac 1S} (1+\abs{\log_2 \eta})^{2+\frac{2 a}S} 
\biggl(\frac{\|\pi^{(\rp)}(\bv)\|_{\Acal^{s_\rp}}}{\eta} \biggr)^{\frac{1}{s_\rp}\left(1- \frac{s_\rp}S\right)} .
\end{equation}
By the same argument as before   one obtains
\be
\label{suppp2}
\#\supp_\rp(\bw_\eta)\lesssim \eta^{-\frac{1}{s_\rp}} \|\pi^{(\rp)}(\bv)\|_{\Acal^{s_\rp}}^{\frac{1}{s_\rp}} (1+\abs{\log_2 \eta})^{2+\frac{2 a}S}.
\ee
We can then continue as above, denoting by $M_\eta$ an upper bound for  $\#\supp_\rp(\bw_\eta)$, to argue  
\begin{align*}
\|\pi^{(\rp)}(\bw_\eta)\|_{\Acal^{s_\rp}} &\le  \sup_{B\in [1,\eta^{-1}]} M_{B\eta}^{s_\rp}\big(\|\pi^{(\rp)}(\bw_\eta)-\pi^{(\rp)}(\bA\bv_\eta)\|
+ \|\pi^{(\rp)}(\bw_{B\eta})-\pi^{(\rp)}(\bA\bv_\eta)\|\big)\\
&\leq
 \sup_{B\in [1,\eta^{-1}]} M_{B\eta}^{s_\rp}\big(\| \bw_\eta - \bA\bv\|
+ \| \bw_{B\eta}- \bA\bv\|\big)\\
&\le   \sup_{B\in [1,\eta^{-1}]} 2B\eta M_{B\eta}^{s_\rp}.
\end{align*}
Thus we obtain 
\be
\label{stabsp}
\|\pi^{(\rp)}(\bw_\eta)\|_{\Acal^{s_\rp}} \lesssim \|\pi^{(\rp)}(\bv)\|_{\Acal^{s_\rp}} (1+\abs{\log_2 \eta})^{2s_\rp\left(1+\frac{a}S\right)},
\ee
which together with \eref{suppp2} shows \eref{rp}.

We now turn to estimating $\#\supp_\rs(\bw_\eta)$ and $\|\pi^{(\rs)}(\bw_\eta)\|_{\Acal^{s_\rs}}$.  
To this end, we specify suitable compressed matrices  $\tilde\bA_{j,p,q}$ in \eqref{etapq}. 
Denoting by $\pi^{(\rs)}(\bv_{[p,q]})_{\ell}$ the best $\ell$-term approximation of $\pi^{(\rs)}(\bv_{[p,q]})$, we set $\Lambda_{p,q,0} := \supp(\pi^{(\rs)}(\bv_{[p,q]})_{1})$ and
\[
\Lambda_{p,q,n}:= \supp\big(\pi^{(\rs)}(\bv_{[p,q]})_{2^n}\big) \setminus \supp\big(\pi^{(\rs)}(\bv_{[p,q]})_{2^{n-1}}\big), \quad n\in\N.
\]
Note that
\[
  \norm{ \Restr{\Lambda_{p,q,n}\times \pidx} \bv_{[p,q]} }
    \leq \norm{  \Restr{\Lambda_{p,q,n}} \pi^{(\rs)}(\bv_{[p,q]}) } \leq 2^{-s_\rs n} \norm{ \pi^{(\rs)}(\bv_{[p,q]})}_{\Acal^{s_\rs}} .
\]
To proceed we employ the following convenient reformulation of Proposition \ref{prop:sparsecompr}.

\begin{remark}\label{rem:sparsecompr}
Let $M \in \N$ and $s < \frac{2\tau}{1+2\tau} S$. Then for any $J\in \N$ we can find $\bA_j^J$, $j\geq 0$, such that 
\begin{equation*}
	\Bignorm{\sum_{j=0}^M  \bigl( \bA_j  -  \bA_j^J \bigr) \otimes \bM_j} \leq \beta_J 2^{-s J},
\end{equation*}
and the following holds: for each $\lambda\in\sidx$, for the sum of the number of corresponding nonzero column entries of the $\bA_j^J$ we have the bound
\begin{equation}\label{sparsecompr_sum}
	\sum_{j = 0}^M \#\supp \,\bigl( \bA_{j,\lambda'\lambda}^J \bigr)_{\lambda'\in\sidx} 
	  \leq \alpha_J 2^J.
\end{equation}
Here $\boldsymbol{\alpha},\boldsymbol{\beta}$ are positive summable sequences.
\end{remark}

For a suitable nonnegative integer $N=N_{j,p,q,\eta}$, let $\tilde\bA_{j,p,q} \coloneqq \sum_{n = 0}^N \bA_j^{N - n} \Restr{\Lambda_{p,q,n}}$ and
\begin{equation}\label{wpq}
 \bw_{p,q} := \sum_{j=0}^{M_{p,q}}(\tilde\bA_{j,p,q} \otimes \bM_j)\bv_{[p,q]}.
\end{equation}
Then
\begin{equation*}
  \Bignorm{\bw_{p,q} - \sum_{j=0}^{M_{p,q}}\big( \bA_j \otimes\bM_j \big)\bv_{[p,q]} } = \Bignorm{\sum_{j=0}^{M_{p,q}} \sum_{n =  0}^N
  \big( (\bA_j^{N-n} - \bA_{j})\Restr{\Lambda_{p,q,n}} \otimes\bM_j\big) \bv_{[p,q]} } .
 \end{equation*}
Using Remark \ref{rem:sparsecompr} with $s=s_\rs$, the right side can be estimated by
 \begin{align*}
    \sum_{n = 0}^N \beta_{N-n} 2^{-s_\rs(N-n)} 2^{-s_\rs n} \norm{ \pi^{(\rs)}(\bv_{[p,q]})}_{\Acal^{s_\rs}} + &2 \norm{\bA} \sum_{n>N}   2^{-s_\rs n} \norm{ \pi^{(\rs)}(\bv_{[p,q]})}_{\Acal^{s_\rs}}   \\
    &\lesssim  2^{-s_\rs N}  \norm{ \pi^{(\rs)}(\bv_{[p,q]})}_{\Acal^{s_\rs}} ,
  \end{align*}
where the constant depends on $s_\rs$, $\norm{\bA}$, and $\norm{ \boldsymbol{\beta} }_{\ell^1}$.
By \eqref{sparsecompr_sum}, we obtain
\begin{equation}\label{supp pq}
 \#\supp_\rs (\bw_{p,q}) \lesssim \sum_{n = 0}^N 2^n \alpha_{N-n} 2^{N-n}  \lesssim 2^N.
\end{equation}
If we now choose the smallest $N$ such that \eqref{etapq} holds,  i.e.,  $2^{-s_\rs N}  \norm{ \pi^{(\rs)}(\bv_{[p,q]})}_{\Acal^{s_\rs}} \lesssim \eta_{p,q}$, we obtain
\begin{equation*}
	 \#\supp_\rs (\bw_{p,q}) \lesssim \eta_{p,q}^{-\frac1{s_\rs}} \norm{ \pi^{(\rs)}(\bv_{[p,q]})}_{\Acal^{s_\rs}}^{\frac1{s_\rs}}
	 \lesssim  \eta_{p,q}^{-\frac1{s_\rs}}\norm{ \pi^{(\rs)}(\bv_{\eta})}_{\Acal^{s_\rs}}^{\frac1{s_\rs}}.
\end{equation*}
Keeping the definition of $\eta_{p,q}=\alpha_{p,q}\eta$ and \eref{vcoarsened}, \eref{supps} in mind, summing over $p,q$
gives \eref{sstab} with $b= 2\big(1+\frac{a}{s_\rs}\big) > 2\big(1+\frac{a}{S}\big)$, where the bound on $\norm{\pi^{(\rp)}(\bw_\eta)}_{\Acal^{s_\rp}}$ follows as in \eqref{stabbars} and \eqref{stabsp}. 
\end{proof}

\begin{remark}
Note that in Assumptions \ref{ass:benchmark}, we state that $S\geq \bar s, s_\rp$ and $S > s_\rs$. While other cases can in principle be considered in the same manner, the convergence rate $S$ of the operator truncation then limits the achievable efficiency: if $S< \bar s$, for instance, it is easy to see that in general one can only obtain $\rank(\bw_\eta) \sim \Ocal(\eta^{-1/S})$.
\end{remark}

\begin{proposition}\label{hsworkest}
	Under the assumptions of Theorem \ref{thmaincompl},
	let $\bv$ be given by its SVD with $r \coloneqq \rank(\bv)$ and $n_i \coloneqq \#\supp_i(\bv)$ for $i\in \{\rs,\rp\}$.
	Then for the number of operations $\ops(\bw_\eta)$ required to obtain $\bw_\eta$, one has
	\begin{multline}
		 \ops(\bw_\eta) \lesssim  (n_\rs +  n_\rp) r^2  
			+ \Bigl( ( 1 + \abs{\log \eta})^{\frac{2a}{s_\rs}} \eta^{-\frac1{s_\rs}}\norm{ \pi^{(\rs)}(\bv)}_{\Acal^{s_\rs}}^{\frac1{s_\rs}}  \\
			    + (1+\abs{\log \eta} )^{\frac {2a}S} \eta^{-\frac1{s_\rp}} \norm{\pi^{(\rp)}(\bv)}^{\frac1{s_\rp}}_{\Acal^{s_\rp}} \Bigr) \eta^{-\frac 1{\bar s}}\|\bv\|_{\Sb}^{\frac 1{\bar s}}.
	\end{multline}
\end{proposition}

For the proof of this proposition, we refer to Appendix \ref{app:thm3}.
With these preparations, we obtain the following complexity estimate for Algorithm \ref{alg:tensor_opeq_solve}.

\begin{theorem}\label{lrcomplexity}
 Let Assumptions \ref{ass:benchmark} hold.
	Then for any $\varepsilon>0$, the approximation $\bu_\varepsilon$  of the form \eqref{fulllr} produced by Algorithm \ref{alg:tensor_opeq_solve} satisfies
	\begin{equation}\label{hsrankbound}
		 \rank(\bu_\varepsilon) \lesssim \varepsilon^{-\frac1{\bar s}} \norm{\bu}_{\Sb}^{\frac1{\bar s}}, \quad \norm{\bu_\varepsilon}_{\Sb} \lesssim \norm{\bu}_{\Sb}
	\end{equation}
	and
	\begin{equation}\label{hssuppbound}
		 \sum_{i\in\{\rs,\rp\}} \#\supp_i (\bu_\varepsilon) \lesssim  \sum_{i\in\{\rs,\rp\}} \varepsilon^{-\frac1{s_i}} \norm{\pi^{(i)}(\bu)}_{\Acal^{s_i}}^{-\frac{1}{s_i}}, \quad 
		 \norm{\pi^{(i)}(\bu_\varepsilon)}_{\Acal^{s_i}} \lesssim \norm{\pi^{(i)}(\bu)}_{\Acal^{s_i}}.
	\end{equation}
	The number of operations $\ops(\bu_\varepsilon)$ required to produce $\varepsilon$ then satisfies
	\begin{equation}\label{hsopbound}
\ops(\bu_\varepsilon) \lesssim  1+ ( 1 + \abs{\log \varepsilon})^\zeta \Bigl( \varepsilon^{-\frac1{\bar s}} \norm{\bu}_{\Sb}^{\frac1{\bar s}} \Bigr)^2  \sum_{i\in\{\rs,\rp\}} \varepsilon^{-\frac1{s_i}} \norm{\pi^{(i)}(\bu)}_{\Acal^{s_i}}^{-\frac{1}{s_i}} ,
	\end{equation}
	where $\zeta>0$ depends on $s_\rs$, on $\operatorname{cond}(\bA)$, and on the choice of $\kappa_1, \beta$ in Algorithm \ref{alg:tensor_opeq_solve}. The constants in \eqref{hsrankbound}, \eqref{hssuppbound}, and \eqref{hsopbound} may also depend on $S$, $\bar s$, $s_\rp$, and on the further parameters of Algorithm \ref{alg:tensor_opeq_solve}.
\end{theorem}

\begin{proof}
We follow the general strategy of the proofs as in \cite{BD} and in Theorem \ref{fullsepresult}, combining the properties of the complexity reduction procedures $\coarsen$ and $\recompress$ with the specific adaptive operator application that we have constructed for the present problem.

  The bound \eqref{hsrankbound} and \eqref{hssuppbound} follow from Theorem \ref{lmm:combined_coarsening} applied to the result of line \ref{alg:cddtwo_coarsen_line} in Algorithm \ref{alg:tensor_opeq_solve}.
  Note that here, the number $J$ of inner iterations depends only on $\operatorname{cond}(\bA)$ (via $\rho,\omega$) and on the choice of $\kappa_1$ and $\beta$.
  With the complexity estimates for $\apply$ from Theorem \ref{thmaincompl} and Proposition \ref{hsworkest} at hand, we obtain \eqref{hsopbound}.
\end{proof}

\begin{remark}
{The present version of Algorithm \ref{alg:tensor_opeq_solve} necessarily performs orthogonalizations of the basis vectors in the computed low-rank expansions.
Under the ensuing requirement of common supports, Theorem \ref{lrcomplexity} shows that the output of Algorithm \ref{alg:tensor_opeq_solve} for the format \eqref{fulllr} has optimal {\em representation complexity}. Regarding the computational complexity of the algorithm,
	as can be seen from the proofs of} Theorem \ref{thmaincompl} and Proposition \ref{hsworkest}, the numerical cost for the approximate operator application is dominated by the cost of performing orthogonalizations of the input. In particular, this leads to a quadratic dependence on the approximation ranks, {and up to the logarithmic term, \eqref{hsopbound} represents the best possible bound for an algorithm performing such orthogonalizations.
	The number of subsequent operations required to construct the low-rank representation of the output, however, remains proportional to the respective number of degrees of freedom. }
\end{remark}

\section{Approximability of parametric problems}
\label{sec:approximability}

In this section, we consider representative instances of \eqref{paramdiffusion-0} 
in order to compare the respective properties that determine the efficiency of the variants of our scheme for \eqref{fullnterm}, \eqref{fulllr}, and \eqref{general}.

\subsection{Isotropic dependence on finitely many parameters}\label{sec:approx iso}

As simple yet instructive examples, we consider problems with $\bar{a}=1$ and
\beqn
\label{cm}
   \pf_j = b_j \Chi_{D_j}, 
\eeqn
 where $b_j \in ]0,1[$ are constants and the subdomains $D_j$ of the domain $D$ have disjoint closures so that 
 the diffusion coefficient is a strictly positive piecewise constant,
 \begin{equation} \label{inclusions}
    a(y) = \bar{a} + \sum_{j\geq 1} y_j b_j \Chi_{D_j}.
 \end{equation}

As a first problem of this type with $D = ]0,1[$, we consider the following.

\begin{example}\label{ex1}
 Let $d := \#\mathcal{I} < \infty$, $D_j\subset D=]0,1[$ for $j=1,\ldots,d$ with pairwise disjoint $\overline{D_j}$, and $b_j =\xi$ for some $\xi\in ]0,1[$. 
\end{example}

For low-rank approximation, we then have the following result for the rank of the Hilbert-Schmidt decomposition \eqref{hsapprox}.

\begin{proposition}
In Example \ref{ex1}, for any $f\in V'$, one has $\rank(\bu) \leq 4 d + 1$.
\end{proposition}

\begin{proof}
This follows by the same arguments as in \cite[Example 2.2]{BC:15}: the endpoints of the $D_j$ induce a partition of $]0,1[$ into at most $2d+1$ intervals. For each such interval $I$, for any $F$ such that $F''=f$, we have
$u(y)|_I \in \linspan\{ \Chi_{I} ,x\,\Chi_{I} , F\,\Chi_{I} \}$.
Hence $u(y)$ is contained in a $y$-independent space of dimension $6d+3$ for all $y$. In addition, there are $2d+2$ continuity conditions, independent of $y$, at the interval boundaries, which leaves at most $4d+1$ degrees of freedom.
\end{proof}

We observe on the other hand that the Legendre expansions for this problem involves 
infinitely many nonzero coefficients, that is, the solution map $y\mapsto u(y)$
is not a polynomial in $y$. This can be checked, for example, by considering the Taylor coefficients of $u$.
For any $\nu=(\nu_j)_{j\geq 1}\in \cF$, the coefficients in the Taylor expansion of $u$ are given by
\be
t_\nu(y) = \frac1{\nu!}\partial^\nu u(y), \quad \nu !:=\prod_{j \geq 1}\nu_j !,
\ee
Denoting by $e^j=(0,\dots,0,1,0,\dots)$ the $j$-th Kronecker sequence, differentiating the equation we
find that these coefficients are given by the recursion
\begin{equation}\label{eq:taylor_recursion}
t_\nu(y) := -A(y)^{-1}\sum_{j\in\supp\nu} A_j  t_{\nu - e^j}(y), \quad t_0(y) = A(y)^{-1} f=u(y).
\end{equation}

We now consider the Taylor coefficients of order $n$ in a given variable $j$ at the origin, that is,
\be
t_{n,j}:=t_{ne^j} (0)=\frac 1 {n!}  \partial^{n}_{y_j} u(0).
\ee
As a particular case of \eref{eq:taylor_recursion}, we have
\be
\int_D \bar a \nabla t_{n,j}\cdot \nabla v \,dx =-\int_D \theta_j \nabla t_{n-1,j} \cdot\nabla v \,dx .
\ee
Since $t_{0,j}=u(0)$ is not trivial, there is at least one variable $j$ such that $t_{1,j}$ does not vanish
on $D_j$. Then, taking $v=t_{n-1,j}$ in the above recursion shows by contradiction that 
$t_{n,j}$ does not vanish on $D_j$, for all values of $n\geq 0$. Thus $y\mapsto u(y)$
cannot be a polynomial.
Low-rank approximations thus give substantially faster convergence than
Legendre expansions in this case.

Similar results showing substantial advantages of best low-rank approximations have also been obtained for spatially two-dimensional examples of analogous structure in \cite{BC:15}.
The test problems considered there are of the form \eqref{inclusions} as well, with the coefficients piecewise constant on $D:=]0,1[^2$ and where $D_j$, $j=1,\ldots, d$ are a partition of $D$ into congruent square subdomains. The resulting ``checkerboard'' geometry is illustrated for $d = 16$ in Figure \ref{fig:checkerboard}.

The low-rank approximability of such problems with respect to space-parameter separation has been studied in \cite{BC:15}.
For the case $d=4$ (that is, a $2\times 2$-checkerboard), it is shown in \cite{BC:15} that for each $n\in \N$ one can find $u^\rs_k$, $u^\rp_k$ for $k=1,\ldots,n$ such that for some $c>0$,
\[
  \Bignorm{ u - \sum_{k=1}^n u^\rs_k \otimes u^\rp_k}_{L^2(\pdom,V)} \lesssim e^{- c n}.
\]
Numerical tests indicate that an analogous estimate can be achieved also for geometries of the type shown in Figure \ref{fig:checkerboard} with $d=9,16,25,\ldots$, where $c$ has a moderate dependence on $d$.
Note also that for a hierarchical tensor representation, the ranks of further matricizations enter as well.  We are not aware of any bounds for these additional ranks. The numerically observed decay of the corresponding singular values for different values of $d$ (using a linear dimension tree) are shown in Figure \ref{fig:checkerboardsvds}. {Note that the singular values of the matricization $T^{\{\rs\}}_\bu$ are precisely those in the decomposition \eqref{hsapprox} underlying \eqref{fulllr}.}

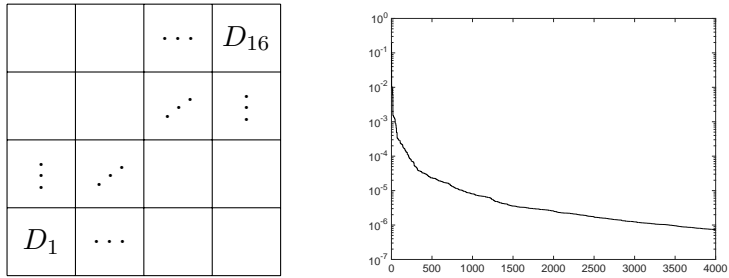
\begin{figure}
\centering
\begin{tabular}{cc}
\hspace{6pt}
\begin{tikzpicture}[scale=.9]
\draw[step=1cm,black,thin] (0,0) grid (4,4);
\node at (0.5,0.5) {$D_1$};
\node at (3.5,3.5) {$D_{16}$};
\node at (.5,1.6) {$\vdots$};
\node at (1.5,.5) {$\hdots$};
\node at (1.5,1.6) {$\iddots$};
\node at (2.5,2.6) {$\iddots$};
\node at (2.5,3.5) {$\hdots$};
\node at (3.5,2.6) {$\vdots$};
\end{tikzpicture}	
&
\hspace{9pt}\raisebox{-6pt}{\includegraphics[width=5.5cm]{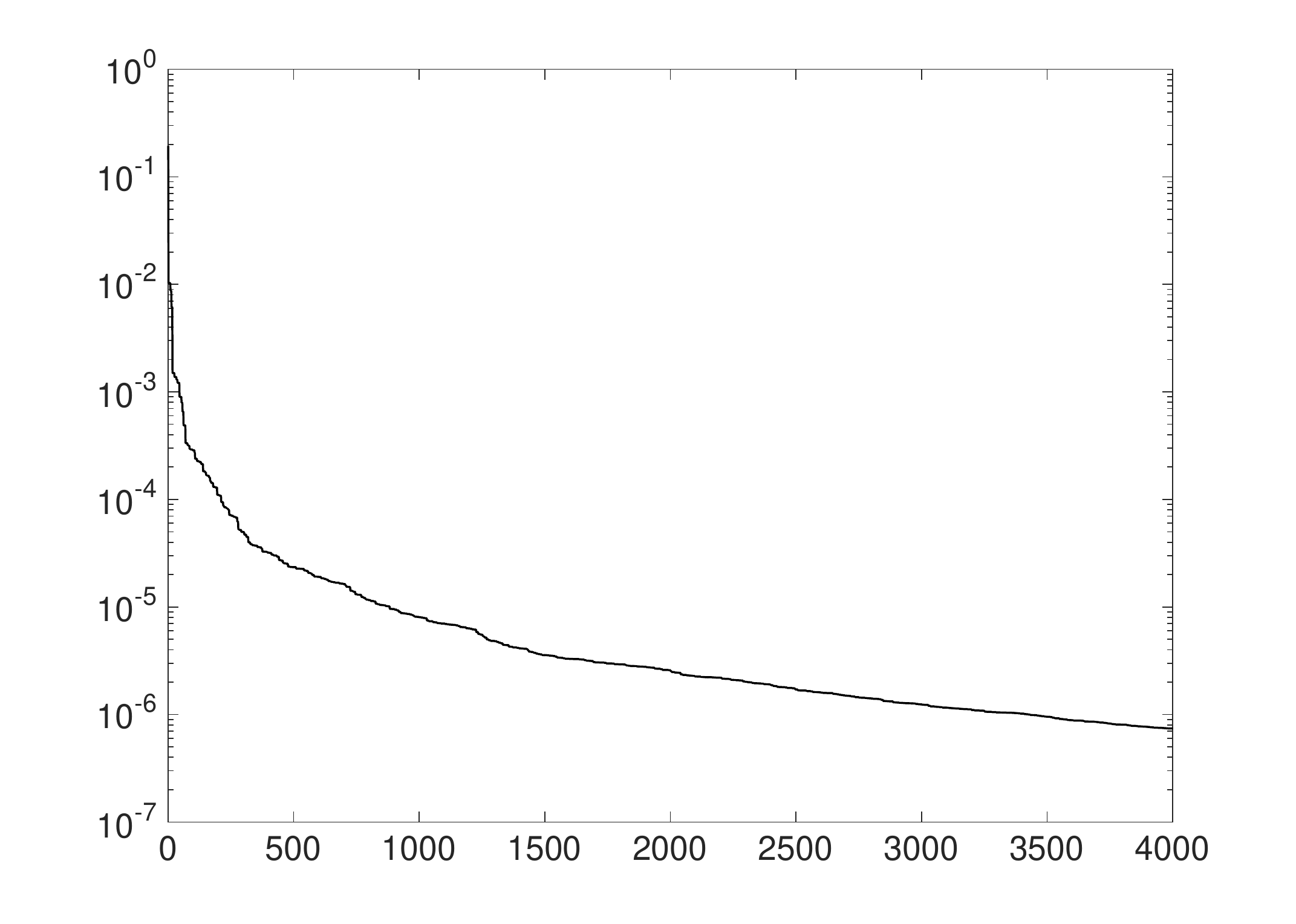}}
\end{tabular}
\caption{Example geometry of piecewise constant coefficients $a$ with $d=16$ and decay of Legendre coefficients of corresponding $u$.}\label{fig:checkerboard}
\end{figure}

\begin{figure}
\centering
\begin{tabular}{ccc}
  \hspace{-18pt}\includegraphics[width=5.3cm]{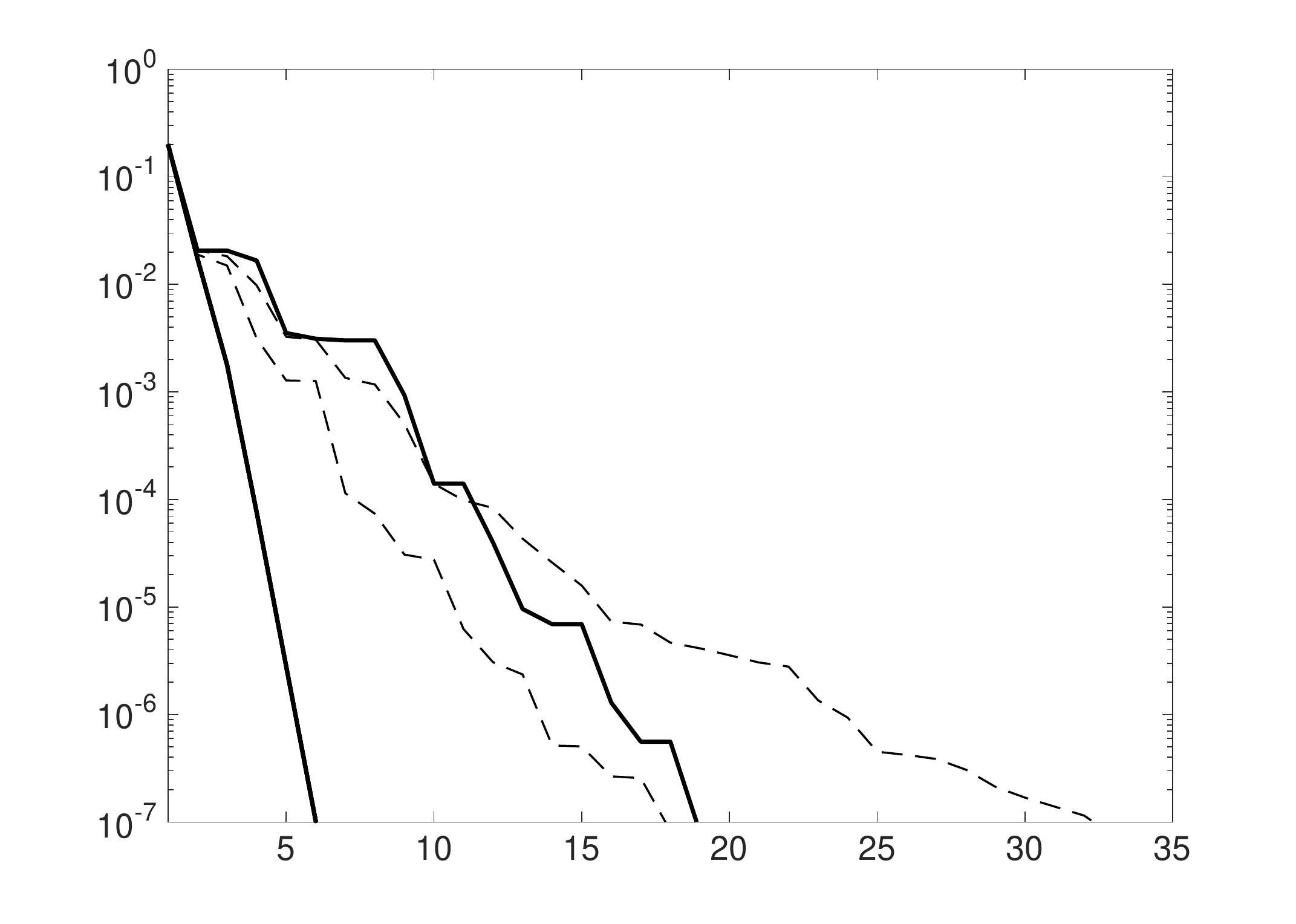} &
  \hspace{-12pt}\includegraphics[width=5.3cm]{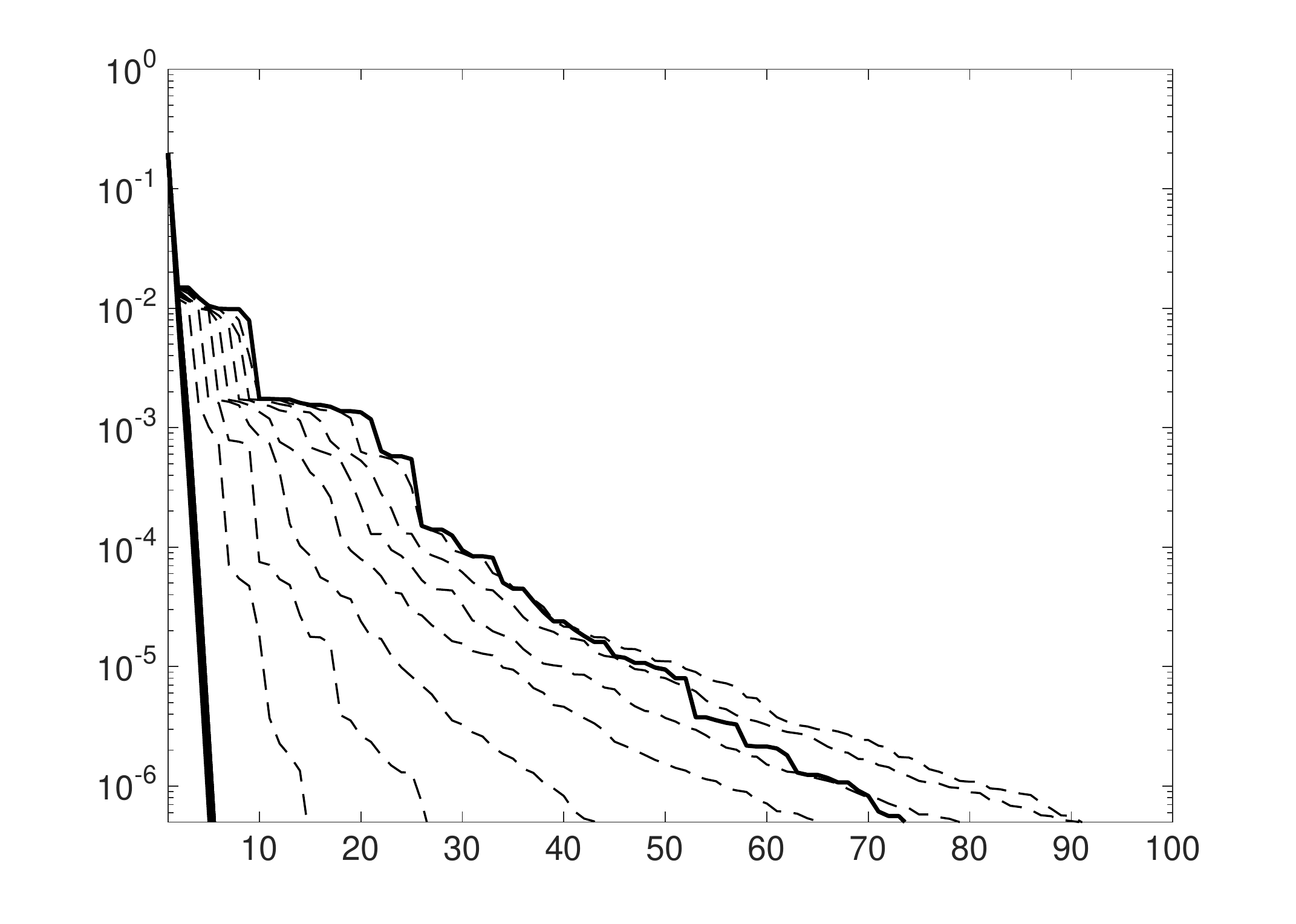} &
  \hspace{-12pt}\includegraphics[width=5.3cm]{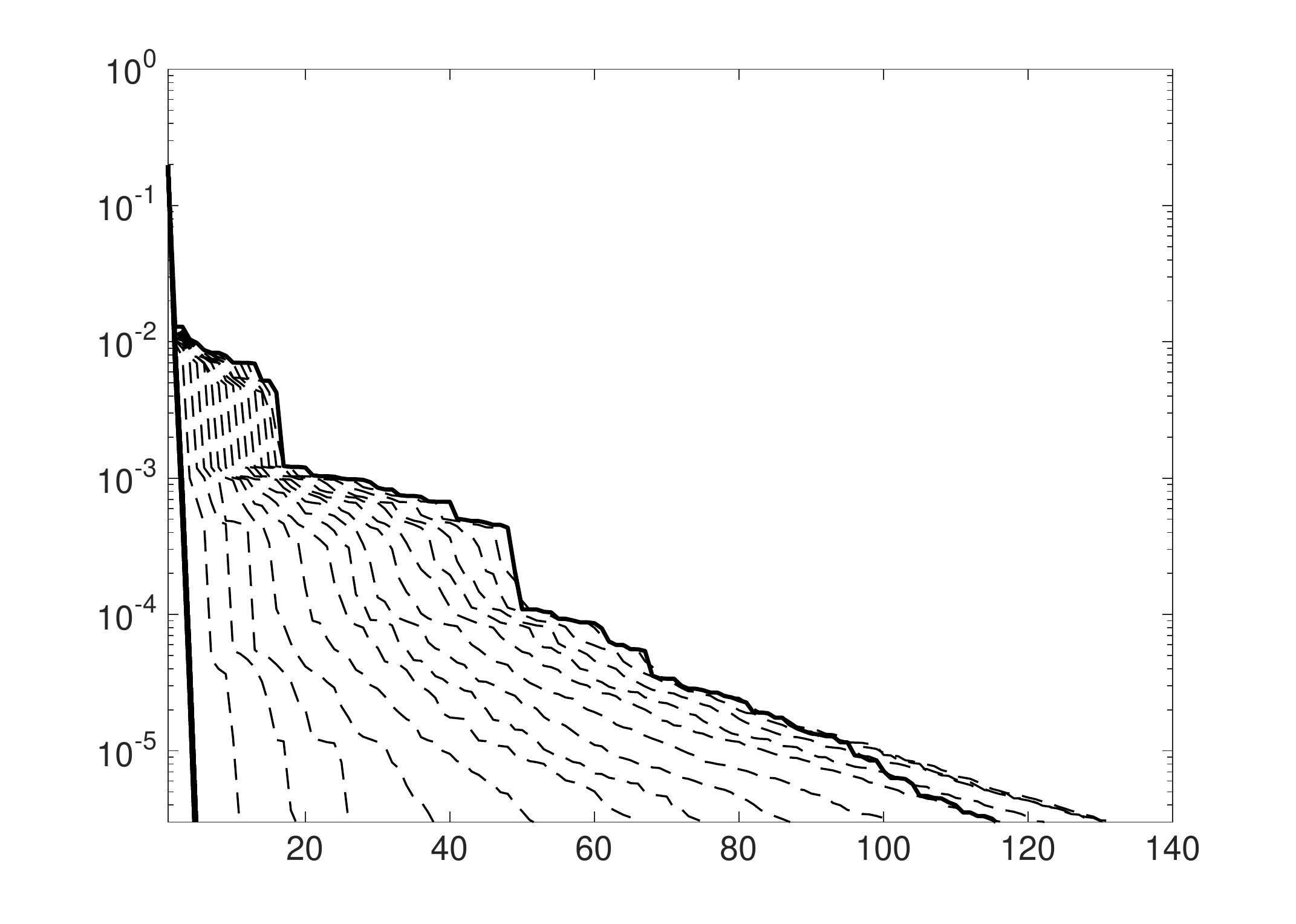}  \\[-3pt]
  	\small $d=4$ & \small $d=9$ & \small $d=16$
\end{tabular}
\caption{Hierarchical singular values of $\bu$, where $D$ has $\sqrt{d}\times\sqrt{d}$-checkerboard geometry as in Figure \ref{fig:checkerboard}. \emph{Solid lines:} singular values of matricizations $T^{(\{i\})}_\mathbf{u}$ associated to $i\in\hat \cI$ (essentially identical for $i\in\cI$, with the line of slower decay corresponding to $i=\rs$), \emph{dashed lines:} singular values of further matricizations in the hierarchical representation. The horizontal axes show the numbers of the decreasingly ordered singular values.} 
\label{fig:checkerboardsvds}	
\end{figure}

\begin{remark}
As we have noted for the spatially one-dimensional case in Example \ref{ex1} in \S\ref{sec:approximability}, for the separation between spatial and parametric variables for that case one always obtains \emph{fixed} finite ranks that grow linearly in the number of parameters $d$. Note, however, that the approximation ranks corresponding to further separations among the parametric variables may then still not be uniformly bounded; see e.g.\ \cite[Prop.\ 2.5]{Khoromskij:11-1} for an analysis of a simple example.
\end{remark}

In all examples considered above, we observe exponential-type decay of singular values. In particular, the numerical results in Figure \ref{fig:checkerboardsvds} indicate that Assumptions \ref{ass:fullbenchmark}, \ref{ass:fullbenchmark dim} are met for this family of problems. In fact, the obtained hierarchical singular values are consistent with the decay $\exp({-\bar c n^{1/\bar b}})$ with some $\bar b > 1$ independent of $d$ and $\bar c >0$ algebraic in $d$. The compressibility of any desired order of the operators $\bA_j$ is known from classical wavelet theory.

In contrast, the decay of Legendre coefficients $\norm{u_\nu} \sim \pi^{(\rp)}_\nu(\bu)$ is significantly slower, where as in Figure \ref{fig:checkerboard} one only observes a decay of the form $\exp({-c n^{1/d}})$ for the $n$-th largest  $\norm{u_\nu}$ as predicted by the available estimates (see, e.g., \cite{BC:15}). This indicates a clear advantage of hierarchical tensor approximations of the form \eqref{general} over sparse polynomial expansions \eqref{fullnterm} for such problems.

\subsection{Anisotropic dependence on infinitely many parameters}\label{sec:approx aniso}

We next consider a problem of the form \eqref{inclusions} with countably many parameters of decreasing influence, where our conclusion are quite different from those concerning Example \ref{ex1}.

\begin{example}\label{ex2}
 Let $\mathcal{I} = \N$, and let $D_j\subset ]0,1[$ be disjoint with $\abs{D_j}>0$ for all $j$. In addition, let $(b_j)_{j\geq 1} \in \spl{q}(\N)$ for some $q>0$. 
\end{example}

As an immediate consequence of the results in \cite[\S4.1]{BCM:2015}, one has the following.

\begin{proposition}
In Example \ref{ex2}, for all right hand sides $f\in V'$, one has $(\norm{u_\nu}_V)_{\nu \in \pidx} \in \spl{p}(\pidx)$ for $p = \frac{2q}{2+q}$. If $(b_j)\notin \spl{q'}(\N)$ for any $0<q'<q$, then there exists $f\in V'$ such that $(\norm{u_\nu}_V)_{\nu \in \pidx} \notin \spl{p'}(\pidx)$ for $0<p'< p$. 
\end{proposition}

If $\sigma_n$ are the singular values of $u$, then for the decreasing rearrangement $(u^*_n)_{n\geq 1}$ of $(\norm{u_\nu}_V)_{\nu\in\pidx}$ we clearly have $u^*_n \geq \sigma_n$.
As the following new result shows by similar arguments as in \cite[\S4.1]{BCM:2015},  the singular values do not necessarily have faster asymptotic decay in this situation than the ordered norms of the Legendre coefficients.

\begin{proposition}\label{prop:decaylowerbound}
In Example \ref{ex2}, if $(b_j)\notin \spl{q'}(\N)$ for any $0<q'<q$, then there exists an $f \in V'$ such that the singular values of $u$ are not in $\spl{p'}(\N)$ for $0<p' < p = \frac{2q}{2+q}$.
\end{proposition}

\begin{proof}
We first observe that the singular values of $u = \sum_{\nu\in\pidx} u_\nu\otimes L_\nu$ are bounded from below by those of $\tilde u =\sum_{j\geq 1} u_{e^j} \otimes L_{e^j}$, with $e^j$ denoting the $j$-th Kronecker sequence. This follows from the fact that $\tilde u = (\id \otimes \tilde P) u$, where $\tilde P$ is the projector onto $\overline{\operatorname{span}}\{L_{e^j}\}_{j\geq 1}$.

For $u_{e^j}$, one has by Rodrigues' formula the explicit representation
\begin{equation}\label{rodrigues2}
   u_{e^j} =  \frac{\sqrt3}{2} \int_Y t_{e^j}(y)\, (1 - y_j^2) \,d\mu(y)
\end{equation}
in terms of the first-order derivatives $t_{e^j}(y)=\partial_{y_j} u(y)$.

Let $h_j$ be the symmetric hat functions with support $D_j$. We now choose
\[ 
   f= -\sum_{j\geq 1} c_j h_j'',
\]
where $\sum_{j\geq 1} c_j^2 / \abs{D_j} < \infty$, which yields $f\in V'$ and 
\[ 
   t_0(y) = \sum_{j\geq 1} (1 + b_j y_j)^{-1} c_j h_j.
\]
By \eqref{eq:taylor_recursion}, $  t_{e^j}(y) = -(1 + b_j y_j)^{-2} b_j c_j h_j$
and as a consequence of \eqref{rodrigues2},
\[
   u_{e^j} =  -M_j  b_j c_j h_j, \qquad M_j :=  \frac{\sqrt3}{2}  \int_{-1}^1 \frac{1-y^2}{(1+b_j y)^2}\frac{dy}{2} \geq \frac1{4\sqrt{3}(1-\max_j b_j)} =: M_0.
\]
We thus obtain $ \langle u_{e^i}, u_{e^j}\rangle_V = 0$ for $i\neq j$,
as well as
\[
  \norm{u_{e^j}}_V \geq M_0 b_j c_j \norm{h_j}_V = \frac{2 M_0  b_j c_j}{\sqrt{\abs{D_j}}}.
\]
Since $(b_j)$ is precisely in $\spl{q}(\N)$, by choosing $c_j = b_j^{q/2} \sqrt{\abs{D_j}}$, which guarantees in particular that $(c_j/\sqrt{\abs{D_j}})_{j\geq 1} \in \spl{2}(\N)$ as required, we arrive at the statement.
\end{proof}

The above result shows that from an asymptotic point of view, in Example \ref{ex2}, there is not necessarily any gain by low-rank approximation: there always exist right hand sides $f$ such that the singular values 
have \emph{precisely the same} asymptotic decay as the ordered norms of Legendre coefficients.
Numerical tests as in Example \ref{ex:cfdecay} indicate that this also holds true for problems with different types of parametrization and more general $f$.

\begin{remark}
The conclusion of Proposition \ref{prop:decaylowerbound}
reveals that, in the case of Example \ref{ex2} and if $(b_j)_{j\geq 1}\notin\ell^{q'}$ for all $0<q'<q$, 
then any separable approximation 
of the form \eqref{basiclrapprox} satisfies
\be
 \|u-u_n\|_{L^2(Y,V)}\geq c_r n^{-r}, \quad n\geq 1,
\ee
for some $c_r>0$, whenever $r>\frac 1 q$. In turn, we also have
\be
\|u-u_n\|_{L^\infty(Y,V)} \geq  c_r n^{-r}, \quad n\geq 1.
\ee
This implies that the Kolmogorov $n$-width
\be
d_n({\cal M})_V=\inf_{\dim(E)\leq n} \max_{v\in {\cal M}}\operatorname{dist}(v,E)_V,
\ee
of the solution manifold ${\cal M}:=\{u(y)\; : \; y\in Y\}$ satisfies a similar lower bound
\be
d_n({\cal M})_V\geq c_r n^{-r}, \quad n\geq 1.
\ee
While upper bounds for $d_n({\cal M})_V$ in parametric PDEs are typically proved by exhibiting
a particular separable approximation and studying its convergence in $L^\infty(Y,V)$,
see \cite{BC:15,CDacta:15}, lower bounds are generally out of reach and the ones given above
constitute a notable exception.
\end{remark}

\begin{remark}
One arrives at analogous observations in similar higher-dimensional settings. The construction of Example \ref{ex2} immediately carries over to spatial domains with $m>1$ when the definition of $f$ is based on higher-dimensional hat functions.
\end{remark}

We next summarize the available knowledge on approximability of $\bu$ by representations \eqref{fullnterm} and \eqref{fulllr}.

\begin{proposition}\label{propapproxsummary}
	Let Assumptions \ref{ass:multiscale} hold with $\alpha \notin\N$ and let $\xi_\mu \in C^{\kappa}(D)$, $\mu\in \Lambda$, for a $\kappa>\alpha$. Then the following holds:
	\begin{enumerate}[{\rm(i)}]
		\item $\pi^{(\rp)}(\bu) \in \Acal^s(\pidx)$ for any $s < s^*_\rp := \frac\alpha{m}$.
		\item For sufficiently regular $f$ and $D$, and sufficiently regular wavelets $\psi_\lambda$, 
	\[
	\pi^{(\rs)}(\bu) \in \Acal^s(\sidx)\quad\text{for any}\quad
   s < s^*_\rs \coloneqq \frac\alpha{m}.
   \]
   \item 
   $\bu \in \Sigma^{\bar s}$ for a $\bar s \geq \max\{s_\rs, s_\rp\}$.
   \item If $0<\alpha \leq 1$, then $\bu \in \Acal^s(\sidx\times\pidx)$ for any $s < \frac23\alpha$ when $m=1$ and for any $s < \frac\alpha{m}$ when $m=2,3$.
	\end{enumerate}
\end{proposition}

\begin{proof}
Statement (i) follows directly from 	\cite[Cor.\ 4.2]{BCM:2015}, since $\pi^{(\rp)}(\bu) \sim \norm{u_\nu}_V$; (iii) follows immediately from properties of the Hilbert-Schmidt decomposition; and (iv) is shown in \cite[\S8]{BCDS:17}.

To see (ii), note first that by regularity of the $\xi_\mu$ and since $\alpha \notin \N$, we have $a(y) \in B^\alpha_{\infty,\infty}(D)= C^\alpha(D)$ for any $y\in\pdom$ and $\sup_{y\in\pdom} \norm{a(y)}_{C^\alpha} < \infty$.
Let $0 < s < \frac\alpha{m}$, and let $\psi_\lambda$ be sufficiently smooth to form a Riesz basis of $H^{1+s}(D)$. Then
\[
    \sum_{\lambda \in \sidx} 2^{2s \abs{\lambda}} \abs{\pi^{(\rs)}_\lambda(\bu)}^2 = \sum_{\lambda\in\sidx}\sum_{\nu\in\pidx} 2^{2s\abs{\lambda}} \abs{\bu_{\lambda\nu}}^2 \sim \int_\pdom\norm{u(y)}_{H^{1+s}(D)}^2\,d\mu(y).
\]
By \cite[Thm.\ 9.1.16]{Hackbusch:92}, using regularity of $f$ and $D$, we have $\norm{u(y)}_{H^{1+s}} \lesssim \norm{f}_{H^{-1+s}}$ uniformly in $y$ for any $s < \alpha/m$.
Here uniformity in $y$ can be seen by inspection of the proof, see \cite{CST:13,TSGU:13}. 
\end{proof}

Let us next illustrate the above estimates by a numerical test for $m=1$ that confirms that in general, this is indeed the best that one can expect.

\begin{example}\label{ex:cfdecay}
We consider $m=1$ with $D=]0,1[$, $\bar a = 1$, $f=1$ and
\[
    \theta_j(x) = c_\alpha 2^{-\alpha\ell} h(2^\ell x - k), \quad  j = 2^\ell + k
\]
for $\ell\geq 0$ and $k = 0,\ldots,2^\ell-1$, where $h(x) = (1 - \abs{2x-1})_+$ and $c_\alpha$
is chosen so as to ensure uniform ellipticity. In other words, the parameter is expanded in a Schauder hat function basis.
As the spatial wavelet basis $\psi_\lambda$, we use piecewise polynomial multiwavelets \cite{DGH:99}.
 Figure \ref{fig:cfdecay} shows the resulting observed decay of the decreasing rearrangements of $\abs{\bu_{\lambda,\nu}}$, $\pi^{(\rs)}_\lambda(\bu)$, $\pi^{(\rp)}_\nu(\bu)$, and of the singular values $\sigma_k(u)$ as in \eqref{Tu} (which satisfy $\sigma_k(u)\sim\sigma_k(\bu)$, where $\sigma_k(\bu)$ are the singular values of $\bu$ as in \eqref{hsapprox}). 
Here, we focus on $0<\alpha \leq 1$. By Proposition \ref{propapproxsummary}, we  expect $\abs{\bu_{\lambda,\nu}}$  in Figure \ref{fig:cfdecay}(a) to decay at approximately the rate $\frac23\alpha + \frac12$,  the values $\pi^{(\rs)}_\lambda(\bu)$, $\pi^{(\rp)}_\nu(\bu)$, and $\sigma_k(u)$ in Figure \ref{fig:cfdecay}(b),(c),(d) to decay at the rate $\alpha + \frac12$.
\end{example}

\begin{figure}
\centering
	\begin{tabular}{cc}
	 \includegraphics[width=6.5cm]{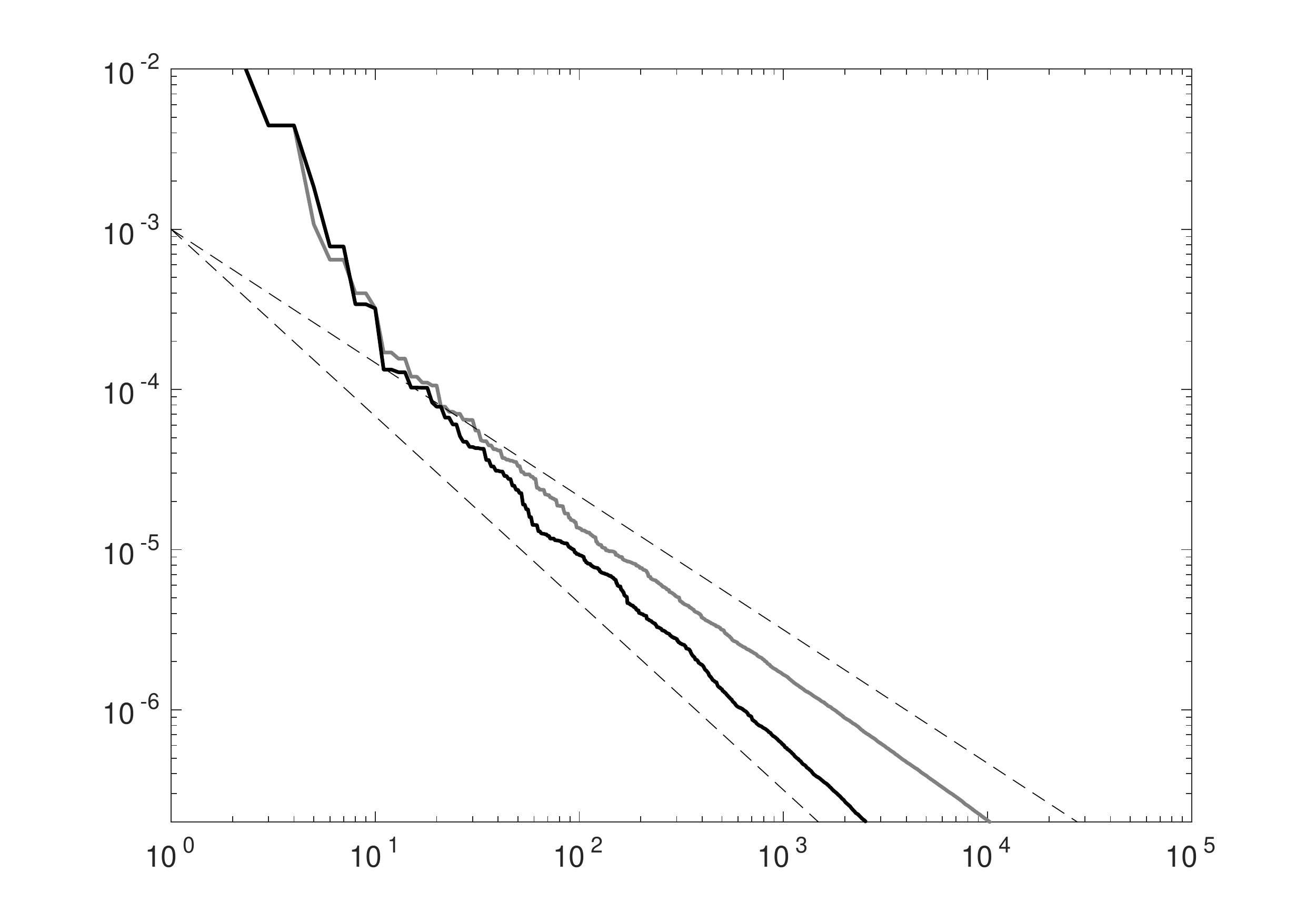}  &
	 \includegraphics[width=6.5cm]{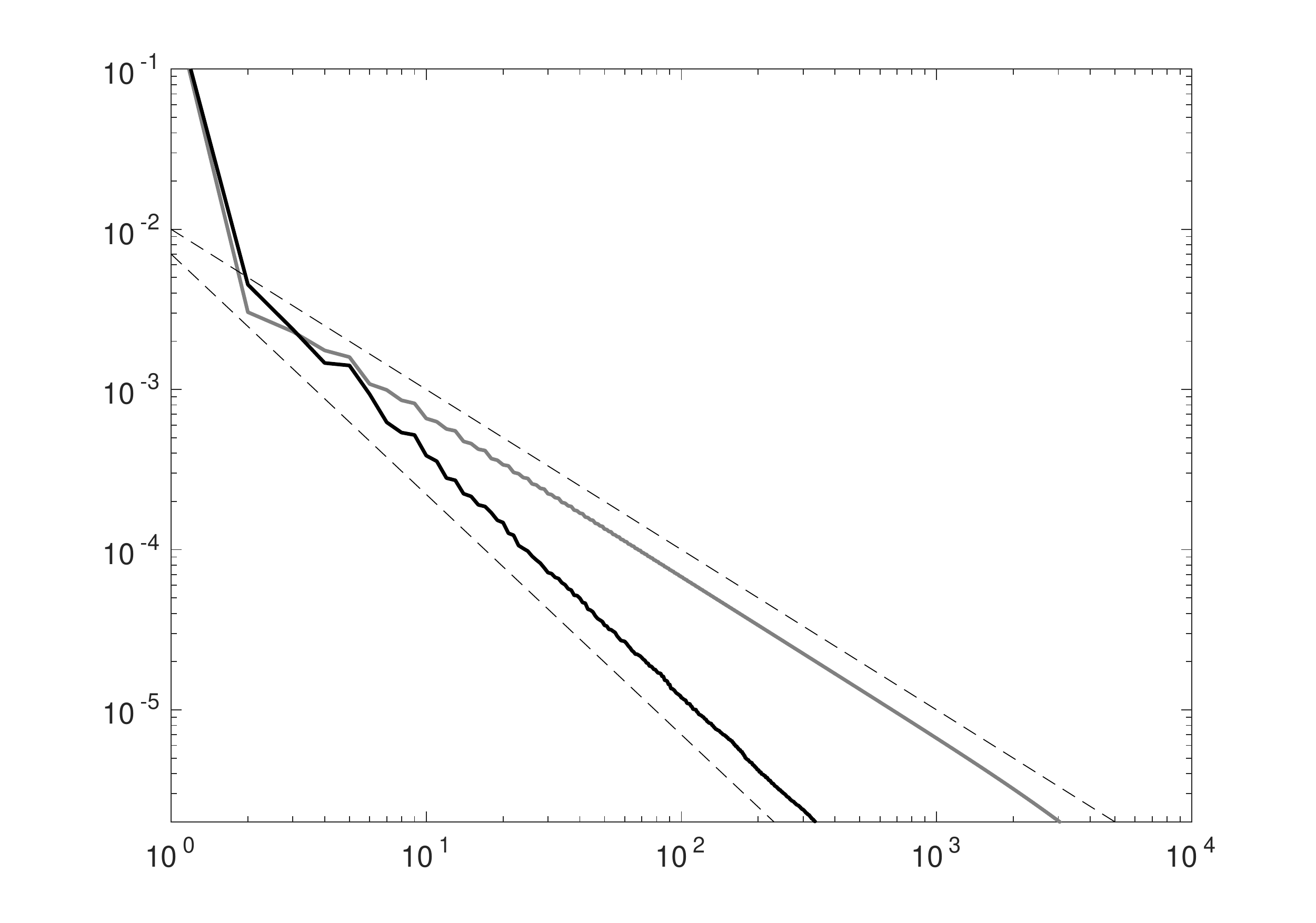} \\
	 \footnotesize(a) ordered $\abs{\bu_{\lambda,\nu}}$ & \footnotesize (b) singular values $\sigma_k$ of $u$ \\
	 \includegraphics[width=6.5cm]{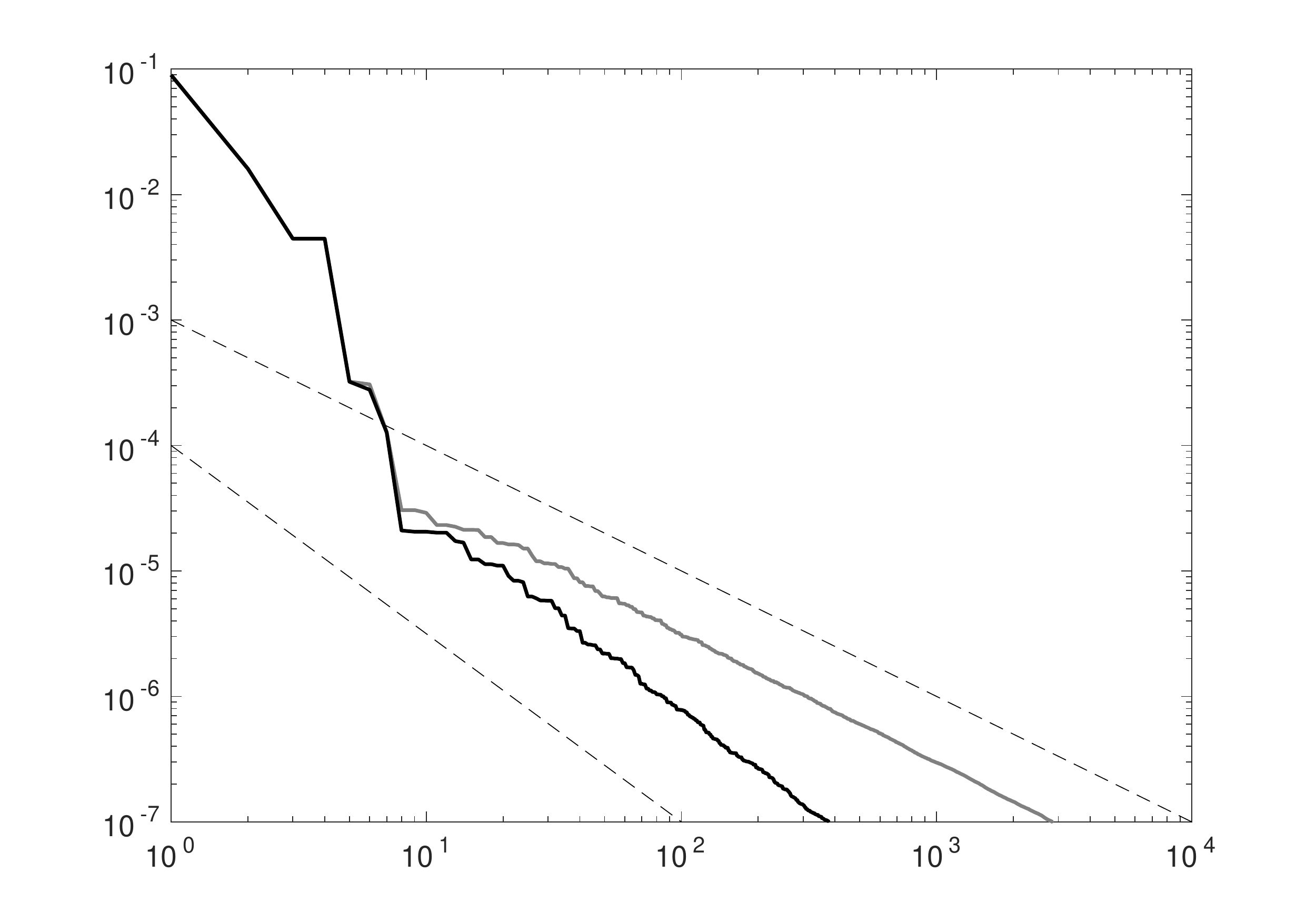} & 
	 \includegraphics[width=6.5cm]{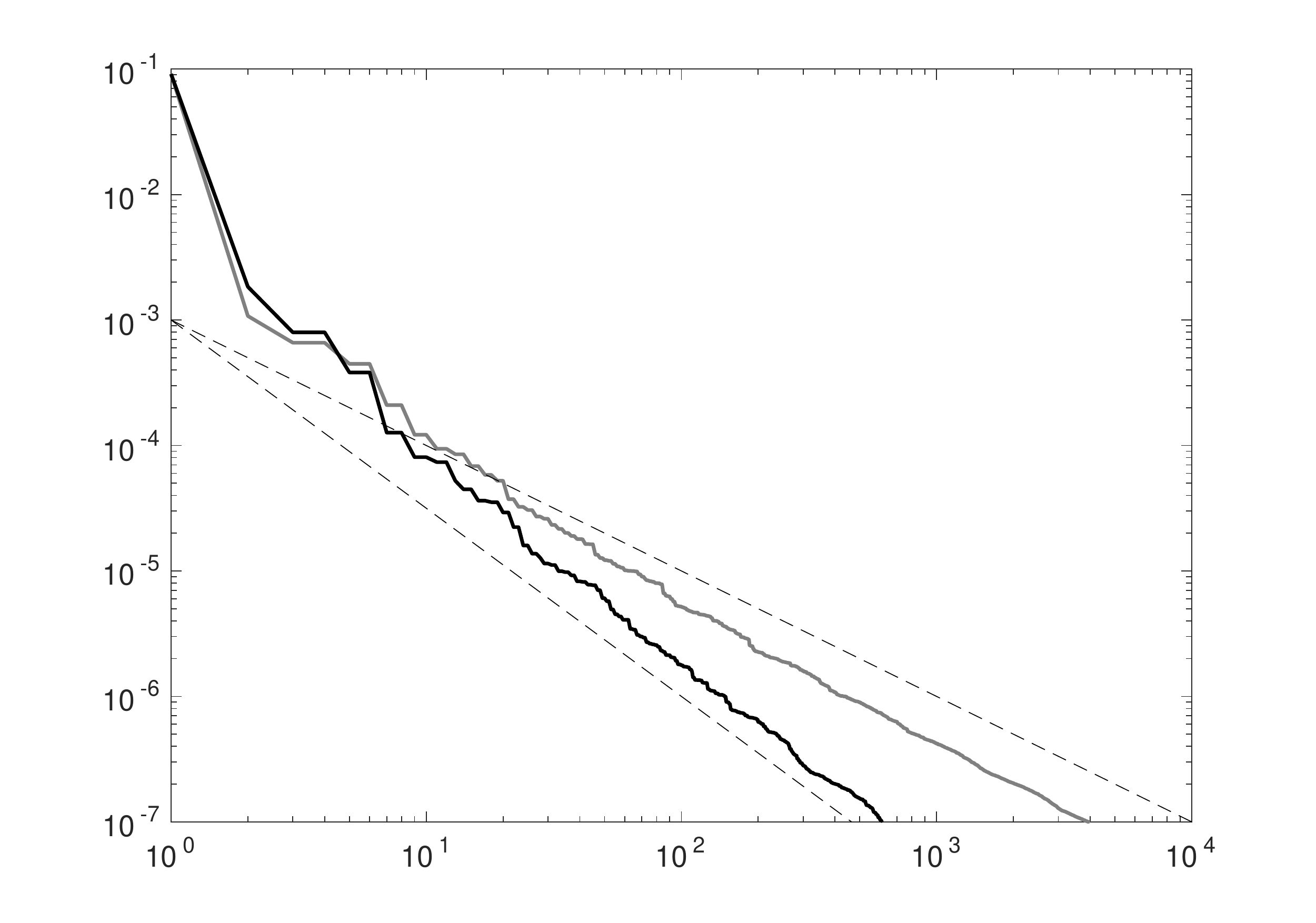} \\
	 \footnotesize (c) ordered $\pi^{(\rs)}_\lambda(\bu)$ & \footnotesize (d) ordered $\pi^{(\rp)}_\nu(\bu)$  
	\end{tabular}
	\caption{Observed decay rates for $\alpha = 1$ (black) and $\alpha=\frac12$ (grey) in Example \ref{ex:cfdecay}. The dashed lines show the respective decay rates expected according to Proposition \ref{propapproxsummary}.}
	\label{fig:cfdecay}
\end{figure}

The numerical tests confirm in particular that the result
\be
\label{spstar}
   s^*_\rs = s^*_\rp = \frac\alpha{m},
\ee
obtained for sufficiently regular problem data, is in general sharp.
We also see that one cannot expect 
$\bu \in \Sigma^{\bar s}$ for any $\bar s$ significantly larger than $\frac\alpha{m}$, and the results indeed suggest the conjecture, similar to what we have shown in a more restricted setting in Proposition \ref{prop:decaylowerbound}, that $(\sigma_k)_{k\in\N} \notin \Sigma^{s}$ for any $s> \max\{s^*_\rs, s^*_\rp\}$.

Moreover, the results in Figure \ref{fig:cfdecay}(a) also demonstrate that in the present case with $m=1$, one indeed only obtains $\bu\in \cA^s(\sidx\times\pidx)$ with $s\approx \frac23\alpha$. In other words, the statement in Proposition \ref{prop:decaylowerbound}(iv), shown in \cite{BCDS:17}, appears to be sharp also for $m=1$. This is a surprising difference to the corresponding results for $m=2,3$ with $s$ up to $\frac\alpha{m}$, which are necessarily sharp.

In contrast parametric expansions with globally supported $\theta_j$ as considered in \S\ref{sec:aniso}, expansions with $\pf_j$ of multilevel type as in Assumptions \ref{ass:multiscale} lead to dimension truncation errors $e_M$ with decay $e_M\lesssim M^{-\alpha/m}$, which matches the approximability properties of $\bu$.
The residual approximation based on $\apply$ also requires compressibility of the matrices $\bA_j$. Under Assumptions \ref{ass:multiscale} and \ref{ass:comprsmoothness} with sufficiently large $\gamma$, the resulting $s^*$-compressibility of $\bA$ comes arbitrarily close to $s^*=\frac\alpha{m}$. Consequently, the complexity of the resulting schemes (both in Theorems \ref{cddresult} and \ref{lrcomplexity}) can come arbitrarily close to the optimal rates determined by the approximability of $\bu$. This is again in contrast to corresponding results for globally supported $\theta_j$, as noted in \S\ref{sec:aniso}.

Finally, we may compare the conclusions of Theorems \ref{cddresult} and \ref{lrcomplexity} for the approximations \eqref{fullnterm} and \eqref{fulllr}, respectively, based on Proposition \ref{propapproxsummary} and Example \ref{ex:cfdecay}.
We first consider the implications of the results for $m=1$ in Figure \ref{fig:cfdecay}, assuming $s^*$ sufficiently close to its limiting value in each case. The scheme for \eqref{fullnterm}, by Theorem \ref{cddresult}, with $n_\text{op}$ operations then converges as $\Ocal(n_\text{op}^{-s})$ for any $s< 2\alpha/3$. By Theorem \ref{lrcomplexity}, the scheme for \eqref{fulllr} converges as $\Ocal(n_\text{op}^{-s})$ for any $s < \alpha/3$.
Thus in this setting, the sparse Legendre expansion \eqref{fullnterm} turns out to be clearly more efficient. By Proposition \ref{propapproxsummary}, one arrives at the same conclusion for $m>1$.

Note that the bound \eqref{hsopbound} for the computation of \eqref{fulllr} is essentially the best that can be expected for any such method that requires orthogonalization of the computed low-rank basis vectors in each step. Besides leading to complexity scaling quadratically in the ranks, this requirement also enforces uniform supports for each set of basis vectors, whereas no such restrictions play any role in the computation of \eqref{fullnterm}.

\section{Summary and conclusions}

In this work, we have studied the approximation of the solution map $Y \ni y \mapsto u(y) \in V$ in $L^2(\pdom,V)$ for parametric diffusion problems, where the parameter domain $Y$ is of high or infinite dimensionality. 
We have considered approximations based on sparse expansions in terms of tensor product Legendre polynomials in $y$, low-rank approximations based on separation of spatial and parametric variables, and higher-order tensor decompositions using further hierarchical low-rank approximation among the parametric variables.

The central aim is to investigate the performance of adaptive algorithms for each type of approximation that require as input only information on the parametric operator and right hand side, and that produce rigorous and computable a posteriori error bounds. 
These goals are achieved, in a unified manner for all considered types of approximations, by Algorithm \ref{alg:tensor_opeq_solve}. Such algorithms are necessarily based on the approximate evaluation of residuals. They are also intrusive, in that they do not treat the underlying parametrized problem as a black box; however, we are not aware of any non-intrusive method with comparable properties.

Although the resulting schemes do not use a priori information on the convergence of the respective approximations of the solution map, they still produce approximations of near-optimal complexity (e.g., with respect to the number of terms or tensor ranks). The question of also guaranteeing a near-optimal operation count for constructing these approximations is more delicate: this computational complexity depends on the costs of approximating the residual, and thus on the approximability properties of the operator. In the case of low-rank approximations, due to the required orthogonalizations, the number of operations also scales at least quadratically with respect to the arising tensor ranks.

Especially keeping the latter point in mind, there is no single type of approximation that is most favorable in all of the representative model scenarios that we have considered. In the case of finitely many parameters of comparable influence, hierarchical tensor representations of $u$ turn out to be advantageous: We can show near-optimal computational complexity on certain natural approximability classes (as in Assumptions \ref{ass:fullbenchmark}, \ref{ass:fullbenchmark dim}) for the adaptive scheme based on the method in \cite{BD}.

The situation turns out to be different in the case of infinitely many parameters of decreasing influence. We have proven in \S\ref{sec:approximability}, for a certain class of such problems, that the norms of Legendre coefficients of $u$ have the same asymptotic decay as the singular values in its Hilbert-Schmidt decomposition. In other words, the ranks in a corresponding low-rank approximation need to increase at the same rate as the number of terms in a sparse Legendre expansion as accuracy is increased. 
The numerical tests given in Figure \ref{fig:cfdecay} indicate that this holds true also for substantially more general problems.  As a consequence, even with the careful residual evaluation given in \S\ref{sec:lowrank}, which can preserve near-optimal ranks, due to the nonlinear scaling with respect to the ranks the computational complexity of finding low-rank approximations scales worse than a direct sparse expansion as considered in \S\ref{sec:sparselegendre}. This conclusion remains true also for hierarchical tensor decompositions involving the same separation between spatial and parametric variables.

For both schemes in \S\ref{sec:sparselegendre} and \S\ref{sec:lowrank}, we have seen that whether the residual can be evaluated at a cost that matches the approximability of the solution depends on the type of parameter-dependence in the diffusion coefficient. As the simple example given in \S\ref{sec:aniso} shows, in the case of diffusion coefficients expanded in terms of increasingly oscillatory functions of global support, due to insufficient operator compressibility, the complexity of the methods is in general worse than the approximability of the solution would allow.
However, in the case of diffusion coefficients whose parametrization has a multilevel structure, we have demonstrated that one can come arbitrarily close to fully exploiting the approximability of $u$.

\bibliography{bcd_parametric_r1}

\begin{appendix}

\section{Compressibility of parametric operators}\label{sec:opapprox}

The approximate application of the operator $\bA$ in Algorithm \ref{alg:tensor_opeq_solve} must involve, in particular, an approximate
application of the spatial components $\bA_j$.
Except for special cases, the infinite matrices $\bA_j$ are not sparse, but contain infinitely many nonzero entries in each column. 
Their approximation hinges on the compressibility of these operator representations as  in Proposition \ref{prop:sparsecompr}.

These are closely related to $s^*$-compressibility of $\bA_j$ as in \eqref{compress}, which here means that there exist matrices $\bA_{j,n}$ with $\alpha_{j,n} 2^n$ entries per row and column and such that 
\beqn
\label{compressAj}
\norm{\bA_j - \bA_{j,n}} \leq \beta_{j,n} 2^{-s n}, \quad \text{ for $0 < s < s^*$},
\eeqn 
and where $\ab_j, \bb_j \in \spl{1}(\N_0)$. This is known  to hold  for each fixed $j$ when employing a piecewise polynomial wavelet-type Riesz basis $\{\psi_\lambda\}_{\lambda \in \mathcal{S}}$ for $V$, see e.g. \cite{Cohen:01,Stevenson:02}.
However, when insisting on the same compressibility bound $s^*$ for all $\bA_j$, the quantities $\|\ab_j\|_{\spl{1}}$ and $\|\bb_j\|_{\spl{1}}$
can in general not be expected to both remain uniformly bounded in $j$ when the $\pf_j$ become increasingly oscillatory.

We consider operators $\bA_j$ arising from multilevel representations of the parameter of the form in Assumptions \ref{ass:multiscale}. In the spatial variable, we use a wavelet Riesz basis $\{ \psi_\lambda \}_{\lambda\in\sidx}$, which yields compressible $\bA_j$. Their compressibility is governed by  the modulus of the entries $\langle \pf_j \nabla \psi_\lambda, \nabla \psi_{\lambda'}\rangle$, where $\langle\cdot,\cdot\rangle$ denotes the $L^2$-inner product. Specifically, recall e.g.\ from \cite{Cohen:01} that compression strategies for
wavelet representations of an elliptic second-order operator with diffusion field $c$ are based on bounds of the type
\beqn
\label{reference}
|\langle c\nabla \psi_\lambda,\nabla \psi_{\lambda'}\rangle|\lesssim \|c\|_{W^{b- m/2}(\spL{\infty}(D))}
 2^{-\abs{|\lambda|-|\lambda'|}b} ,
\eeqn
where $m$ is the dimensionality of the spatial domain, and
where $b> m/2$ depends on the smoothness of the diffusion coefficient $c$ and of the wavelets $\psi_\lambda$.

However, in our case the higher-order norms of $c$ on the right hand side of \eqref{reference} with $c=\theta_j$ depend on $j$, and the overall compression rate is also limited by the decay of the operator truncation error \eqref{eM}. In view of Proposition \ref{prop:sparsecompr}, the objective here is thus to have a compression rate for the individual components $\bA_j$ that is as high as possible, so that one approaches the limiting value imposed by \eqref{eM}.

We now summarize the conditions on the multilevel parametric expansion functions and the spatial wavelet basis under which we will verify the requirements of Proposition \ref{prop:sparsecompr}. To simplify notation, let $S_\lambda \coloneqq \supp \psi_\lambda$.

\begin{assumptions}\label{ass:comprsmoothness}
With $\{\xi_\mu\}_{\mu \in \Lambda}$ as in Assumptions \ref{ass:multiscale}, for some $\gamma > 0$, 
\begin{equation}\label{productsmoothness}
	   \xi_\mu \nabla\psi_{\lambda'}\in H^{\gamma}(S_\lambda), \quad \mu \in \Lambda, \; \lambda,\lambda'\in\sidx, 
\end{equation}	   
and the $\psi_\lambda$ have vanishing moments of order $k$ with $k>\gamma-1$. 
\end{assumptions}

Note that the $\nabla\psi_\lambda$ then have vanishing moments of order $k+1>\gamma$.
If $\abs{\lambda},\abs{\mu} \leq \abs{\lambda'}$, using \eqref{productsmoothness} we obtain the standard estimate
\beqn
\label{vanmom}
\abs{\langle \xi_\mu \nabla \psi_{\lambda}, \nabla \psi_{\lambda'} \rangle}
  \le \inf_{P\in \Pi_{k+1}^{m}}\|\xi_\mu\nabla\psi_{\lambda} - P\|_{\spL{2}(S_{\lambda'})} \norm{\psi_{\lambda'}}_{\spL{2}}
\lesssim 2^{-|\lambda'|\gamma} \abs{ \xi_\mu \nabla\psi_{\lambda} }_{H^{\gamma}(S_{\lambda'})}.
\eeqn
Combining this with
$  \abs{ \xi_\mu \nabla\psi_{\lambda} }_{H^{\gamma}(S_{\lambda'})} 
    \lesssim 2^{- \frac{m}2 \abs{\abs{\lambda}-\abs{\lambda'}}} 2^{ \gamma \max\{\abs{\mu},\abs{\lambda}\}}  $,
we obtain
\beqn
\label{entryestimate}
  \abs{\langle \xi_\mu \nabla \psi_{\lambda} , \nabla \psi_{\lambda'} \rangle}
   \lesssim 2^{-(\gamma + \frac{m}2) \abs{\abs{\lambda}-\abs{\lambda'}}} 2^{\gamma(\abs{\mu}-\abs{\lambda})_+}.
\eeqn
Note that the requirement \eqref{productsmoothness} could be weakened along the lines of \cite{Stevenson:02} to \emph{piecewise} smoothness, in which case combinations of wavelets with overlapping singular supports need to be considered separately. Since this is not essential for our purposes, to keep the exposition accessible we do not consider this in further detail.

The consequences of the estimate \eqref{entryestimate} depend on the relations between $\abs{\mu}$, $\abs{\lambda}$, and $\abs{\lambda'}$. We distinguish the three following cases:

If $ \abs{\mu} \leq \abs{\lambda}, \abs{\lambda'}$,
we obtain an estimate analogous to the standard case \eqref{reference},
\beqn
 \abs{\langle \xi_\mu \nabla \psi_\lambda, \nabla \psi_{\lambda'} \rangle}
 \lesssim 2^{-(\gamma + \frac{m}2) \abs{\abs{\lambda}-\abs{\lambda'}}}.
\eeqn

If $ \abs{\lambda} \leq \abs{\mu} < \abs{\lambda'} $, we obtain the modified estimate
\beqn\label{sandwichcompression}
   \abs{\langle \xi_\mu \nabla \psi_\lambda, \nabla \psi_{\lambda'} \rangle}
 \lesssim 2^{- \gamma (\abs{\lambda'}-\abs{\mu})} 2^{- \frac{m}2 \abs{\abs{\lambda}-\abs{\lambda'}} }.
\eeqn
Note that for each fixed $\mu$ and fixed levels $\abs{\lambda}, \abs{\lambda'}$, there exist in this case $\Ocal(2^{m(\abs{\lambda'}-\abs{\mu})})$ entries that may be nonzero.

Finally, if $ \abs{\lambda}, \abs{\lambda'} \leq \abs{\mu}$, then for each $\mu$, there exist $\abs{\mu}$ indices $\lambda$ such that the corresponding supports overlap, and in turn there exist $\Ocal(\abs{\mu}^2)$ pairs of $\lambda,\lambda'$ that may give a nonvanishing entry. These entries satisfy
\begin{equation}\label{slowcompression}
  \abs{ \langle \xi_\mu \nabla\psi_\lambda, \nabla\psi_{\lambda'}\rangle }
     \lesssim 2^{-m\abs{\mu}} 2^{\frac{m}{2}(\abs{\lambda}+\abs{\lambda'})}.
\end{equation}
Note that we do not assume any vanishing moments for $\xi_\mu$. Hence in general not much can be gained by discarding further entries in this third case.

Our strategy for dealing with the increasingly oscillatory nature of $\xi_\mu$ as $\abs{\mu}\to\infty$ is to retain a common  compression rate $s^*$ in \eqref{compressAj} uniformly in $\mu$ without losing the decay induced by the factors $c_\mu$.
To this end, we retain additional entries of the $\bA_j$ in the cases \eqref{sandwichcompression} and \eqref{slowcompression}. This results in the $j$-dependent number of nonzero entries  in each row and column of the compressed operators $\bA_{j,n}$, which is of order $\Ocal((1+\abs{\mu_j}^q) 2^n)$. 

Let $a_{\mu_j,\lambda,\lambda'}$ denote the entries of $\bA_{j}$, that is,
\[
   a_{\mu_j,\lambda,\lambda'} = c_{\mu_j} \langle \xi_{\mu_j} \nabla \psi_{\lambda},\nabla\psi_{\lambda'}\rangle.
\]

\begin{proposition}\label{prop:wvcompr}
	Under Assumptions \ref{ass:comprsmoothness}, 
 $\bA_{j,n}$ defined for $n\in\N$ by retaining only those entries from $\bA_j = (a_{\mu_j,\lambda,\lambda'})_{\lambda,\lambda'\in\sidx}$ for which
\[
  d_{\mu_j}(\lambda,\lambda'):= \max\bigl\{ \abs{\lambda},\abs{\lambda'}\bigr\} - \max\bigl\{ \abs{\mu_j}, \min\{ \abs{\lambda},\abs{\lambda'} \} \bigr\} 
  \leq  \frac nm + \frac{ \log_2 (1+\abs{\mu_j})}{\gamma},
\]	
and where we set $\bA_{j,0} = 0$,
satisfy the following conditions:
\begin{enumerate}[{\rm(i)}]
\item
With $\tau \coloneqq \gamma /m$, one has $
\|\bA_j - \bA_{j,n}\|\lesssim c_{\mu_j}  2^{-\tau n}$, $n\in\N$,
where the hidden constant is independent of $j,n$.
\item  The number of nonvanishing entries in each column of $\bA_{n,j}$ does not exceed
a uniform constant multiple of $\bigl(1+|\mu_j|^q\bigr)2^n$, for $q\coloneqq \max\{1, \tau^{-1} \}$.
\end{enumerate}
\end{proposition}

\begin{proof}
Note first that the choice $\bA_{j,0}=0$ satisfies the stated conditions because $\norm{\bA_j}\lesssim c_{\mu_j}$, and it remains to consider $n>0$.
For $j\in\N$, we set $\mu \coloneqq \mu_j$.
In a first step, for $N>0$, we obtain a compressed version $\bA_{j}^{N}$ of $\bA_j$ as follows: for the column $\lambda$, retain only those entries with row index $\lambda'$ such that 
$d_{\mu}(\lambda,\lambda') \leq N$.
Note that by symmetry of $d_\mu$ in its two arguments and that of $\bA_j$, the approximation $\bA_{j}^{N}$ is also  symmetric.
We now show that for 
\begin{equation}\label{Nnchoice}
N=  N_n := \frac nm + \frac{ \log_2 (1+\abs{\mu})}{\gamma},
\end{equation}  
we arrive at the statement.
We use the standard weighted Schur Lemma, which in the present symmetric case yields that
\beqn\label{schur}
\omega_{\lambda}^{-1}\sum_{\lambda'\colon d_{\mu}(\lambda,\lambda')>N}\omega_{\lambda'} \abs{a_{\mu,\lambda,\lambda'}} \le B, \quad 
\lambda \in \mathcal{S} ,\qquad \text{ implies}\quad 
  \|\bA_j - \bA_{j}^{N}\| \le B.
\eeqn
Note that $d_{\mu}(\lambda,\lambda')>0$ implies that $\abs{\lambda}>\abs{\mu}$ or $\abs{\lambda'}>\abs{\mu}$. Thus, as a particular consequence of \eqref{entryestimate}, if $d_{\mu}(\lambda,\lambda')>0$ we have
\beqn
  \abs{a_{\mu,\lambda,\lambda'}} \lesssim c_{\mu}  2^{-\gamma d_\mu(\lambda,\lambda')} 2^{- \frac{m}2 \abs{\abs{\lambda}-\abs{\lambda'}} }.
\eeqn
With the usual choice $\omega_\lambda := 2^{-\frac{m}2 |\lambda|}$, and setting 
\[
 I(\lambda;N):= \{ \lambda'\colon d_{\mu}(\lambda,\lambda')>N \},
\]
we obtain
\[ 
 \omega_{\lambda}^{-1}\sum_{\lambda'\in I(\lambda;N)}\omega_{\lambda'} \abs{a_{\mu,\lambda,\lambda'}} \lesssim 
   c_{\mu} \sum_{\lambda' \in I(\lambda;N)} 
   2^{-m(\abs{\lambda'} -\abs{\lambda})_+}
   2^{- \gamma d_\mu(\lambda,\lambda')} .
\]
We now decompose $I(\lambda;N) = I_1 \cup I_2 \cup I_3 \cup I_4$, where
\begin{align*}
 I_1 &:= \{ \lambda'\in I(\lambda;N)\colon \abs{\lambda'}\leq \abs{\mu}<\abs{\lambda}\}, &
  I_2 &:= \{ \lambda'\in I(\lambda;N)\colon \abs{\lambda}\leq \abs{\mu}<\abs{\lambda'}\},\\
 I_3 &:= \{ \lambda'\in I(\lambda;N)\colon  \abs{\mu}<\abs{\lambda'} \leq \abs{\lambda}\}, &
  I_4 &:= \{ \lambda'\in I(\lambda;N)\colon  \abs{\mu}<\abs{\lambda} < \abs{\lambda'}\}.
\end{align*}
Since  $\#(I_1)\lesssim 1 + |\mu|$,
\beqn\label{I1est}
  \sum_{\lambda'\in I_1} 2^{-m(\abs{\lambda'} -\abs{\lambda})_+}
   2^{-\gamma d_\mu(\lambda,\lambda')}
    \lesssim (1+\abs{\mu}) 2^{-\gamma N}. 
\eeqn
Likewise, we obtain  the   estimates
\begin{align*}
 \sum_{\lambda'\in I_2} 2^{-m(\abs{\lambda'} -\abs{\lambda})_+}
   2^{-\gamma d_\mu(\lambda,\lambda')} 
  & \lesssim \sum_{\ell=\abs{\mu}+N}^\infty \sum_{\substack{\lambda'\in I_2 \\ \abs{\lambda'}=\ell }} 2^{-m(\abs{\lambda'}-\abs{\lambda})} 2^{-\gamma \ell}   \\
   &\lesssim  \sum_{\ell=\abs{\mu}+N}^\infty  2^{-\gamma\ell} \bigl( 2^{m(\ell - \abs{\mu})} 2^{-m(\ell-\abs{\lambda})}  \bigr) \\
   &\lesssim 2^{- \gamma N} 
\end{align*}
and
\[
   \sum_{\lambda'\in I_3} 2^{-m(\abs{\lambda'} -\abs{\lambda})_+}
   2^{- \gamma d_\mu(\lambda,\lambda')}
    \lesssim \sum_{\ell = \abs{\mu}}^{\abs{\lambda} - N} 2^{-\gamma (\abs{\lambda}-\ell)} \lesssim 2^{-\gamma N},
\]
as well as
\begin{align*}
  \sum_{\lambda'\in I_4} 2^{-m(\abs{\lambda'} -\abs{\lambda})_+}
   2^{- \gamma d_\mu(\lambda,\lambda')}
    &  \lesssim 
     \sum_{\ell = \abs{\lambda} + N}^\infty 2^{-\gamma \ell} \sum_{\substack{\lambda'\in I_4 \\ \abs{\lambda'}=\ell }} 2^{-m(\ell-\abs{\lambda})}  \\
     & \lesssim \sum_{\ell = \abs{\lambda} + N}^\infty 2^{-\gamma\ell}
       \bigl( 2^{m(\ell - \abs{\lambda})} 2^{-m(\ell-\abs{\lambda})}  \bigr) \\
       &\lesssim 2^{-\gamma N}.
\end{align*}
Note that $\#(I_3)\lesssim N$ and $\#(I_2), \#(I_4)\lesssim 2^{mN}$.
Except for \eqref{I1est}, the constants in these bounds are independent of $\mu$.
 In summary, we thus obtain
\be
\label{AjN}
    \| \bA_j - \bA_{j}^{N} \| \lesssim c_{\mu} (1 + \abs{\mu})\, 2^{-\gamma N}
\ee
with a uniform constant. As pointed out above, each column of $\bA_j^N$ has at most $\Ocal(\abs{\mu} + 2^{mN})$ entries. 
With $\tau = \gamma /m$ and $N_n$ as in \eqref{Nnchoice}, the estimate
\eref{AjN} takes the desired form 
\be
\label{P1}
 \| \bA_j - \bA_{j}^{N_n} \| \lesssim c_{\mu} 2^{-\tau n},\
\ee
where the number of nonzero entries can be bounded further by 
\be
\label{P2}
 |\mu|+2^{mN_n} \lesssim |\mu| + 2^n (1+|\mu|)^{\frac{m}{\gamma}} \lesssim \bigl(1+|\mu|^{\max\{1,m/\gamma\}} \bigr)2^n,
\ee
which was to be shown.
\end{proof}

Relations \eref{P1}, \eref{P2} show that the resulting compression rate is limited by the smoothness of the expansion functions $\xi_\mu$ and the spatial wavelets $\psi_\lambda$, as well as by the number of vanishing moments of the $\psi_\lambda$, expressed by the value $\gamma$. As Proposition \ref{prop:sparsecompr} shows, with increasing $\gamma$ the rate of compressibility of the complete operator $\bA$ approaches the limiting value determined by the decay of its tail \eqref{eM}.
 
\begin{remark}
	\label{rem:badcompr}
Proposition \ref{prop:wvcompr} yields, as we have also noted in \S\ref{sec:aniso}, a compressibility result for multilevel-type parametrizations that is substantially more favorable than what can in general be obtained for globally supported, increasingly oscillatory $\theta_j$.
In the case $\theta_j \sim j^{-\beta} \sin(j\pi \cdot)$ on $D=]0,1[$ considered in \S\ref{sec:aniso}, in place of \eqref{entryestimate} we obtain the analogous bound
\[
   |\langle \theta_j \psi_\lambda',  \psi_{\lambda'}' \rangle|
       \lesssim  j^{-\beta} 2^{-(\gamma + \frac12) \abs{\abs{\lambda}-\abs{\lambda'}}} 2^{\gamma(\log_2 j - \abs{\lambda})_+}.
\]
One may thus proceed as in the proof of Proposition \ref{prop:wvcompr}, with $\abs{\mu}$ replaced by $\log_2 j$, to obtain $\bA_{j,n}$ such that 
\[
   \norm{\bA_j -\bA_{j,n}} \lesssim j^{-\beta} 2^{-\gamma n} . 
\]
However, among the pairs of indices $(\lambda,\lambda')$ with $\abs{\lambda} \leq \log_2 j$, we are eventually left with $\Ocal( j (1 + \log_2 j) 2^n)$ entries per row and column. 

These bounds yield a compressibility result for $\bA$ similarly to Proposition \ref{prop:sparsecompr}. In the present case we have, for $\bA_{\mathbf{n}}$ as defined in \eqref{defAn}, the simpler estimate
\[
\norm{\bA - \bA_\mathbf{n}} \lesssim   \sum_{j=0}^{M} \bignorm{ \bA_j - \bA_{j,n_j} }  +  M^{-(\beta-1)} \lesssim  2^{-\gamma n_0} + \sum_{j=1}^{M}  j^{-\beta} 2^{-\gamma n_j}   + M^{-(\beta-1)}.
\]
Choosing $n_j$ appropriately to ensure that the right hand side is of order $M^{-(\beta-1)}$  and summing the resulting total numbers of nonzero entries, as in \cite{Gittelson:14} one arrives at the limiting value $s^* = \frac12 (\beta -1)$ for the compressibility of $\bA$.
\end{remark}

\section{Proofs of auxiliary results}
	\label{app:thm3}
	
\begin{proof}[Proof of Theorem \ref{lmm:combined_coarsening}]
The estimates \eqref{weta}, \eqref{eq:combinedcoarsen_errest} are obtained exactly as in \cite{BD}.
To prove \eqref{eq:combinedcoarsen_suppest} we follow the lines of the argument in \cite{BD}, and adopt the notation used there,    
let $N= N(\eta)\in \N$ be the minimal integer such that
$\|\bu - \bar C_{\bu,N}\bu\| \le \alpha \eta$.
Then
\begin{align}
\label{alphaeta1}
\alpha\eta  & < \|\bu - \bar C_{\bu,N-1}\bu\| \\
& \le \inf_{\#\Lambda_\rs + \#\Lambda_\rp \le N-1}\Big\{\| \pi^{(\rs)}(\bu) - \Restr{\Lambda_\rs}\pi^{(\rs)}(\bu)\|
+ \| \pi^{(\rp)}(\bu) - \Restr{\Lambda_\rp}\pi^{(\rp)}(\bu)\|\Big\}\nonumber\\
&\le  (\#\Lambda_\rs)^{-s_\rs}\|\pi^{(\rs)}(\bu)\|_{\Acal^{s_\rs}}+  (\#\Lambda_\rp)^{-s_\rp}\|\pi^{(\rp)}(\bu)\|_{\Acal^{s_\rp}}.
\end{align}
Abbreviating $n_i := \#\Lambda_i$, $i=\rs,\rp$, to obtain a good upper for bound $N$, we would like to
find the minimal $n_\rs+n_\rp$ such that
\beqn
\label{alphaeta}
\alpha\eta \le (n_\rs)^{-s_\rs} \|\pi^{(\rs)}(\bu)\|_{\Acal^{s_\rs}}+  (n_\rp)^{-s_\rp}\|\pi^{(\rp)}(\bu)\|_{\Acal^{s_\rp}},
\eeqn
to conclude that $N(\eta)\le n_\rs+n_\rp$. 
Equilibrating the upper bound yields a pair $n_\rs, n_\rp$ given by
\beqn
\label{ni}
n_i = n_i(\eta) := \Big\lceil \Big(2\|\pi^{(i)}(\bu)\|_{\Acal(\gamma_i)}/\alpha\eta\Big)^{1/s_i} \Big\rceil, \quad i=\rs,\rp,
\eeqn
This yields
\begin{equation*}
\#\supp_\rs\bw_\eta +\#\supp_\rp\bw_\eta  \le 2+  \biggl( \frac{2\|\pi^{(\rs)}(\bu)\|_{\Acal^{s_\rs}}}{\alpha\eta}\biggr)^{1/s_\rs} +
 \biggl( \frac{2\|\pi^{(\rp)}(\bu)\|_{\Acal^{s_\rp}}}{\alpha\eta} \biggr)^{1/s_\rp},
\end{equation*}
which is the first inequality in \eqref{eq:combinedcoarsen_suppest}.

Regarding the second inequality   in \eqref{eq:combinedcoarsen_suppest}, note first that
$N\le B_i n_i$, $i=\rs,\rp$,
where $B_i$ depend only on $s_\rs,s_\rp$. To bound $\|\pi^{(i)}(\bw_\eta)\|_{\Acal^{s_i}}$ we only need 
to estimate 
\[\sup_n n^{s_i} \inf_{\#\supp \hat\bw \leq n}\norm{\hat\bw-\pi^{(i)}(\bw_\eta)}, \quad i=\rs,\rp, \]
for $n \le \#\supp_i \bw_\eta\le N$.
To that end, denoting by $\hat \bu^{(i)}_n$ a best $n$-term approximation to $\pi^{(i)}(\bu)$ and using \eqref{eq:combinedcoarsen_rankest},
we obtain
\begin{align*}
\inf_{\#\supp \hat\bw \leq n}\norm{\hat\bw-\pi^{(i)}(\bw_\eta)} & \le \|\pi^{(i)}(\bw_\eta) - \pi^{(i)}(\bu)\|+ \|\pi^{(i)}(\bu)- \hat \bu^{(i)}_n\|\\
 &\le \|\bw_\eta - \bu\| +  n^{-s_i}\|\pi^{(i)}(\bu)\|_{\Acal^{s_i}}\\
 &\le C(\alpha)\eta +  n^{-s_i}\|\pi^{(i)}(\bu)\|_{\Acal^{s_i} }\nonumber\\
 &\le  \frac{2C(\alpha)}{\alpha}n_i^{-s_i}\|\pi^{(i)}(\bu)\|_{\Acal^{s_i} }   + n^{-s_i}\|\pi^{(i)}(\bu)\|_{\Acal^{s_i} },
 \end{align*}
 where we have used \eqref{ni} and where $C(\alpha):= \bigl( 2+\alpha +  2^{3/2} (1+\alpha)\bigr)$.
Hence
\begin{multline*}
 n^{s_i} \inf_{\#\supp \hat\bw \leq n}\norm{\hat\bw-\pi^{(i)}(\bw_\eta)} \\ 
 \le \biggl(1+\frac{2C(\alpha)}{\alpha}\biggl(\frac{n}{n_i}\biggr)^{s_i} \biggr)\|\pi^{(i)}(\bu)\|_{\Acal^{s_i}} 
 \le \biggl(1+\frac{2C(\alpha)B_i^{s_i}}{\alpha}\biggr)\|\pi^{(i)}(\bu)\|_{\Acal^{s_i}} ,
\end{multline*}
which completes the proof.
\end{proof}

\begin{proof}[Proof of Proposition \ref{hsworkest}]
As we assume $\bv$ to be given in SVD form, $\recompress$ in step (S1) of the procedure $\apply$ takes only $\Ocal(r)$ operations. Since it preserves the SVD form, the subsequent $\coarsen$ using quasi-sorting takes $\Ocal(r(n_\rs + n_\rp))$ operations (with the computation of the contractions as the dominating contribution).

In computing the quantities  $\norm{\bv_{p,q}}$ and $\pi^{(\rs)}_\nu(\bv_{[p,q]})$ in steps (S2) and (S3), we need to take into account that the vectors $\Restr{\Lambda^{(\rp)}_{q}}  \bU^{(\rp)}_k$, $k \in K_p$, need no longer be orthonormal.

To this end, let $\mathbf{V}\in \R^{2^q \times 2^p}$ denote the matrix with columns $\mathbf{V}_k := \sigma_k \Restr{\Lambda^{(\rp)}_{q}}  \bU^{(\rp)}_k$, and let 
$ \mathbf{\hat u}_\nu = (\bU^{(\rs)}_{k,\nu})_{k\in K_p}\in \R^{2^p}$.
If $q \geq p$, we compute the Gramian $\mathbf{V}^T\mathbf{V}$, which takes $\Ocal(2^{2p + q} )$ operations. We then directly obtain $\norm{\bv_{[p,q]}}^2 = {\operatorname{tr}(\bV^T\bV)}$. Moreover, for each given $\nu$ we can evaluate 
 \[  \abs{\pi^{(\rs)}_\nu(\bv_{[p,q]})}^2 = \mathbf{\hat u}_\nu^T (\mathbf{V}^T\mathbf{V}) \mathbf{\hat u}_\nu  
 \] 
using $\Ocal(2^{2p})$ operations.
If $p > q$, we first factorize $\mathbf{V}^T = \mathbf{Q} \mathbf{R}$, where $\mathbf{Q} \in \R^{2^p\times 2^q}$ has orthonormal columns and $\mathbf{R} \in \R^{2^q\times 2^q}$. This takes $\Ocal(2^{p + 2q})$ operations. In addition, we form $\mathbf{R}\mathbf{R}^T$ using $\Ocal(2^{3q})$ operations. We then have $\norm{\bv_{[p,q]} }^2 = \operatorname{tr}(\mathbf{R}\mathbf{R}^T)$ and for each $\nu$, we can evaluate $\mathbf{\hat u}_\nu^T \mathbf{Q}$  and subsequently $\abs{\pi^{(\rs)}_\nu(\bv_{[p,q]})}^2 = (\mathbf{\hat u}_\nu^T \mathbf{Q})(\mathbf{R}\mathbf{R}^T)(\mathbf{\hat u}_\nu^T \mathbf{Q})^T$ using $\Ocal(2^{p + q} + 2^{2q})$ operations.

Altogether,  abbreviating $r_\eta \coloneqq \rank(\bv_\eta)$ and $n_{\eta,\rp} \coloneqq \#\supp_\rp(\bv_\eta)$,
the computational work required for obtaining $\norm{\bv_{[p,q]} }$ and $\abs{\pi^{(\rs)}_\nu(\bv_{[p,q]})}$ is of order 
$ (n_{\rs} + n_{\eta,\rp}) r_\eta^2 \leq
   (n_\rs + n_\rp) r^2$.

With these values at hand, it remains to assemble $\bw_\eta$ in the form \eqref{weta assembly}, which amounts to building each $\bw_{p,q}$ as in \eqref{wpq}. The action of the bidiagonal matrices $\bM_j$, on the one hand, for each $p,q$ and $j$ requires $2^{p+q}$ operations, and the total costs for assembling the $\rp$-components of the result are therefore bounded up to a constant by
\begin{align*}
	  \sum_{p,q\geq 0} 2^{p+q} M_{p,q}  &\lesssim   \eta^{-\frac 1S}\|\bv\|_{\Sb}^{\frac 1S}(1+\log_2(n_{\eta,\rp}))^{\frac aS}n_{\eta,\rp} 
	  (1+\log_2(r_\eta))^{\frac aS}(r_\eta)^{1- \frac{\bar s}S}  \\
	&\lesssim   \eta^{-\frac 1{\bar s}}\|\bv\|_{\Sb}^{\frac 1{\bar s}}(1+\abs{\log \eta} )^{\frac {2a}S} \eta^{-\frac1{s_\rp}} \norm{\pi^{(\rp)}(\bv)}^{\frac1{s_\rp}}_{\Acal^{s_\rp}},
\end{align*}
where the estimate on the right is obtained as in \eqref{pestimate} and \eqref{rankfirst}.
Assembling the $\rs$-components requires the action of the approximate operators $\tilde\bA_{p,q,j}$. By our construction, the combined action of $\tilde\bA_{p,q,j}$, $j=1,\ldots,M_{p,q}$, on a single vector $\bU^{(\rs)}_k$, $k\in K_p$, takes a number of operations proportional to the resulting $\#\supp_\rs (\bw_{p,q})$. Consequently, the total number of operations for the $\rs$-components is bounded up to a constant by
\begin{align*}
	 \sum_{p,q\geq 0} 2^p \#\supp_\rs (\bw_{p,q}) 
	  &\lesssim \sum_{p,q\geq 0} 2^p (1 + p )^{\frac{a}{s_\rs}} (1 + q )^{\frac{a}{s_\rs}}  \eta^{-\frac1{s_\rs}}\norm{ \pi^{(\rs)}(\bv_{\eta})}_{\Acal^{s_\rs}}^{\frac1{s_\rs}} \\
	  &\lesssim r_\eta ( 1 + \abs{\log \eta})^{\frac{2a}{s_\rs}} \eta^{-\frac1{s_\rs}}\norm{ \pi^{(\rs)}(\bv_{\eta})}_{\Acal^{s_\rs}}^{\frac1{s_\rs}} \\
	  &\lesssim \eta^{-\frac 1{\bar s}}\|\bv\|_{\Sb}^{\frac 1{\bar s}} ( 1 + \abs{\log \eta})^{\frac{2a}{s_\rs}} \eta^{-\frac1{s_\rs}}\norm{ \pi^{(\rs)}(\bv_{\eta})}_{\Acal^{s_\rs}}^{\frac1{s_\rs}} . \qedhere
\end{align*}
\end{proof}

\end{appendix}

\end{document}